\documentclass[12pt, pdflatex, a4paper]{article}
\usepackage[utf8]{inputenc}
\usepackage{amssymb, graphicx, amsmath, amsthm, mathtools, tikz-cd, enumitem, nccmath, sectsty, mathdots, array, tabularray}
\usepackage[mathscr]{eucal} 
\usepackage{xcolor} 
\usepackage{comment} 
\usepackage[affil-it]{authblk} 
\usepackage{mathrsfs}
\usepackage{makeidx}
\usepackage{setspace}
\usepackage{titlesec}

\usepackage{titletoc}
\usepackage{lipsum} 

\let\OLDthebibliography\thebibliography
\renewcommand\thebibliography[1]{
  \OLDthebibliography{#1}
  \setlength{\parskip}{0pt}
  \setlength{\itemsep}{1pt plus 0.3ex}
}

\usepackage{geometry}
\usepackage{hyperref}
\usepackage{wrapfig}

\usepackage[font=footnotesize]{caption} 

\usepackage{tocloft}
\usepackage{fancyhdr}

\theoremstyle{plain}
	\newtheorem{thm}{Theorem}[section]
	\newtheorem{prop}[thm]{Proposition}
	\newtheorem{corol}[thm]{Corollary}
	\newtheorem{lemma}[thm]{Lemma}
	\newtheorem{prob}[thm]{Problem}
	\newtheorem{probb}{Problem}

\theoremstyle{definition}
	\newtheorem{defn}[thm]{Definition}
	\newtheorem{rem}[thm]{Remark}

\providecommand{\keywords}[1]
{
  \small	
  \textbf{\textit{Keywords---}} #1
}

\counterwithout{figure}{section}

\usepackage{chngcntr}

\counterwithin*{equation}{section}

\sectionfont{\fontsize{14}{15}\selectfont}
\subsectionfont{\fontsize{13}{15}\selectfont}

\def\ad{\text{\normalfont{ad}}}
\def\Ad{\text{\normalfont{Ad}}}
\def\antidiag{\text{\normalfont{antidiag}}}

\def\C{{\mathbb C}}
\def\Geg{\widetilde{C}}

\def\diag{\text{\normalfont{diag}}}
\def\Diff{\text{\normalfont{Diff}}}
\def\Weyl{\mathcal{D}}
\def\D{\mathbb{D}}
\def\Dist{\mathcal{D}^\prime}

\def\End{\text{\normalfont{End}}}

\def\g{\mathfrak{g}}

\def\Harm{\mathcal{H}}
\def\Hom{\text{\normalfont{Hom}}}

\def\id{\text{\normalfont{id}}}
\def\Ind{\text{\normalfont{Ind}}}
\def\im{\text{\normalfont{im}}}

\def\Lie{\text{\normalfont{Lie}}}
\def\l{\mathfrak{l}}
\def\L{\mathcal{L}}

\def\N{{\mathbb N}}
\def\n{\mathfrak{n}}

\def\P{\mathcal{P}}
\def\Pol{\text{\normalfont{Pol}}}

\def\R{{\mathbb R}}
\def\Re{\text{\normalfont{Re}}}
\def\Rest{\text{\normalfont{Rest}}}

\def\sgn{\text{\normalfont{sgn}}}
\def\so{\mathfrak{so}}
\def\Sol{\text{\normalfont{Sol}}}
\def\spanned{\text{\normalfont{span}}}
\def\stab{\text{\normalfont{stab}}}

\def\Symb{\text{\normalfont{Symb}}}

\def\T{{\mathbb T}}

\def\V{\mathcal{V}}

\def\Z{{\mathbb Z}}

\newcommand\arrowsimeq{\stackrel{\mathclap{\thicksim}}{\longrightarrow}}

\newcommand\arrowiota{\stackrel{\mathclap{\iota}}{\hookrightarrow}}

\AtBeginEnvironment{pmatrix}{\setlength{\arraycolsep}{4pt}}
\renewcommand*{\arraystretch}{0.8} 

\setlist[enumerate]{topsep=1pt, itemsep=-2pt} 
\setlist[itemize]{topsep=1pt, itemsep=0pt}

\title{On sporadic symmetry breaking operators for principal series representations of the de Sitter and Lorentz groups}
\author{V\'{i}ctor P\'{E}REZ-VALD\'{E}S}
\date{}

\begin{document}
\maketitle
\vspace*{-1cm} 
\begin{abstract}
In this paper, we construct and classify all differential symmetry breaking operators between certain principal series representations of the pair $SO_0(4,1) \supset SO_0(3,1)$. In this case, we also prove a localness theorem, namely, all symmetry breaking operators between the principal series representations in concern are necessarily differential operators. In addition, we show that all these symmetry breaking operators are sporadic in the sense of T. Kobayashi, that is, they cannot be obtained by residue formulas of meromorphic families of symmetry breaking operators.
\end{abstract}
\begin{flushleft}
\end{flushleft}

\renewcommand{\contentsname}{Table of contents}
\tableofcontents


\section{Introduction}

A representation $\Pi$ of a group $G$ defines naturally a representation of a subgroup $G^\prime \subset G$ via restriction. In general, irreducibility is not preserved under restriction, and thus, the question of understanding the nature of the resulting representation $\Pi\big\rvert_{G^\prime}$ becomes a fundamental problem. For instance, if $G$ is compact and $\Pi$ is finite-dimensional, then $\Pi\big\rvert_{G^\prime}$ decomposes as a sum of irreducible finite-dimensional representations of $G^\prime$, and the structure of this decomposition can be studied by using classical techniques of finite-dimensional representation theory. 

In the general case when $G$ is non-compact and $\Pi$ is infinite-dimensional, the nature of the restriction $\Pi\big\rvert_{G^\prime}$ is unknown in most cases, and understanding its behaviour—such as its decomposition into irreducible factors—is a notoriously difficult and long-standing problem in representation theory commonly referred to as a \emph{branching problem}. 

For example, if $\Pi$ is a unitary representation of $G$, we have a notion for the \lq\lq irreducible decomposition\rq\rq\ of the restriction by using the concept of the direct integral (Mautner--Teleman theorem \cite{Mau50, Tel76}). However, in the general case when such a decomposition does not exist, alternative approaches are required to analyze the behaviour of the restriction $\Pi\big\rvert_{G^\prime}$.


Approximately ten years ago, T. Kobayashi proposed a program divided in three stages, to address these branching problems for real reductive Lie groups in a general framework (\emph{ABC program} \cite{Kob15}). In this work, he defined the concept of a \emph{symmetry breaking operator} (SBO) as follows: given a representation $\pi$ of $G^\prime$, a symmetry breaking operator is an element of the space $\Hom_{G^\prime}(\Pi\big\rvert_{G^\prime}, \pi)$; in other words, a $G^\prime$-equivariant linear map from the restriction $\Pi\big\rvert_{G^\prime}$ to the representation $\pi$.

Since in a general setting the restriction $\Pi\big\rvert_{G^\prime}$ does not decompose as a direct sum of irreducible representations of $G^\prime$, the study of the space of SBOs $\Hom_{G^\prime}(\Pi\big\rvert_{G^\prime}, \pi)$ can be thought of as a starting point to understand branching laws in a general situation. 

In the present paper, we focus on the problem of constructing concrete SBOs, which corresponds to Stage C in the ABC program of T. Kobayashi. An example of these operators are the classical Rankin--Cohen bidifferential operators, which are symmetry breaking operators for the tensor product of two holomorphic discrete series representations of $SL(2,\R)$ (see \cite{Coh75, KP16b, Ran56}).

Another example is given by the conformally covariant differential operators that A. Juhl constructed in the frame of conformal geometry (\cite{Juh09}), which correspond to symmetry breaking operators for spherical principal series representations of the connected Lorentz group $SO_0(1, n+1)$.\medskip

The problem of constructing and classifying all SBOs given a concrete pair $(G, G^\prime)$ is a challenging one. One of the first results was obtained for the pair $(G, G^\prime) = (O(n+1,1), O(n,1))$ in the case of principal series representations induced from exterior power representations. Concretely, all SBOs that can be written as differential operators—referred to as \emph{differential symmetry breaking operators} (DSBOs)—were constructed and classified in \cite{KKP16}, while the non-local ones were addressed in \cite{KS15, KS18}. Another early result along these lines was given for the pair $(G, G^\prime) = (SO_0(1, n) \times SO_0(1,n), \diag(SO_0(1,n)))$ for spherical principal series representations (\cite{Cle16, Cle17}).


The family of DSBOs constitutes a proper and significant subfamily within the set of all SBOs. In particular, there exists a duality between the space of DSBOs and the space of certain homomorphisms between generalized Verma modules (see \cite[Thm. 2.9]{KP16a}), highlighting the fact that the construction and classification of DSBOs is also algebraically meaningful. Moreover, in some special cases, the sets of SBOs and DSBOs coincide—a phenomenon known as \emph{localness theorem} (cf. \cite[Thm. 5.3]{KP16a}, \cite[Thm. 3.6]{KS18})—thus reducing the problem of constructing SBOs to that of constructing only the differential ones. For these reasons, among others, substantial progress has been made in recent years toward the construction and classification of DSBOs.

A major turning point in this line of research came in 2013, when T. Kobayashi introduced a general method—known as the \emph{F-method}—to construct and classify DSBOs for principal series representations (see \cite{Kob13, Kob14}). This approach has proven to be remarkably powerful and has been successfully applied in a variety of settings in recent years  (see, for instance \cite{KKP16, KKP18, KOSS15, KP16a, KP16b, Kub24, KO25, Nak22, Per-Val23, Per-Val24}).


Another fruitful technique to construct DSBOs is the \textit{source operator method}, developed by J.-L. Clerc and his collaborators (see \cite{BC12, BSCK20, BSCK25}). Recently, this method has been applied to the pair $(GL(n+1, \R), GL(n,\R))$, yielding new examples of DSBOs (see \cite{DL25}).


Besides the SBOs that can be written as differential operators, and after the work of T. Kobayashi and B. Speh (\cite{KS15, KS18}), a lot of progress has been made in the construction of SBOs for several pairs, such as $(GL(n+1, \R), GL(n, \R))$, $(O(p+1, q+1), O(p,q+1))$ or $(U(1,n+1; \mathbb{F}), U(1,m+1; \mathbb{F}) \times U(n-m; \mathbb{F}))$, among others (see, for instance \cite{DF24, Fra23, FO19, FW19, FW20, KL18}).\medskip

Following this line of study of symmetry breaking operators, we turn our attention to the connected pair $(G,G^\prime) = (SO_0(n+1,1), SO_0(n,1))$ and aim to construct and classify all SBOs. For certain principal series representations, this problem shares strong connections with the corresponding problem for the non-connected pair $(O(n+1, 1), O(n,1))$, and some of the techniques developed in \cite{KS15, KS18} can be adapted to the connected case. However, the problem for general principal series representations remains open, even for $n=3$ (see, for instance Table \ref{table-cases}).

As a first step towards solving this problem for the connected pair, we focus in this paper on the case $n = 3$, and aim to construct and classify all SBOs for arbitrary principal series representations. More precisely, we address the following problem:

\begin{prob}\label{prob_SBO} For any $\lambda, \nu \in \C$, $N\in\N$ and $m \in \Z$, characterize the space
\begin{equation}\label{SBO-space}
\Hom_{SO_0(3,1)}\left(C^\infty(S^3, \V_\lambda^{2N+1}), C^\infty(S^2, \L_{m,\nu})\right)
\end{equation}
of symmetry breaking operators $\mathbb{T}_{\lambda, \nu}^{N, m}: C^\infty(S^3, \V_\lambda^{2N+1}) \rightarrow C^\infty(S^2, \L_{m,\nu})$.
\end{prob}
In particular, as a subproblem we have:

\begin{prob}\label{prob_DSBO} For any $\lambda, \nu \in \C$, $N\in\N$ and $m \in \Z$, characterize the subspace of \eqref{SBO-space} given by the differential symmetry breaking operators
\begin{equation}\label{DSBO-space}
\Diff_{SO_0(3,1)}\left(C^\infty(S^3, \V_\lambda^{2N+1}), C^\infty(S^2, \L_{m,\nu})\right).
\end{equation}
\end{prob}
Here, $C^\infty(S^3, \V_\lambda^{2N+1})$ and $C^\infty(S^2, \L_{m,\nu})$ denote principal series representations induced from minimal parabolic subgroups of $SO_0(4,1)$ and $SO_0(3,1)$,  respectively. They are realized as spaces of smooth sections over the flag varieties $S^3$ and $S^2$. The former is parametrized by a natural number $N \in \N$, corresponding to the induced representation of the compact factor of the Levi subgroup of dimension $2N+1$, and by a complex parameter $\lambda \in \C$. The latter is parametrized by an integer number $m \in \Z$, corresponding to the characters of the compact factor of the Levi subgroup, and by a complex number $\nu \in \C$ (these are precisely defined in Section \ref{section-setting}).

In this paper, we provide a complete solution to Problem~\ref{prob_SBO} in the case where $|m| > N$. Our approach begins with solving Problem~\ref{prob_DSBO}, followed by establishing a localness theorem (Theorem~\ref{Localness-thm}) which guarantees that every SBO in \eqref{SBO-space} is actually given by a differential operator in \eqref{DSBO-space}, reducing Problem~\ref{prob_SBO} to Problem~\ref{prob_DSBO}.

By this argument, it suffices to solve Problem~\ref{prob_DSBO}, which can be divided into the following two problems:

\begin{probb} Give necessary and sufficient conditions on the tuple of parameters $(\lambda, \nu, N, m) \in \C^2 \times \N \times \Z$
such that the space \eqref{DSBO-space} of differential symmetry breaking operators
\begin{equation*}
\D_{\lambda, \nu}^{N, m}: C^\infty(S^3, \V_\lambda^{2N+1})\rightarrow C^\infty(S^2, \L_{m,\nu})
\end{equation*}
is non-zero. Moreover, determine
\begin{equation*}
\dim_\C \Diff_{SO_0(3,1)}\left(C^\infty(S^3, \V_\lambda^{2N+1}), C^\infty(S^2, \L_{m,\nu})\right).
\end{equation*}
\end{probb}

\begin{probb} Construct explicitly the generators
 \begin{equation*}
\D_{\lambda, \nu}^{N, m} \in \Diff_{SO_0(3,1)}\left(C^\infty(S^3, \V_\lambda^{2N+1}), C^\infty(S^2, \L_{m,\nu})\right).
\end{equation*}
\end{probb}

We give a complete solution to these two problems for $|m| > N$ in Theorems \ref{thm_main_class} and \ref{thm_main_const},  respectively.

Although this paper primarily focuses on the case $|m| > N$, Problems A and B have been addressed for some other cases in the past. For example, when $m = N = 0$, a solution was obtained from a conformal geometric perspective in \cite{Juh09, KOSS15}. For $N=1$ and $m=0$, a solution can be deduced from the results of \cite{KKP16}, while the case $|m| \geq N = 1$ was completely solved in \cite{Per-Val23}. Furthermore, for general $N \in \N$, the case $|m| = N$ is treated in \cite{Per-Val24}. A summary of the known and open cases is presented in Table~\ref{table-cases}.

Although Problems A and B above remain unsolved for $0 \leq |m| < N$, the case $|m| > N$ treated in this work corresponds to the hardest one, as the SBOs appearing are all \emph{sporadic} in the sense of T. Kobayashi (see Definition \ref{def-sporadic}). That is, they cannot be obtained as residues of a meromorphic family of SBOs, contrarily to the \emph{regular} SBOs (cf. \cite{KS15,KS18}). We prove this fact in Section \ref{section-sporadicity}.


\begin{table}
\caption{Known and open cases for Problems A and B}
\label{table-cases}
\centering
\setlength{\tabcolsep}{15pt} 
\renewcommand{\arraystretch}{1.8} 
\fontsize{11pt}{11pt}\selectfont 
\begin{tabular}{| c | c | c | c |}
\cline{2-4}
\multicolumn{1}{c|}{}& $0 \leq |m| < N$ & $|m| = N$ & $|m| > N$\\[1ex]
\hline
$N = 0$ & \multicolumn{2}{c|}{\cite{Juh09, KOSS15}} & Present work\\
\hline
$N = 1$ & \cite{KKP16} & \multicolumn{2}{c|}{\cite{Per-Val23}}\\
\hline
$N \geq 2$ & Undone & \cite{Per-Val24} & Present work\\
\hline
\end{tabular}
\end{table}

Throughout the text, we use $\N$ to denote the set of non-negative ($\geq 0$) integers, and $\N_+$ to denote the set of positive ($\geq 1$) integers. 

\subsection{Organization of the paper}
This paper is structured into eight sections, including the Introduction and an Appendix. In Section~\ref{section-mainthms} we present our main results, which provide a complete solution to Problems~\ref{prob_SBO} and~\ref{prob_DSBO} in the case $|m| > N$. In particular, Theorems \ref{thm_main_class} and~\ref{thm_main_const} address Problem~\ref{prob_DSBO}, and Theorem \ref{Localness-thm} solves Problem~\ref{prob_SBO} by establishing a localness theorem that reduces it to Problem~\ref{prob_DSBO}.

Section~\ref{section-setting} provides a concrete geometric description of the setting. In Section~\ref{section-applyingFmethod}, we review and apply the F-method, which serves as the primary tool for solving Problem~\ref{prob_DSBO}. This method allows us to reduce Problem~\ref{prob_DSBO} to solving a specific system of ordinary differential equations, as formulated in Theorems \ref{Thm-findingequations} and \ref{Thm-solvingequations}.

We show in Section~\ref{section-case_m_lessthan_-N} that it suffices to consider the case $m > N$, since the case $m < -N$ follows from the former one (see Proposition~\ref{prop-duality}). 
The proofs of Theorems \ref{thm_main_class} and \ref{thm_main_const} are given in Section~\ref{section-proof-mainthms} by using Theorems \ref{Thm-findingequations} and \ref{Thm-solvingequations}, while the proof of Theorem~\ref{Thm-solvingequations}—the most technically demanding part of this work—is presented in Section~\ref{section-proof_of_solving_equations}. (We note that a proof of Theorem \ref{Thm-findingequations} has already been given in \cite{Per-Val24}.)

Section~\ref{section-localnessthm} is devoted to the proof of the localness theorem (Theorem~\ref{Localness-thm}) and in Section~\ref{section-sporadicity} we analyze the sporadic nature of the SBOs for $|m| > N$.

The Appendix collects auxiliary results on Gegenbauer polynomials and hypergeometric functions used mainly in Section~\ref{section-proof_of_solving_equations}.

\subsection{Main results}\label{section-mainthms}
In this subsection we state our main results. We start with the result that solves Problem A for $|m| > N$.

\begin{thm}\label{thm_main_class} Let $\lambda, \nu \in \C$, $N \in \N$ and $m \in \Z$, and suppose that $|m| > N$. Then, the following three conditions on the quadruple $(\lambda, \nu, N, m)$ are equivalent:
\begin{enumerate}[label=\normalfont{(\roman*)}]
\item $\Diff_{SO_0(3,1)}\left(C^\infty(S^3, \V_\lambda^{2N+1}), C^\infty(S^2, \L_{m,\nu})\right) \neq \{0\}$.
\item $\dim_\C \Diff_{SO_0(3,1)}\left(C^\infty(S^3, \V_\lambda^{2N+1}), C^\infty(S^2, \L_{m,\nu})\right) = 1$.
\item $\lambda \in \Z_{\leq 1-|m|}$ and $\nu \in [1-N, N+1]\cap \Z$.
\end{enumerate}
\end{thm}

In Figure \ref{parameters} we can visualize how the parameters satisfying condition (iii) in Theorem \ref{thm_main_class} are distributed in the plane. 
Concretely, for a given $(\nu, m)$ depicted as a white point ($\bigcirc$), the pairs $(\lambda, m)$ satisfying (iii) are shown as squares ($\blacksquare$). We note from the condition (iii) that $(\nu, m) \in \{(a,b)\in \Z^2: 1-N\leq a \leq N+1, |b| > N\}$ (striped area in blue), while $(\lambda, m) \in \{(c,d)\in \Z^2 : c \leq 1-|d|, |d| > N\}$ (dotted area in red).

\begin{figure}
\centering
\includegraphics[scale=1.1]{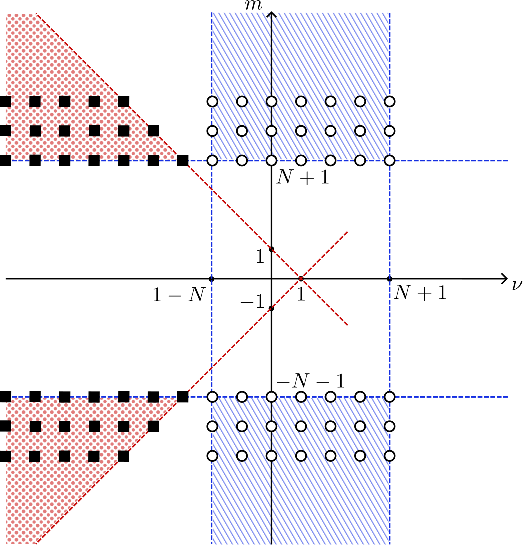}
\caption{Distribution of the parameters $(\lambda, \nu, N, m)$ satisfying (iii) of Theorem \ref{thm_main_class}}\label{parameters}
\end{figure}
\medskip
As a consequence of Theorem \ref{thm_main_class} above, we have the following corollary.

\begin{corol}\label{corol_main} Let $\lambda, \nu \in \C$, $N \in \N$ and $m \in \Z$ with $|m| > N$, and let $\D$ be any non-zero differential symmetry breaking operator in \eqref{DSBO-space}. Then, the following holds.
\begin{enumerate}[label=\normalfont{(\arabic*)}]
\item The principal series representation $C^\infty(S^3, \V^{2N+1}_\lambda)$ of $SO_0(4,1)$ is reducible.
\item The principal series representation $C^\infty(S^2, \L_{m, \nu})$ of $SO_0(3,1)$ is irreducible. In particular, $\im(\D)$ is dense in $C^\infty(S^2, \L_{m, \nu})$.
\end{enumerate}
\end{corol}

It follows from Theorem \ref{thm_main_class} that the space (\ref{DSBO-space}) has dimension at most one. In Theorem \ref{thm_main_const} below, we give an explicit formula for the generator $\D_{\lambda, \nu}^{N, m}$ of \eqref{DSBO-space} in terms of the coordinates $(x_1, x_2, x_3) \in \R^3$ via the conformal compactification $\R^3 \arrowiota S^3$, as shown in the diagram below:
\begin{equation*}
\begin{tikzcd}
C^\infty(S^3, \V_\lambda^{2N+1}) \arrow[d, hookrightarrow, "{\displaystyle{\iota^*}}"'] \arrow[r, dashed] & C^\infty(S^2, \L_{m, \nu}) \arrow[d, hookrightarrow, "{\displaystyle{\iota^*}}"']\\
C^\infty(\R^3, V^{2N+1}) \arrow[r, "\D_{\lambda, \nu}^{N, m}"] & C^\infty(\R^2)
\end{tikzcd}
\end{equation*}
Here,  $\R^2 \subset \R^3$ is realized as $\R^2 = \{(x_1, x_2, 0) : x_1, x_2 \in \R\}$.

To give an explicit formula of $\D_{\lambda, \nu}^{N, m}$, let $\{u_d : d = 0, 1, \ldots, 2N\}$ be the standard basis of $V^{2N+1} \simeq \C^{2N+1}$ (see \eqref{basis-V2N+1}), and denote by $\{u_d^\vee : d = 0, 1, \ldots, 2N\}$ the dual basis of $(V^{2N+1})^\vee$. We write $z = x_1 +ix_2$ so that the Laplacian $\Delta_{\R^2} = \frac{\partial^2}{\partial x_1^2} + \frac{\partial^2}{\partial x_2^2}$ on $\R^2$ is given as $\Delta_{\R^2} = 4\frac{\partial^2}{\partial z \partial \overline{z}}$.

Suppose $\nu-\lambda \in \N$, and let $\widetilde{\C}_{\lambda, \nu}$ denote the following (scalar-valued) differential operator, sometimes called the \emph{Juhl operator} (cf. \cite[2.22]{KKP16}):

\begin{equation}\label{def-operator-Ctilda}
\begin{aligned}
\widetilde{\C}_{\lambda, \nu} & = \Rest_{x_3 = 0} \circ \left(I_{\nu-\lambda} \widetilde{C}^{\lambda-1}_{\nu-\lambda}\right)\left(-4\frac{\partial^2}{\partial z \partial \overline{z}}, \frac{\partial}{\partial x_3}\right),
\end{aligned}
\end{equation}
where $(I_\ell\Geg_\ell^\mu)(x,y) := x^{\frac{\ell}{2}}\Geg_\ell^\mu\left(\frac{y}{\sqrt{x}}\right)$ is the two-variable inflation of the renormalized Gegenbauer polynomial $\Geg_\ell^\mu(z)$ (see (\ref{Gegenbauer-polynomial(renormalized)})). 

We also define the following constants for $d = 0, 1, \ldots, 2N$ and $r = 0, \ldots, \min(d, 2N-d)$:
\begin{equation}\label{const-AB}
\begin{aligned}
\Gamma(d, r) &:= \begin{pmatrix}
d\\
r
\end{pmatrix}
\frac{\Gamma\left(\lambda+\left[\frac{\nu-\lambda-|m|+N-d-1}{2}\right]\right)}{\Gamma\left(\lambda+\left[\frac{\nu-\lambda-|m|-1}{2}\right]\right)}
\frac{\Gamma(2N-r+1)}{\Gamma(N+1)}\frac{\Gamma(N+|m|+1)}{\Gamma(N+|m|+1-r)},\\[6pt]
A(d, r) &:= (-1)^{\nu-1}\Gamma(d,r)\frac{\Gamma\left(\lambda+\left[\frac{-\nu-\lambda-|m|+N-d+1}{2}\right]\right)}{\Gamma\left(\lambda+\left[\frac{\nu-\lambda-|m|+N-d-1}{2}\right]\right)}\frac{\Gamma(N+\nu)}{\Gamma(N+\nu-r)}, \\[6pt]
B(d, r) &:= (-1)^{d-N}\Gamma(2N-d, r)\frac{\Gamma(N-\nu+2)}{\Gamma(N-\nu+2-r)}.
\end{aligned}
\end{equation}

Now, we have

\begin{thm}\label{thm_main_const} Let $N \in \N$ and $m \in \Z$ with $|m| > N$, and let $\lambda \in \Z_{\leq 1-|m|}$ and $\nu \in [1-N, N+1]\cap\Z$. Then, any differential symmetry breaking operator in {\normalfont{(\ref{DSBO-space})}} is proportional to the differential operator $\D_{\lambda, \nu}^{N, m}$ of order $\nu-\lambda$ given as follows:
\begin{itemize}[leftmargin=0.5cm]
\item[\normalfont{$\bullet$}] $m > N:$
\begin{multline}\label{Operator+}
\D_{\lambda, \nu}^{N,m} = \sum_{d=0}^{N-1}\sum_{r=0}^d 2^{d+2\nu-2r-2}A(d, r)\widetilde{\C}_{\lambda+N-r, 2-\nu-d-m+r}\frac{\partial^{N+m+d+2\nu-2r-2}}{\partial z^{N+\nu-r-1} \partial \overline{z}^{m+d+\nu-r-1}}\otimes u^\vee_{d}\\
 + \sum_{d=N}^{2N}\sum_{r=0}^{2N-d}2^{2N-d-2r}B(d,r)\widetilde{\C}_{\lambda+N-r, \nu+d-m-2N+r}\frac{\partial^{3N+m-d-2r}}{\partial z^{2N-d-r}\partial \overline{z}^{N+m-r}} \otimes u^\vee_d.
\end{multline}
\item[\normalfont{$\bullet$}] $m < -N:$


\begin{multline}\label{Operator-}
\D_{\lambda, \nu}^{N,m} = \sum_{d=0}^{N}\sum_{r=0}^{d}2^{d-2r}(-1)^dB(2N-d,r)\widetilde{\C}_{\lambda+N-r, \nu-d+m+r}\frac{\partial^{N-m+d-2r}}{\partial z^{N-m-r}\partial \overline{z}^{d-r}} \otimes u^\vee_{d}\\
+ \sum_{d=N+1}^{2N}\sum_{r=0}^{2N-d} 2^{2N-d+2\nu-2r-2}(-1)^d A(2N-d, r)\widetilde{\C}_{\lambda+N-r, 2-\nu+d-2N+m+r}\\
\times\frac{\partial^{3N-m-d+2\nu-2r-2}}{\partial z^{-m+2N-d+\nu-r-1} \partial \overline{z}^{N+\nu-r-1}}\otimes u^\vee_{d}.
\end{multline}
\end{itemize}
\end{thm}
Both Theorems \ref{thm_main_class} and \ref{thm_main_const} will be proved in Section \ref{section-proof-mainthms}, while Corollary \ref{corol_main} will be proved in Section \ref{section-proof-corol}

\begin{thm}[\textbf{Localness Theorem}]\label{Localness-thm} Let $\lambda, \nu \in \C$, $N \in \N$ and $m \in \Z$, and suppose that $|m| > N$. Then, any symmetry breaking operator in \eqref{SBO-space} is given by a differential operator. In other words,
\begin{equation*}
\Hom_{SO_0(3,1)}\left(C^\infty(S^3, \V_\lambda^{2N+1}), C^\infty(S^2, \L_{m,\nu})\right) = \Diff_{SO_0(3,1)}\left(C^\infty(S^3, \V_\lambda^{2N+1}), C^\infty(S^2, \L_{m,\nu})\right).
\end{equation*}
\end{thm}

Theorem \ref{Localness-thm} above will be proved in Section \ref{section-localnessthm}.

\section{Review of Principal Series Representations of the de Sitter and Lorentz Groups}\label{section-setting}

In this section we do a quick review on the geometric realization of the connected de Sitter and Lorentz groups $G := SO_0(4,1)$ and $G^\prime := SO_0(3,1)$, and their principal series representations $\Ind_P^G(\sigma_\lambda^{2N+1})$ and $\Ind_{P^\prime}^{G^\prime}(\tau_{m, \nu})$ with respect to their minimal parabolic subgroups $P$ and $P^\prime$.

Recall that the non-connected de Sitter group $O(4,1)$ can be realized as the subgroup of $GL(5, \R)$ that leaves invariant the standard quadratic form of signature $(4,1)$ on $\R^5$:
\begin{align*}
O(4,1) = \{g \in GL(5,\R) : Q_{4,1}(gx) = Q_{4,1}(x) \enspace \forall x \in \R^5\},\\
\text{ with }Q_{4,1}(x) := x_0^2 + x_1^2 + x_2^2 + x_3^2 - x_4^2, \quad x = {}^t(x_0, \dots, x_4) \in \R.
\end{align*}
Let $\g(\R) := \so(4,1)$ be its Lie algebra and define:
\begin{equation}\label{elements}
\begin{aligned}
H_0 & := E_{1,5}+E_{5,1}\\
N_j^+ & := -E_{1,j+1} + E_{j+1,1} - E_{5,j+1} - E_{j+1,5} & (1 \leq j  \leq 3),\\
N_j^- & := -E_{1,j+1} + E_{j+1,1} + E_{5,j+1} + E_{j+1,5} & (1 \leq j  \leq 3),\\
X_{p,q} & := E_{q+1, p+1} - E_{p+1, q+1} & (1 \leq p \leq q \leq 3).
\end{aligned}
\end{equation}
Observe that the sets $\{N_j^+: 1 \leq j  \leq 3\}$, $\{N_j^-: 1 \leq j  \leq 3\}$ and $\{X_{p,q} : 1 \leq p,q \leq 3\}\cup\{H_0\}$ form bases of the following Lie subalgebras of $\g(\R)$:
\begin{equation*}
\begin{aligned}
\n_+(\R) & := \ker(\ad(H_0) - \id)\subset \g(\R),\\
\n_-(\R) & := \ker(\ad(H_0) + \id)\subset \g(\R),\\
\l(\R) & := \ker\ad(H_0)\subset \g(\R).
\end{aligned}
\end{equation*}
Note that these give the Gelfand--Naimark decomposition of $\g(\R)$: $\g(\R) = \n_-(\R) + \l(\R) + \n_+(\R)$.

The group $O(4,1)$ (as well as its identity component $G = SO_0(4,1)$) acts transitively and conformally on $\Xi/\R^\times$, where $\Xi$ is the isotropic cone defined by $Q_{4,1}$:
\begin{equation*}
\Xi := \{x \in \R^5\setminus\{0\} : Q_{4,1}(x) = 0\} \subset \R^5.
\end{equation*}
Therefore, if we define $P := \stab_G(1:0:0:0:1)$, we obtain isomorphisms
\begin{equation*}
G/P \arrowsimeq \Xi/\R^\times \arrowsimeq S^3,
\end{equation*}
where the second map comes from the projection $\Xi \rightarrow S^3, \enspace x \mapsto \frac{1}{x_4}{}^t(x_0, x_1, x_2, x_3)$. The group $P \subset G$ is a parabolic subgroup of $G$ with Langlands decomposition as follows:
\begin{align*}
&P = LN_+ = MAN_+,\\
&M = \left\{
\begin{pmatrix}
1 & &\\
& g &\\
& & 1
\end{pmatrix} : g \in SO(3)
\right\} \simeq SO(3),\\
&A = \exp(\R H_0) \simeq SO_0(1,1),\\
&N_+ = \exp(\n_+(\R)) \simeq \R^3.
\end{align*}
Moreover, if we define a diffeomorphism $n_-: \R^3 \arrowsimeq N_- := \exp(\n_-(\R))$ by
\begin{equation}
\label{open-Bruhat-cell}
\begin{gathered}
n_-(x) := \exp\left(\sum_{j = 1}^3 x_jN_j^-\right) =
I_5 +
\begin{pmatrix}
- \frac{1}{2}\|x\|^2 & -{}^tx& -\frac{1}{2}\|x\|^2)\\
x & 0 & x\\
\frac{1}{2}\|x\|^2 & {}^tx & \frac{1}{2}\|x\|^2
\end{pmatrix},
\end{gathered}
\end{equation}
we obtain the coordinates on the open Bruhat cell
\begin{equation*}
\n_-(\R) \hookrightarrow G/P \simeq \Xi/\R^\times \simeq S^3.
\end{equation*}

Now, for any $\lambda \in \C$, let $\C_\lambda$ be the following one-dimensional representation of $A$:
\begin{equation*}
\C_\lambda: A \rightarrow \C^\times, \quad a = e^{tH_0} \mapsto a^\lambda := e^{\lambda t}.
\end{equation*}
Moreover, for any $N \in \N$, let $(\sigma^{2N+1}, V^{2N+1})$ be the $(2N+1)$-dimensional irreducible representation of $M \simeq SO(3)$ realized in $V^{2N+1} := \Pol^{2N}(\C^2)$, the space of homogeneous polynomials of degree $2N$ in two variables (for details about this realization, see \cite{Per-Val24}). Now, we consider the representation of $P = MAN_+$ defined as $\sigma_\lambda^{2N+1} := \sigma^{2N+1} \boxtimes \C_\lambda \boxtimes \mathbf{1}$ and define the homogeneous vector bundle over $G/P\simeq S^3$ associated to $(\sigma_\lambda^{2N+1}, V_\lambda^{2N+1}):$
\begin{equation}\label{vectorbundle-V}
\V_\lambda^{2N+1} := G \times_P V_\lambda^{2N+1}.
\end{equation}
Then, we form the principal series representation $\Ind_P^G(\sigma_\lambda^{2N+1})$ of $G$ on the space of smooth sections $C^\infty(G/P, \V_\lambda^{2N+1})$, or equivalently, on the space
\begin{equation*}
\begin{aligned}
C^\infty(G, V_\lambda^{2N+1})^P =
\{f \in C^\infty(G, V_\lambda^{2N+1}) : f(gman) &= \sigma^{2N+1}(m)^{-1} a^{-\lambda}f(g) \\
& \text{for all } m\in M, a\in A, n\in N_+, g\in G\}.
\end{aligned}
\end{equation*}

Its $N$-picture is defined on $C^\infty(\R^3, V^{2N+1})$ by the restriction to the open Bruhat cell via (\ref{open-Bruhat-cell}):
\begin{align*}
C^\infty(G/P, \V_\lambda^{2N+1}) \simeq C^\infty(G, V_\lambda^{2N+1})^P &\longrightarrow C^\infty(\R^3, V^{2N+1})\\
f &\longmapsto F(x) := f(n_-(x)).
\end{align*}

Now, we realize $G^\prime = SO_0(3,1)$ as $G^\prime \simeq \stab_G{}^t(0,0,0,1,0) \subset G$, and denote $\g^\prime(\R) := \Lie(G^\prime)$. The action of $G^\prime$ leaves invariant $\Xi\cap\{x_3 = 0\}$, thus it acts transitively on
\begin{equation*}
S^2 = \{(y_0, y_1, y_2, y_3) \in S^3 : y_3 = 0\} \simeq \left(\Xi\cap\{x_3 = 0\}\right)/\R^\times.
\end{equation*}
Then, $P^\prime := P \cap G^\prime$ is a parabolic subgroup of $G^\prime$ with Langlands decomposition as follows:
\begin{align*}
& P^\prime = L^\prime N_+^\prime = M^\prime A N_+^\prime\\
& M^\prime = M\cap G^\prime =\left\{
\begin{pmatrix}
1 & &\\
& g^\prime &\\
& & I_2
\end{pmatrix} : g^\prime \in SO(2) \right\} \simeq SO(2),\\
&N^\prime_+ =  N_+ \cap G^\prime = \exp(\n_+^\prime(\R)) \simeq \R^2.
\end{align*}
Here, the Lie algebras $\n_\pm^\prime(\R)$ stand for $\n_\pm^\prime(\R) := \n_\pm(\R)\cap \g^\prime(\R)$.

For any $m \in \Z$, let $\C_m$ denote the following one-dimensional representation of $M^\prime \simeq SO(2)$.
\begin{align*}
\C_m: SO(2) \longrightarrow \C^\times, 
\begin{pmatrix}
\cos s & -\sin s\\
\sin s & \cos s
\end{pmatrix} \mapsto e^{ims}
\end{align*}
Now, for any $\nu \in \C$ and any $m \in \Z$, we consider the representation of $P^\prime$ defined as $\tau_{m, \nu} := \C_m \boxtimes \C_\nu \boxtimes \mathbf{1}$, and define the homogeneous vector bundle over $G^\prime/P^\prime\simeq S^2$ associated to $(\tau_{m,\nu}, \C_{m,\nu}):$
\begin{equation}\label{vectorbundle-L}
\L_{m,\nu} := G^\prime \times_{P^\prime} \C_{m,\nu}.
\end{equation}
As before, we form the principal series representation $\Ind_{P^\prime}^{G^\prime}(\tau_{m,\nu})$ of $G^\prime$ on the space of smooth sections $C^\infty(G^\prime/P^\prime, \L_{m,\nu})$.


\section{Applying the F-method}\label{section-applyingFmethod}
The key tool we use to prove Theorems \ref{thm_main_class} and \ref{thm_main_const} is the F-method. A very powerful machinery developed by T. Kobayashi that reduces the problem of constructing and classifying differential symmetry breaking operators into the computational problem of solving a certain system of partial differential equations. We state the main result we use in our setting in Theorem \ref{F-method-thm}. This will allow us to derive Theorems~\ref{thm_main_class} and \ref{thm_main_const} from Theorem~\ref{thm-step2}.

For more details about the F-method, we would like to refer to \cite{Kob13, Kob14, KP16a, KP16b, KO25, Per-Val23}. We start by recalling the definition of the symbol map.

\begin{defn}[{\cite[Sec. 4.1]{KP16a}}] Let $E$ be a finite-dimensional complex vector space and denote by $E^\vee$ its dual space. Write $n:= \dim E$ and let $\Diff^\text{const}(E)$ denote the ring of differential operators over $E$ with constant coefficients. Then, the following natural isomorphism is called the \textbf{symbol map}
\begin{equation}\label{symbolmap}
\Symb: \Diff^\text{const}(E) \arrowsimeq \Pol(E^\vee),
\hspace{0.3cm} \sum_{\alpha \in \N^n} a_\alpha \frac{\partial^\alpha}{\partial z^\alpha} \longmapsto  \sum_{\alpha \in \N^n} a_\alpha \zeta^\alpha.
\end{equation}
\end{defn}

Now, denote by $\g$ the complexification of the real Lie algebra $\g(\R)$. Similar notation will be used for other Lie algebras. Put $L^\prime = M^\prime A$ and consider
\begin{equation*}
\Hom_{L^\prime}\left(V_\lambda^{2N+1}, \C_{m,\nu} \otimes \Pol(\n_+)\right),
\end{equation*}
where the $L^\prime$-action on $\Pol(\n_+)$ is given by
\begin{equation*}
\begin{aligned}
\Ad_\#(\ell): \Pol(\n_+) &\longrightarrow \Pol(\n_+)\\
\hspace{2.8cm} p(\cdot) &\longmapsto p\left(\Ad(\ell^{-1}) \cdot\right).
\end{aligned}
\end{equation*}

Now, for $\psi \in \left(\Pol(\n_+)\otimes(V_\lambda^{2N+1})^\vee \right) \otimes \C_{m, \nu} \simeq \Hom_\C\left(V_\lambda^{2N+1}, \C_{m,\nu} \otimes \Pol(\n_+)\right)$, we consider the system
\begin{equation}\label{F-system-2}
\left(\widehat{d \pi_\mu}(C) \otimes \id_{\C_{m,\nu}} \right)\psi = 0, \quad \forall \enspace C\in \n_+^\prime,
\end{equation}
and set
\begin{equation*}
\Sol(\n_+; \sigma_\lambda^{2N+1}, \tau_{m, \nu}) := \{\psi \in \Hom_{L^\prime}\left(V_\lambda^{2N+1}, \C_{m,\nu} \otimes \Pol(\n_+)\right) :\text{\normalfont{(\ref{F-system-2}) holds}}\}.
\end{equation*}
The map $\widehat{d\pi_\mu}$ above stands for a Lie algebra homomorphism
\begin{equation*}
\widehat{d \pi_\mu}: \g \rightarrow \Weyl(\n_+)\otimes \End((V^{2N+1}_\lambda)^\vee),
\end{equation*} 
where $\Weyl(\n_+)$ denotes the Weyl algebra on $\n_+$ and $\mu = (\sigma_\lambda^{2N+1})^\vee \otimes \C_{2\rho}$. Thus, for any $C \in \n_+$, $\widehat{d\pi_\mu}(C)$ defines a differential operator with coefficients on $\End((V^{2N+1}_\lambda)^\vee)$. In other words, \eqref{F-system-2} defines a differential equation on $\Pol(\n_+) \otimes \Hom_\C(V^{2N+1}_\lambda, \C_{m, \nu})$. For details about the precise definition of $\widehat{d \pi_\mu}$, see, for instance \cite[Sec. 3]{KP16a} or \cite[Sec. 2]{KO25}.

Since $\n_+ \simeq \C^3$ is abelian, we have (see \cite[Sec. 4.3]{KP16a})
\begin{equation*}
\Diff_{G^\prime}\left(C^\infty(S^3, \V_\lambda^{2N+1}), C^\infty(S^2, \L_{m,\nu})\right) \subset \Diff^\text{const}(\n_-) \otimes \Hom_\C(V^{2N+1}_\lambda,\C_{m, \nu}).
\end{equation*}
Thus, by applying the symbol map, we obtain the following:
\begin{thm}[F-method, {\cite[Thm. 4.1]{KP16a}}]\label{F-method-thm} The map below is a linear isomorphism
\begin{equation}\label{isomorphism-Fmethod-concrete}
\begin{tikzcd}[column sep = 1cm]
\Diff_{G^\prime}\left(C^\infty(S^3, \V_\lambda^{2N+1}), C^\infty(S^2, \L_{m,\nu})\right) \arrow[r, "\displaystyle{_{\Symb \otimes \id}}", "\thicksim"'] & \Sol(\n_+; \sigma_\lambda^{2N+1}, \tau_{m, \nu}).
\end{tikzcd}
\end{equation}
\end{thm}
By Theorem \ref{F-method-thm} above, in order to determine the space \eqref{DSBO-space}, it suffices to describe 
\begin{equation}\label{sol-space}
\Sol(\n_+; \sigma_\lambda^{2N+1}, \tau_{m,\nu}).
\end{equation}
To do so, by \eqref{F-system-2} one may proceed with the following two steps:
\begin{itemize}[topsep=10pt]
\item Step 1. Determine the generators $\psi$ of $\Hom_{L^\prime}\left(V_\lambda^{2N+1}, \C_{m,\nu} \otimes \Pol(\n_+)\right)$.
\item Step 2. Solve $\left(\widehat{d \pi_\mu}(C) \otimes \id_{\C_{m,\nu}} \right)\psi = 0, \quad \forall \enspace C\in \n_+^\prime$.
\end{itemize}

We address these two steps in the following sections.

\subsection{Description of $\Sol(\n_+; \sigma_\lambda^{2N+1}, \tau_{m, \nu})$ for $m > N$}\label{section-description-of-Sol} 

In this section we recall some results proved in \cite{Per-Val24} that help us to tackle the two steps of the F-method in our setting. In particular, the first step can be achieved by using Proposition \ref{prop-FirststepFmethod}, and the second step can be reduced to solving a system of ODEs by using Theorem \ref{Thm-findingequations}. The complete description of \eqref{sol-space} is given in Theorem \ref{thm-step2}. We focus mainly on the case $m>N$, since the case $m <-N$ can be deduced from the former one (see Proposition \ref{prop-duality}).

Since $\n_+ \simeq \C^3$, we identify $\Pol(\n_+) \simeq \C[\zeta_1, \zeta_2, \zeta_3]$, where $(\zeta_1, \zeta_2, \zeta_3)$ are the coordinates of $\n_+$ with respect to the basis $\{N_j^+: j=1,2,3\}$. Via this realization, a direct computation shows that the action of $L^\prime = M^\prime A$ on $\Pol(\n_+)$ is given by:
\begin{align*}
M^\prime \times A \times \C[\zeta_1, \zeta_2, \zeta_3] &\longrightarrow \C[\zeta_1, \zeta_2, \zeta_3]\\
(g^\prime, e^{tH_0}, p(\cdot)) &\longmapsto p\left(e^{-t}\begin{pmatrix}
g^\prime &\\
& 1
\end{pmatrix}^{-1}
\cdot\right).
\end{align*}

Given $k \in \N$, denote by $\Harm^k(\C^2)$ the space of harmonic homogeneous polynomials of degree $k$ on two variables. In other words,
\begin{equation*}
\Harm^k(\C^2) := \{h\in \C[\zeta_1, \zeta_2] :\Delta h = 0, \deg(h) = k\}.
\end{equation*}
We recall that this space is given as follows (cf. \cite[Thm. 3.1]{He00}):
\begin{equation*}
\Harm^k(\C^2) = \begin{cases}
\C, & \text{ if } k = 0,\\
\C(\zeta_1 + i\zeta_2)^k \oplus \C(\zeta_1 - i\zeta_2)^k, & \text{ if } k > 0.
\end{cases}
\end{equation*}
Also, we consider the following basis of $V^{2N+1} = \Pol^{2N}(\C^2) = \{p \in \C[\xi_1, \xi_2]: \deg p = 2N\}$:
\begin{equation}\label{basis-V2N+1}
B(V^{2N+1}) = \{u_d : d=0, 1, \ldots, 2N\} = \{\xi_1^{2N-d}\xi_2^{d}: d = 0, 1, \ldots, 2N\}.
\end{equation}

Now, for $|m| \geq N$ and $k \in K_{N,m} := \{|m-\ell|, |m+\ell| : \ell = 0, 1, \ldots, N\}$, we define a linear $SO(2)$-intertwining homomorphism $h_{k}^\pm: V^{2N+1}\rightarrow\C_{m} \otimes \Harm^k(\C^2)$
as follows: 

\begin{equation}\label{generators-Hom(V2N+1,CmHk)}
\begin{array}{l l}
h_{m+\ell}^+: V^{2N+1} \longrightarrow \C_m \otimes \Harm^{m+\ell}(\C^2), &\enspace
u_{N+\ell}=\xi_1^{N-\ell}\xi_2^{N+\ell} \longmapsto 1 \otimes(\zeta_1 + i\zeta_2)^{m+\ell},\\[7pt]
h_{m-\ell}^+: V^{2N+1} \longrightarrow \C_m \otimes \Harm^{m-\ell}(\C^2), &\enspace
u_{N-\ell} = \xi_1^{N+\ell}\xi_2^{N-\ell} \longmapsto 1 \otimes (\zeta_1 + i\zeta_2)^{m-\ell},\\[7pt]
h_{-m-\ell}^-: V^{2N+1} \longrightarrow \C_m \otimes \Harm^{-m-\ell}(\C^2), &\enspace
u_{N+\ell} = \xi_1^{N-\ell}\xi_2^{N+\ell} \longmapsto 1 \otimes (\zeta_1 - i\zeta_2)^{-m-\ell},\\[7pt]
h_{-m+\ell}^-: V^{2N+1} \longrightarrow \C_m \otimes \Harm^{-m+\ell}(\C^2), &\enspace
u_{N-\ell}= \xi_1^{N+\ell}\xi_2^{N-\ell} \longmapsto 1 \otimes(\zeta_1 - i\zeta_2)^{-m+\ell}.
\end{array}
\end{equation}

Note that the sign index $\pm$ in $h_k^\pm$ is chosen to coincide with that of $m\pm\ell$, being $+$ if $m \geq N$ and $-$ if $m\leq -N$.

Given $b\in \Z$ and $g \in \C[t]$, we define a meromorphic function $(T_bg)(\zeta)$ of three variables $\zeta = (\zeta_1, \zeta_2,\zeta_3)$ as follows:
\begin{equation}\label{def-T_b}
(T_bg)(\zeta) := (\zeta_1^2+\zeta_2^2)^\frac{b}{2} g\left(\frac{\zeta_3}{\sqrt{\zeta_1^2+\zeta_2^2}}\right).
\end{equation}

We observe that $T_bg$ is a homogeneous polynomial of degree $b$ on three variables if $b \in \N$ and $g \in \Pol_b[t]_\text{{\normalfont{even}}}$, where
\begin{equation}\label{def-Pol_even}
\Pol_b[t]_\text{{\normalfont{even}}}:= \spanned_\C\left\{t^{b-2j}: j = 0, \dots, \left[\frac{b}{2}\right]\right\}.
\end{equation}
Moreover, we obtain a bijection
\begin{equation}\label{bijection-T_b}
T_b: \Pol_b[t]_\text{{\normalfont{even}}} \arrowsimeq \bigoplus_{2b_1+b_2 = b} \Pol^{b_1}[\zeta_1^2 + \zeta_2^2]\otimes\Pol^{b_2}[\zeta_3].
\end{equation}
The following result from \cite{Per-Val24} allows us to complete the first step of the F-method.

\begin{prop}[{\cite[Prop. 4.3]{Per-Val24}}]\label{prop-FirststepFmethod} Let $\lambda, \nu \in \C$, $N \in \N$, $m \in \Z$ with $|m| \geq N$ and $a \in \N$. Then, the following two conditions on the tuple $(\lambda, \nu, N, m, a)$ are equivalent.
\begin{enumerate}[label=\normalfont{(\roman*)}, topsep=3pt]
\item $\Hom_{L^\prime}\left(V^{2N+1}_\lambda, \C_{m, \nu} \otimes \Pol^a(\n_+)\right) \neq \{0\}$.
\item $a = \nu - \lambda$ and $a \geq \min K_{N,m}$.
\end{enumerate}
Moreover, if one (therefore all) of the conditions above is satisfied, we have
\begin{equation*}
\begin{gathered}
\Hom_{L^\prime}\left(V^{2N+1}_\lambda, \C_{m, \nu} \otimes \Pol^a(\n_+)\right)= 
\spanned_\C\{\left(T_{a-k}g_k\right)h_k^{\sgn(m)}: g_k \in \Pol_{a-k}[t]_\text{{\normalfont{even}}}, k\in K_{N,m}\},
\end{gathered}
\end{equation*}
where we regard $\Pol_{d}[t]_\text{{\normalfont{even}}} = \{0\}$ for $d < 0$.
\end{prop}

In the following, we suppose $m > N$.

Recall that the goal of Step 2 is to determine the space $\Sol(\n_+; \sigma_\lambda^{2N+1}, \tau_{m, \nu})$. By Proposition \ref{prop-FirststepFmethod}, this is to solve the system of partial differential equations
\begin{equation*}
\left(\widehat{d\pi_\mu}(C)\otimes \id_{\C_{m, \nu}}\right)\psi = 0 \quad (\forall C\in \n_+^\prime),
\end{equation*}
where  $\psi = \sum_{k =m-N}^{m+N}\left(T_{\nu-\lambda-k}g_k\right)h_k^+, \text{ for } g_k \in \Pol_{\nu-\lambda-k}[t]_\text{{\normalfont{even}}}$.

Theorem \ref{thm-step2} below gives a complete description of $\Sol(\n_+; \sigma_\lambda^{2N+1}, \tau_{m, \nu})$ for $m > N$.

\begin{thm}\label{thm-step2} Let $\lambda, \nu \in \C$, $N \in \N$ and $m\in \Z$ such that $m > N$. Then, the following three conditions on the quadruple $(\lambda, \nu, N, m)$ are equivalent.
\begin{enumerate}[label=\normalfont{(\roman*)}, topsep=3pt]
\item $\Sol(\n_+; \sigma_\lambda^{2N+1}, \tau_{m, \nu}) \neq \{0\}$.
\item $\dim_\C \Sol(\n_+; \sigma_\lambda^{2N+1}, \tau_{m, \nu}) = 1.$
\item $\lambda \in \Z_{\leq 1-m}$ and $\nu \in [1-N, N+1]\cap \Z$.
\end{enumerate}
Moreover, if one of the above (therefore all) conditions is satisfied, then
\begin{equation*}
\Sol(\n_+; \sigma_\lambda^{2N+1}, \tau_{m, \nu}) = \C\sum_{k = m-N}^{m+N}\left(T_{a -k}g_k\right)h_k^+,
\end{equation*}
where $a := \nu - \lambda$ and the polynomials $g_{k} \in \Pol_{a-k}[t]_\text{even} \enspace (k = m-N, \dots, m+N)$ are given as follows:
\begin{equation}\label{solution-all}
g_k(t) = \begin{cases}
\eqref{solution-1}, & \text{ if } k= m-N, \ldots, m-1,\\
\eqref{solution-2}, & \text{ if } k=m, \ldots, m+N.
\end{cases}
\end{equation}
\begin{align}
\label{solution-1}
& i^{k-m}(-1)^{\nu-1}\sum_{r=0}^{N-m+k} (-1)^rA(N-m+k, r)\Geg_{a-k+2(1-N-\nu+r)}^{\lambda+N-1-r}(it),\\
\label{solution-2}
& i^{m-k}\sum_{r=0}^{N+m-k} (-1)^rB(N-m+k, r)\Geg_{a+k-2(m+N-r)}^{\lambda+N-1-r}(it).
\end{align}
Here, $\widetilde{C}_\ell^\mu(z)$ stands for the renormalized Gegenbauer polynomial (see {\normalfont{(\ref{Gegenbauer-polynomial(renormalized)})}}), and we regard $\widetilde{C}_b^\mu \equiv 0$ for $b < 0$. The constants $A(d,r)$ and $B(d,r)$ are defined as in \eqref{const-AB}.
\end{thm}

\subsection{Proof of Theorems \ref{thm_main_class} and \ref{thm_main_const} for $m > N$}\label{section-proof-mainthms}
We shall discuss the proof of Theorem \ref{thm-step2} in Section \ref{section-reductiontheorems}. By using Theorem \ref{thm-step2} we first give a proof of Theorems \ref{thm_main_class} and \ref{thm_main_const} for $m > N$. The case $m < -N$ is considered in Section \ref{section-case_m_lessthan_-N}.

\begin{proof}[Proof of Theorem \ref{thm_main_class} for $m > N$]
Recall from (\ref{isomorphism-Fmethod-concrete}) that there exists a linear isomorphism
\begin{equation*}
\begin{tikzcd}[column sep = 1cm]
\Diff_{G^\prime}\left(C^\infty(S^3, \V_\lambda^{2N+1}), C^\infty(S^2, \L_{m,\nu})\right) \arrow[r, "\displaystyle{_{\Symb \otimes \id}}", "\thicksim"'] & \Sol(\n_+; \sigma_\lambda^{2N+1}, \tau_{m, \nu}).
\end{tikzcd}
\end{equation*}
Now, Theorem \ref{thm_main_class} follows from Theorem \ref{thm-step2}.
\end{proof}

\begin{proof}[Proof of Theorem \ref{thm_main_const} for $m>N$]
From (\ref{isomorphism-Fmethod-concrete}), \eqref{generators-Hom(V2N+1,CmHk)} and Theorem \ref{thm-step2}, in order to give an explicit expression of the generator $\D_{\lambda, \nu}^{N, m}$ of \eqref{DSBO-space}, it suffices to compute 
\begin{equation}\label{proof-generator}
(\Symb \otimes \id)^{-1} \left( \sum_{k = m-N}^{m+N}\left(T_{a-k}g_k\right)(\zeta)(\zeta_1 + i\zeta_2)^{k}u_{N-m+k}^\vee \right),
\end{equation}
where $g_k$ are the polynomials given in (\ref{solution-all}).

A straightforward computation by using the definition of $T_b$ (see \eqref{def-T_b}) gives the following identities for any $r=0, \ldots, N$:
\begin{align*}
\left(T_{a-k}\Geg_{a-k+2(1-N-\nu+r)}^{\lambda+N-1-r}(i\cdot)\right)(\zeta) &= (-1)^{\frac{a-k}{2}+1-N-\nu+r} (\zeta_1^2 +\zeta_2^2)^{N+\nu-1-r}\\
&\enspace \times\left(I_{a-k+2(1-N-\nu+r)}\Geg_{a-k+2(1-N-\nu+r)}^{\lambda+N-1-r}\right)(-\zeta_1^2-\zeta_2^2, \zeta_3),\\
\left(T_{a-k}\Geg_{a+k-2(m+N-r)}^{\lambda+N-1-r}(i\cdot)\right)(\zeta) &= (-1)^{\frac{a+k}{2}-m-N+r} (\zeta_1^2 +\zeta_2^2)^{m+N-r-k}\\
&\enspace \times\left(I_{a+k-2(m+N-r)}\Geg_{a+k-2(m+N-r)}^{\lambda+N-1-r}\right)(-\zeta_1^2-\zeta_2^2, \zeta_3).
\end{align*}
Here, $(I_\ell\Geg_\ell^\mu)(x,y) = x^{\frac{\ell}{2}}\Geg_\ell^\mu\left(\frac{y}{\sqrt{x}}\right)$, and $k$ runs over $k = m-N, \ldots, m-1$ in the first equation, and over $k = m, \ldots, m+N$ in the second one. Thus, we have
\begin{align*}
(T_{a-k}g_k)(\zeta) &= (-1)^\frac{a-m}{2}(-1)^N\sum_{r=0}^{N-m+k}\Big[A(N-m+k, r)(\zeta_1^2+\zeta_2^2)^{N+\nu-1-r} & \\
& \qquad \times \left(I_{a-k+2(1-N-\nu+r)}\Geg_{a-k+2(1-N-\nu+r)}^{\lambda+N-1-r}\right)(-\zeta_1^2-\zeta_2^2, \zeta_3)\Big], &(k = m-N, \ldots, m-1),\\
(T_{a-k}g_k)(\zeta) &= (-1)^\frac{a-m}{2}(-1)^N\sum_{r=0}^{N+m-k}\Big[B(N-m+k, r)(\zeta_1^2+\zeta_2^2)^{m+N-r-k} &\\
& \qquad \times \left(I_{a+k-2(m+N-r)}\Geg_{a+k-2(m+N-r)}^{\lambda+N-1-r}\right)(-\zeta_1^2-\zeta_2^2, \zeta_3)\Big], & (k = m, \ldots, m+N).
\end{align*}

Now, by dividing by $(-1)^\frac{a-m}{2}(-1)^N$, \eqref{proof-generator} amounts to

\begin{align*}
\sum_{k=m-N}^{m-1}\sum_{r=0}^{N-m+k}  A(N-m+k,r) (\Delta_{\R^2})^{N+\nu-1-r} \widetilde{\C}_{\lambda+N-r, 2-\nu-N-k+r}\left(\frac{\partial}{\partial x_1} + i \frac{\partial}{\partial x_2}\right)^k\otimes u_{N-m+k}^\vee\\
+\sum_{k=m}^{m+N}\sum_{r=0}^{N+m-k}  B(N-m+k,r) (\Delta_{\R^2})^{N+m-r-k} \widetilde{\C}_{\lambda+N-r, \nu-2m-N+k +r}\left(\frac{\partial}{\partial x_1} + i \frac{\partial}{\partial x_2}\right)^k\otimes u_{N-m+k}^\vee,
\end{align*}
where
\begin{equation*}
\begin{gathered}
\widetilde{\C}_{\lambda, \nu} := \Rest_{x_3 = 0} \circ \left(I_{\nu-\lambda} \widetilde{C}^{\lambda-1}_{\nu-\lambda}\right)\left(-\Delta_{\R^2}, \frac{\partial}{\partial x_3}\right).
\end{gathered}
\end{equation*}
A change in the indices given by $d = N-m+k$ gives the following expression of $\D_{\lambda, \nu}^{N,m}$:
\begin{multline*}
\sum_{d=0}^{N-1}\sum_{r=0}^{d}  A(d,r) (\Delta_{\R^2})^{N+\nu-1-r} \widetilde{\C}_{\lambda+N-r, 2-\nu-d-m+r}\left(\frac{\partial}{\partial x_1} + i \frac{\partial}{\partial x_2}\right)^{d+m-N}\otimes u_{d}^\vee\\
+\sum_{d=N}^{2N}\sum_{r=0}^{2N-d}  B(d,r) (\Delta_{\R^2})^{2N-d-r} \widetilde{\C}_{\lambda+N-r, \nu+d-m-2N+r}\left(\frac{\partial}{\partial x_1} + i \frac{\partial}{\partial x_2}\right)^{d+m-N}\otimes u_{d}^\vee.
\end{multline*}

The expression \eqref{Operator+} of the operator $\D_{\lambda, \nu}^{N,m}$ follows now from doing the change of variables $z = x_1 + ix_2$ and from diving by $2^{N+m}$ in the expression above. This can be easily checked by using the identities below:
\begin{align*}
\frac{\partial}{\partial x_1} + i\frac{\partial}{\partial x_2} = 2 \frac{\partial}{\partial \overline{z}}, \quad \Delta_{\R^2} = 4\frac{\partial^2}{\partial z \partial \overline{z}}.
\end{align*}
\end{proof}

\begin{rem}
The results obtained in \cite{Per-Val23} for the case $|m| > N = 1$ coincide with the ones obtained in this paper. In fact, the condition on the parameters $(\lambda, \nu)$ in \cite[Thm. 1.2]{Per-Val23} clearly coincides with that given in Theorem \ref{thm_main_class}. As for the generator of \eqref{DSBO-space}, in \cite{Per-Val23} it is given as follows for $\lambda \in \Z_{\leq 1-|m|}$ and $\nu = 0, 1, 2$.
\begin{equation}
\D_{\lambda, \nu}^{m} = 
\begin{cases}
{\displaystyle{\Rest_{x_3 = 0} \circ \frac{\partial^{m-1}}{\partial \overline{z}^{m-1}} \otimes u_0^\vee}}, &\text{ if } m > 1, \enspace \nu-\lambda = m-1,\\
\eqref{Operator_m=1_general+}, & \text{ if } m > 1, \enspace \nu - \lambda \geq m,\\
{\displaystyle{\Rest_{x_3 = 0} \circ \frac{\partial^{-m-1}}{\partial z^{-m-1}} \otimes u_2^\vee}}, &\text{ if } m < -1, \enspace \nu-\lambda = -m-1,\\
\eqref{Operator_m=1_general-}, & \text{ if } m < -1, \enspace \nu - \lambda \geq -m.
\end{cases}
\end{equation}
\begin{multline}\label{Operator_m=1_general+}
2^{2\nu}(-1)^{\nu+1}\frac{\Gamma\left(\lambda + \left[\frac{-\lambda-\nu - m+2}{2}\right]\right)}{\Gamma(\lambda + \left[\frac{\nu - \lambda - m -1}{2}\right])}\widetilde{\C}_{\lambda+1, 2-\nu-m}\frac{\partial^{2\nu+m-1}}{\partial z^\nu \partial \overline{z}^{\nu+m-1}} \otimes u_0^\vee\\
+\left((m(1-\nu)-\lambda+2)\widetilde{\C}_{\lambda, \nu-m} -2\gamma(\lambda-1, \nu-\lambda-m)\widetilde{\C}_{\lambda+1, \nu-m}\frac{\partial}{\partial x_3}\right)\frac{\partial^m}{\partial \overline{z}^m} \otimes u_1^\vee\\
-4\gamma(\lambda-1, \nu-\lambda-m)\widetilde{\C}_{\lambda+1, \nu-m}\frac{\partial^{m+1}}{\partial \overline{z}^{m+1}} \otimes u_2^\vee,
\end{multline}
\begin{multline}\label{Operator_m=1_general-}
-4\gamma(\lambda-1, \nu-\lambda+m)\widetilde{\C}_{\lambda+1, \nu+m}\frac{\partial^{-m+1}}{\partial z^{-m+1}} \otimes u_0^\vee\\
+\left((m(1-\nu)+\lambda-2)\widetilde{\C}_{\lambda, \nu+m} +2\gamma(\lambda-1, \nu-\lambda+m)\widetilde{\C}_{\lambda+1, \nu+m}\frac{\partial}{\partial x_3}\right)\frac{\partial^{-m}}{\partial z^{-m}} \otimes u_1^\vee\\
+ 2^{2\nu}(-1)^{\nu+1}\frac{\Gamma\left(\lambda + \left[\frac{-\lambda-\nu + m+2}{2}\right]\right)}{\Gamma(\lambda + \left[\frac{\nu - \lambda + m -1}{2}\right])}\widetilde{\C}_{\lambda+1, 2-\nu+m}\frac{\partial^{2\nu-m-1}}{\partial z^{\nu-m-1} \partial \overline{z}^\nu}  \otimes u_2^\vee.
\end{multline}
Here, $\gamma(\mu, \ell)$ denotes the constant defined in \eqref{gamma-def}.

Let us check that $\D_{\lambda, \nu}^m$ is proportional to the operator $\D_{\lambda, \nu}^{1, m}$ of Theorem \ref{thm_main_const}. We consider the case $\nu - \lambda \geq m > 1$ as other cases can be checked in a similar way. In this case, one can show that
\begin{equation}\label{Operators-relation}
\D_{\lambda, \nu}^m = 2\D_{\lambda, \nu}^{1,m}.
\end{equation}
In fact, by \eqref{Operator+}, $\D_{\lambda, \nu}^{1,m}$ amounts to
\begin{multline*}
2^{2\nu-2}A(0,0)\widetilde{\C}_{\lambda+1, 2-\nu-m}\frac{\partial^{2\nu+m-1}}{\partial z^\nu \partial \overline{z}^{\nu+m-1}} \otimes u_0^\vee\\
+\left(2B(1,0)\widetilde{\C}_{\lambda+1, \nu-m-1}\frac{\partial^2}{\partial z \partial \overline{z}} + \frac{1}{2}B(1,1)\widetilde{\C}_{\lambda, \nu-m}\right)\frac{\partial^m}{\partial \overline{z}^m}\otimes u_1^\vee\\
+B(2,0)\widetilde{\C}_{\lambda+1, \nu-m}\frac{\partial^{m+1}}{\partial \overline{z}^{m+1}} \otimes u_2^\vee.
\end{multline*}
On the other hand, from the three-term relation \cite[Eq. (9.5)]{Per-Val23}, we have
\begin{align*}
& (m(1-\nu)-\lambda+2)\widetilde{\C}_{\lambda, \nu-\lambda-m}
-2\gamma(\lambda-1, \nu-\lambda-m) \widetilde{\C}_{\lambda+1, \nu-m}\frac{\partial}{\partial x_3}\\
= \enspace & (m+1)(2-\nu)\widetilde{\C}_{\lambda, \nu-m} + 8\widetilde{\C}_{\lambda+1, \nu-m-1}\frac{\partial^2}{\partial z \partial \overline{z}}.
\end{align*}
And by \eqref{const-AB}, the following holds:
\begin{align*}
2^{2\nu-2}A(0,0) &= 2^{2\nu-1}(-1)^{\nu-1} \frac{\Gamma\left(\lambda + \left[\frac{-\lambda-\nu - m+2}{2}\right]\right)}{\Gamma(\lambda + \left[\frac{\nu - \lambda - m -1}{2}\right])},\\
2B(1,0) & = 4,\\
\frac{1}{2}B(1,1) & =\frac{1}{2}(m+1)(2-\nu),\\
B(2,0) &= -2\gamma(\lambda-1, \nu-\lambda-m).
\end{align*}
Now, \eqref{Operators-relation} clearly follows.
\end{rem}

\subsection{Proof of Corollary \ref{corol_main}}\label{section-proof-corol}
\begin{proof}[Proof of Corollary \ref{corol_main}] Suppose that there exist a non-zero differential symmetry breaking operator $\D$ in \eqref{DSBO-space}. By Theorem \ref{thm_main_class}, this in particular implies $\lambda, \nu \in \Z$.
Consider the representations $\Harm^N(\C^3) \in \widehat{O(3)}$ and $\Harm^{|m|}(\C^2) \in \widehat{O(2)}$, and following the notation given in \cite[Sec. 2.3]{KS18}, consider the principal series representations
\begin{align*}
I(\Harm^N(\C^3),\lambda) \quad \text{and} \quad J(\Harm^{|m|}(\C^2),\nu)
\end{align*}
of $O(4,1)$ and $O(3,1)$, respectively. Now, since
\begin{align*}
&\Harm^N(\C^3)\big\rvert_{SO(3)} \simeq V^{2N+1}, \text{ as } SO(3)\text{-modules},\\
&\Harm^{|m|}(\C^2)\big\rvert_{SO(2)} \simeq \C_{m} \oplus \C_{-m}, \text{ as } SO(2)\text{-modules},
\end{align*}
we clearly have
\begin{align}
\label{proof_corol_I}
&I(\Harm^N(\C^3),\lambda)\big\rvert_{SO_0(4,1)} \simeq C^\infty(S^3, \V_\lambda^{2N+1}), \text{ as } SO_0(4,1)\text{-modules},\\
\label{proof_corol_J}
&J(\Harm^{|m|}(\C^2),\nu)\big\rvert_{SO_0(3,1)} \simeq C^\infty(S^2, \L_{m,\nu}) \oplus C^\infty(S^2, \L_{-m,\nu}) , \text{ as } SO_0(3,1)\text{-modules}.
\end{align}
Since $\lambda \in \Z$, by the irreducibility criterion \cite[Thm. 14.15]{KS18} and by \cite[Ex. 2.8]{KS18}, $I(\Harm^N(\C^3),\lambda)$ is irreducible as an $O(4,1)$-module if and only if $\lambda \in \{1-N, 2+N\}$. On the other hand, by Theorem \ref{thm_main_class} we must have $\lambda \in \Z_{\leq 1-|m|}$. In particular, $\lambda \not\in \{1-N, 2+N\}$ as $|m| > N$. Therefore, $I(\Harm^N(\C^3),\lambda)$ is reducible as an $O(4,1)$-module, and by \cite[Lem. 15.3]{KS18} and \eqref{proof_corol_I}, $C^\infty(S^3, V_\lambda^{2N+1})$ is also reducible as an $SO_0(4,1)$-module. This proves (1).

In order to prove (2), we follow a similar argument. As $\nu \in \Z$, by \cite[Thm. 14.15]{KS18} and \cite[Ex. 2.8]{KS18}, $J(\Harm^{|m|}(\C^2),\nu)$ is irreducible as an $O(3,1)$-module if and only if $\nu \in \{1-|m|, 1+|m|\}\cup\{\nu \in \Z: 0<|\nu-1| < |m|\} = [1-|m|, 1+|m|]\cap \Z$. On the other hand, by Theorem \ref{thm_main_class} we must have $\nu \in [1-N, N+1]\cap \Z$. However, since $|m| > N$, we have $[1-N, N+1]\subset [1-|m|, 1+|m|]$. Hence, $J(\Harm^{|m|}(\C^2),\nu)$ is irreducible as an $O(3,1)$-module, and by \cite[Lem. 15.3]{KS18} and \eqref{proof_corol_J}, $C^\infty(S^2, \L_{\pm m,\nu})$ is also irreducible as an $SO_0(3,1)$-module. In particular, $\overline{\im(\D)} = C^\infty(S^2, \L_{m,\nu})$.
\end{proof}

\subsection{Reduction Theorems}\label{section-reductiontheorems}
The proof of Theorem \ref{thm-step2} will be separated into to reduction theorems: Theorem \ref{Thm-findingequations} (finding equations), and Theorem \ref{Thm-solvingequations} (solving equations). The first one was proved in \cite{Per-Val24}, but the second one will be proved in Section \ref{section-proof_of_solving_equations}, and conforms the most difficult and technical part of this paper. From now on, we assume $a = \nu - \lambda\in \N$, $m > N$ and $a \geq m-N$.

For any $\mu \in \C$ and any $\ell \in \N$, the \textit{imaginary Gegenbauer differential operator} $S_\ell^\mu$ is defined as follows (see (\ref{Gegen-imaginary})) (cf. \cite[Eq. (6.5)]{KKP16}):
\begin{equation*}
S_\ell^\mu = - \left((1 + t^2) \frac{d^2}{dt^2} + (1 + 2\mu)t\frac{d}{dt} - \ell(\ell + 2\mu)\right).
\end{equation*}
Now, for $f_0, f_{\pm 1}, f_{\pm 2} \ldots, f_{\pm N}  \in \C[t]$ set $\mathbf{f} := {}^t(f_{-N}, \ldots, f_0, \ldots, f_N)$ and let $L_j^{A, \pm}(\mathbf{f})(t)$ and $L_j^{B, \pm}(\mathbf{f})(t)$  $(j = 0, \dots, N)$ be the following polynomials:
\begin{equation}\label{expression-operators-L-AB}
\begin{array}{l}
\begin{rcases*}
L_j^{A, +}(\mathbf{f})(t) := S_{a+m-j}^{\lambda+j-1} f_j - 2(N-j)\frac{d}{dt}f_{j+1}.\\[6pt]

L_j^{A, -}(\mathbf{f})(t) := S_{a-m-j}^{\lambda+j-1} f_{-j} + 2(N-j)\frac{d}{dt}f_{-j-1}.
\end{rcases*} (j = 0, 1, \ldots, N)\\[30pt]

\begin{rcases*}
L_j^{B, +}(\mathbf{f})(t) := 2(-m(\lambda + a-1) + j(\lambda-1 + \vartheta_t))f_j \\
\hspace*{\fill}+ (N-j)\frac{d}{dt}f_{j+1} +(N+j) \frac{d}{dt}f_{j-1}.\\[6pt]

L_j^{B, -}(\mathbf{f})(t) := 2(m(\lambda + a-1) + j(\lambda-1 + \vartheta_t))f_{-j}\\
\hspace*{\fill} - (N+j)\frac{d}{dt}f_{-j+1} -(N-j) \frac{d}{dt}f_{-j-1}.
\end{rcases*} (j = 1, \ldots, N)
\end{array}
\end{equation}
Here, $\vartheta_t$ stands for the Euler operator $\vartheta_t := t\frac{d}{dt}$. For $j=0$, we set $L_0^{B,\pm} \equiv 0$.

Now, for $N \in \N$ and $a \geq m-N$ we define:
\begin{equation}\label{def-lambda-set}
\begin{aligned}
\Lambda(N,a) &:= \{N+1-a-q : \max(0, N-a+m) \leq q \leq 2N\}\\
& =\{1-N-a-r: 0\leq r \leq  \min(N+a-m, 2N)\}.
\end{aligned}
\end{equation}
Note that when $m-N \leq a < m+N$, the set $\Lambda(N, a)$ amounts to
\begin{align*}
\{N+1-a-q : q = N-a+m, N-a+m+1, \ldots, 2N\} = [1-N-a, 1-m]\cap \Z,
\end{align*}
while if $a \geq m+N$, it is equal to
\begin{align*}
\{N+1-a-q : q = 0, 1, \ldots, 2N\} = [1-N-a, N+1-a]\cap \Z.
\end{align*}

Given $\mathbf{g} = (g_k)_{k = m-N}^{m+N} = {}^t(g_{m-N}, \ldots, g_{m+N})$ with $g_k \in \C[t]$, we define the \lq\lq reflection\rq\rq\, $\mathbf{g}^\text{ref}$ of $\mathbf{g}$ as follows:
\begin{equation*}
\mathbf{g}^\text{ref} = {}^t(g_{-N}^\text{ref}, \ldots, g_N^\text{ref}) := \begin{pmatrix}
& & 1\\
& \iddots &\\
1 & &
\end{pmatrix}\mathbf{g} = {}^t(g_{m+N}, \ldots, g_{m-N}).
\end{equation*}
We also define the solution space of the system determined by \eqref{expression-operators-L-AB}:
\begin{equation}\label{equation-space}
\Xi(\lambda, a, N, m) := \left\{(g_k)_{k=m-N}^{m+N} \in \bigoplus_{k = m-N}^{m+N} \Pol_{a-k}[t]_\text{{\normalfont{even}}} : \begin{array}{l}
L_j^{A,\pm}(\mathbf{g}^\text{ref}) = 0\\[4pt]
L_j^{B,\pm}(\mathbf{g}^\text{ref}) = 0\\[4pt]
\forall j = 0, \dots, N
\end{array}
\right\}.
\end{equation}
Now, we have
\begin{thm}[{\cite[Thm. 5.2]{Per-Val24}}]\label{Thm-findingequations} Let $N \in \N$, $m \in \Z$ with $m > N$ and set $\psi = \sum_{k = m-N}^{m+N}\left(T_{a-k}g_k\right)h_k^+$. Then, the following two conditions on $\mathbf{g} := (g_k)_{k = m-N}^{m+N}$ are equivalent:
\begin{enumerate}[label=\normalfont{(\roman*)}, topsep=2pt]
\setlength{\itemsep}{3pt}
\item $\left(\widehat{d\pi_\mu}(C)\otimes \id_{\C_{m,\nu}}\right)\psi = 0$, for all $C \in \n_+^\prime$.
\item $L_j^{A, \pm}(\mathbf{g}^{\normalfont\text{ref}}) = L_j^{B, \pm}(\mathbf{g}^{\normalfont\text{ref}}) = 0$, for all $j = 0, \dots, N$.
\end{enumerate}
\end{thm}

\begin{thm}\label{Thm-solvingequations} Let $\lambda \in \C$, $a, N \in \N$ and $m \in \Z$ and suppose $a \geq N-m$ and $m > N$. Then, the following three conditions on the quadruple $(\lambda, a, N, m)$ are equivalent.
\begin{enumerate}[label=\normalfont{(\roman*)}, topsep=3pt]
\item $\Xi(\lambda, a, N, m) \neq \{0\}$.
\item $\dim_\C \Xi(\lambda, a, N, m)= 1.$
\item $\lambda \in \Lambda(N,a)$.
\end{enumerate}

Moreover, if one of the above (therefore all) conditions is satisfied, the generator of $\Xi(\lambda, a, N, m)$ is given by \eqref{solution-all}.
\end{thm}

The remaining task to complete the proof of Theorem~\ref{thm-step2} is to show Theorem~\ref{Thm-solvingequations}, which will be the main focus of the next section.

\section{Proof of Theorem \ref{Thm-solvingequations}: Solving the System $\Xi(\lambda, a, N, m)$}
\label{section-proof_of_solving_equations}

In this section we prove Theorem \ref{Thm-solvingequations}, which is the key result to show both Theorems \ref{thm_main_class} and \ref{thm_main_const}. In other words, we solve the system (\ref{equation-space}), which is given by the following differential equations:
\begin{equation}\label{ODEsystem}
\begin{array}{ll}
(A_j^+) & S_{a+m-j}^{\lambda+j-1} f_j - 2(N-j)\frac{d}{dt}f_{j+1} = 0.\\[6pt]

(A_j^-) & S_{a-m-j}^{\lambda+j-1} f_{-j} + 2(N-j)\frac{d}{dt}f_{-j-1} = 0.\\[6pt]

(B_j^+) & 2(-m(\lambda + a-1) + j(\lambda-1 + \vartheta_t))f_j 
+ (N-j)\frac{d}{dt}f_{j+1} +(N+j) \frac{d}{dt}f_{j-1} = 0.\\[6pt]

(B_j^-) & 2(m(\lambda + a-1) + j(\lambda-1 + \vartheta_t))f_{-j}
- (N+j)\frac{d}{dt}f_{-j+1} -(N-j) \frac{d}{dt}f_{-j-1} = 0.
\end{array}
\end{equation}
Here, $j \in \{0, 1, \ldots, N\}$ in $(A_j^\pm)$ while $j \in \{1, \ldots, N\}$ in $(B_j^\pm)$, and the functions $f_{\pm j}$ belong to the following space (see \eqref{def-Pol_even}):
\begin{equation}\label{condition-space-f+-j}
f_{\pm j}\in \Pol_{a-m\pm j}[t]_\text{{\normalfont{even}}} \subseteq \C[t]  \enspace (j = 0,1, \ldots, N).
\end{equation}

Even though this system is formed by linear ordinary differential equations of order at most 2 with polynomial solutions, it is an overdetermined system, since there are $2(2N+1)$ equations and $2N+1$ functions. This makes it difficult to solve, especially since there is not a clear relation between equations of type A and B.

Below we solve the system completely and show that it admits a nontrivial solution only for certain values of $\lambda$. More precisely, we prove that the system $\Xi(\lambda, a, N, m)$ admits a non-zero solution $(f_{-N}, \ldots, f_N)$ if and only if $\lambda \in \Lambda(N, a)$, and that in this case, the solution is unique up to a constant multiple, and is explicitly given by:
\begin{align*}
f_j(t) &= i^{-j}(-1)^{\nu-1}\sum_{r=0}^{N-j}(-1)^r A(N-j, r) \Geg_{a-m+j+2(1-N-\nu+r)}^{\lambda+N-1-r}(it) & (j= 1, \ldots, N),\\
f_{-j}(t) &= i^{-j}\sum_{r=0}^{N-j}(-1)^rB(N+j,r)\Geg_{a-m+j-2N+2r}^{\lambda+N-1-r}(it) & (j= 0, 1, \ldots, N).
\end{align*}
This is nothing but a restatement of \eqref{solution-all} where $(f_{-N}, \ldots, f_N) = (g_{m+N}, \ldots, g_{m-N})$.

We solve the system following a strategy divided in three phases:

\begin{itemize}[leftmargin=2cm, topsep=3pt]
\setlength{\itemsep}{2pt}
\item[\textbf{Phase 1}:] Solve equations $(A_N^\pm)$ and obtain $f_{\pm N}$ up to two constants, say $q_N^\pm \in \C$.

\item[\textbf{Phase 2}:] Inductively, for $j = 0, \dots, N-1$, obtain $f_{\pm j}$ by solving $(B_j^\pm)$. After that, check that these $f_{\pm j}$ satisfy equations $(A_j^{\pm})$. For $f_0$ we obtain two different expressions (one coming from the $+$ side and another coming from the $-$ side).

\item[\textbf{Phase 3}:] Check when the two expressions for $f_0$  obtained in Phase 2 coincide. Here, we obtain a compatibility condition for the constants $q_N^\pm$ and for the parameter $\lambda$. 
\end{itemize}

We carry out this three-phase strategy in the following and give a proof of Theorem \ref{Thm-solvingequations} at the end in Section \ref{section-proof_of_solving_equations-subsection}.

After completing Phases 1 and 2, we will have that the space of solutions $\Xi(\lambda, a, N, m)$ has dimension at most two, since the polynomials $f_{\pm j}$ depend on two constants $q_N^\pm$ that a priori can take any value. In Phase 3 we will deduce that this space has dimension at most one, since we obtain a compatibility condition (i.e., a linear relation) between $q_N^+$ and $q_N^-$. 

The three-phase strategy explained above was also used to solve the system \eqref{ODEsystem} when $m = N$ (see \cite{Per-Val24}). However, when $m > N$, the situation changes considerably, and even if the same idea works, the procedure is more complicated; especially Phase 3. One aspect that illustrates the fundamental difference between the cases $m = N$ and $m >N$, is that while in the former case the system \eqref{ODEsystem} has non-zero solutions for any $(\lambda, \nu) \in \C^2$ satisfying $\nu-\lambda \in \N$, when $m > N$ non-zero solutions do not exist if $\lambda \notin \Z$ (see Proposition \ref{prop-condition-lambda} below).\medskip

The proofs in our three-phase approach are not straightforward and use several technical results, including properties of renormalized Gegenbauer polynomials, imaginary Gegenbauer operators and hypergeometric functions, which are listed in the Appendix (Section \ref{section-appendix}). Let us start with Phase 1.

\subsection{Phase 1: Obtaining $f_{\pm N}$ up to constant---Necessary condition on $\lambda$}\label{section-phase1}

Since the system \eqref{ODEsystem} does not have non-zero solutions for $a < m-N$ (see Proposition \ref{prop-FirststepFmethod}), from now on we suppose that $a \geq m-N$.
Recall that we are also under the assumption $m>N$, so that $a \in \N_+$.

The first phase is quite simple and straightforward. By taking a look at \eqref{ODEsystem} we find that:
\begin{align*}
(A_N^+) \enspace S_{a+m-N}^{\lambda + N - 1} f_N = 0, && (A_N^-) \enspace  S_{a-m-N}^{\lambda + N - 1} f_{-N} = 0.
\end{align*}
Thus, from Theorem \ref{thm-Gegenbauer-solutions} and Lemma \ref{lemma-relationS-G}, we have
\begin{align}\label{expression_fN-and-f-N}
f_N(t) = q_N^+\Geg_{a+m-N}^{\lambda+N-1}(it), && f_{-N}(t) = q_N^-\Geg_{a-m-N}^{\lambda+N-1}(it),
\end{align}
for some constants $q_N^\pm \in \C$. This completes Phase 1.

Observe that the degree of $\Geg_{a-m-N}^{\lambda+N-1}(it)$ coincides with that of the space $\Pol_{a-m-N}[t]_\text{{\normalfont{even}}}$, so the expression for $f_{-N}$ makes sense. However, this does not happen for $f_N$. In fact, since $m>N$, we have
\begin{equation*}
a-m+N < a+m-N.
\end{equation*}
Thus, by \eqref{condition-space-f+-j}, we have
\begin{equation*}
f_N \in \Pol_{a-m+N}[t]_\text{{\normalfont{even}}} \subsetneq \Pol_{a+m-N}[t]_\text{{\normalfont{even}}} \ni \Geg_{a+m-N}^{\lambda+N-1}(it).
\end{equation*}
So a priori the degree of the polynomial $\Geg_{a+m-N}^{\lambda+N-1}(it)$ is greater than the degree of $f_N$. We will see that $\Geg_{a+m-N}^{\lambda+N-1}(it)$ actually has degree $a-m+N$ for an appropriate parameter $\lambda \in \C$ (in other words, that the higher terms of $\Geg_{a+m-N}^{\lambda+N-1}(it)$ vanish), so that the expression for $f_N$ is indeed correct. If we take a closer look at the definition of the renormalized Gegenbauer polynomial \eqref{Gegenbauer-polynomial(renormalized)}, we deduce that this can only happen if $\lambda \in \Z$ (an example of this type of phenomenon can be seen in Lemmas \ref{lemma-Gegenbauer-Mlk-condition} and \ref{lemma-KOSS}). The next result shows that, as a matter of fact, the system does not have non-zero solutions if $\lambda \notin \Z$.

\begin{prop}\label{prop-condition-lambda} Suppose that $\lambda \in \C\setminus \Z$ and that $m > N$. Then, the only solution of the system $\Xi(\lambda, a, N, m)$ is the zero solution. That is, $f_{\pm j} \equiv 0$ for all $j = 0, 1, \ldots, N$.
\end{prop}

The proposition above relies on the following key lemma.

\begin{lemma}\label{lemma-fj=0-implies-fj-1=0}
Suppose $m >N$ and let $M_\ell^k$ be the set defined in \eqref{def-Mlk-set} and $\gamma(\mu, \ell)$ the constant defined in \eqref{gamma-def}. Then, the following statements hold.
\begin{enumerate}[label=\normalfont{(\arabic*)}, topsep=0pt]
\item Let $j=0, 1, \ldots, N-1$ and suppose $\lambda+j-1 \notin M_{a+m-j}^{m-j-1}$. Then $f_{j+1} = 0$ implies $f_j = 0$. 
\item Let $j = 1, 2, \ldots, N$ and suppose $\gamma(\lambda, a+m-j-2) \neq 0$. Then $f_{-j+1} = 0$ implies $f_{-j} = 0$.
\end{enumerate}
\end{lemma}
\begin{proof}
(1) Suppose that $f_{j+1} = 0$ for some $j = 0, 1, \dots, N-1$. Then, by $(A_j^+)$ we have:
\begin{equation*}
S_{a+m-j}^{\lambda+j-1} f_j - 2(N-j) \frac{d}{dt}f_{j+1} = S_{a+m-j}^{\lambda+j-1} f_j = 0.
\end{equation*}
Thus, by Theorem \ref{thm-Gegenbauer-solutions} and Lemma \ref{lemma-relationS-G} we have that
\begin{equation*}
f_j(t) = q_j^+\Geg_{a+m-j}^{\lambda+j-1}(it)
\end{equation*}
for some $q_j^+ \in \C$. By the condition \eqref{condition-space-f+-j}, we have that $f_j \in \Pol_{a-m+j}[t]_\text{even}$. However, since $m>N$, we have $a-m+j < a+m-j$. That is,
\begin{equation*}
f_j \in \Pol_{a-m+j}[t]_\text{even} \subsetneq \Pol_{a+m-j}[t]_\text{even} \ni \Geg_{a+m-j}^{\lambda+j-1}(it).
\end{equation*}
By the hypothesis $\lambda+j-1 \notin M_{a+m-j}^{m-j-1}$ and by Lemma \ref{lemma-Gegenbauer-Mlk-condition} and Remark \ref{rem-Mlk}, we deduce $\Geg_{a+m-j}^{\lambda+j-1}(it) \notin \Pol_{a-m+j}[t]_\text{even}$. Thus $q_j^+ = 0$ necessarily. That is, $f_j = 0$.

(2) Suppose that $f_{-j+1} = 0$ for some $j = 1, 2, \dots, N$. By $(A_{j-1}^-)$ we have
\begin{equation*}
S_{a-m-j+1}^{\lambda+j-2} f_{-j+1} + 2(N-j+1) \frac{d}{dt}f_{-j} =  2(N-j+1) \frac{d}{dt}f_{-j} = 0.
\end{equation*}
Thus, $f_{-j}$ is constant. By \eqref{condition-space-f+-j} we deduce that if $a-m-j$ is odd, $f_{-j} = 0$ necessarily. Suppose then that $a-m-j$ is even and set $f_{-j} = q_{j}^- \in \C$. By $(A_j^-)$ we have that
\begin{equation*}
\begin{aligned}
&S_{a-m-j}^{\lambda+j-1}f_{-j} + 2(N-j)\frac{d}{dt}f_{-j-1} = 0\\
&\Rightarrow (a-m-j)(a-m+j+2\lambda-2)q_j^- = -2(N-j)\frac{d}{dt}f_{-j-1},
\end{aligned}
\end{equation*}
and by $(B_j^-)$ we obtain
\begin{equation}\label{proof-condition-lambda}
\begin{aligned}
0 &=  4(m(\lambda + a-1) + j(\lambda-1 + \vartheta_t))f_{-j}
- 2(N+j)\frac{d}{dt}f_{-j+1} - 2(N-j) \frac{d}{dt}f_{-j-1} \\
&= 4(m(\lambda + a-1) + j(\lambda-1 + \vartheta_t))q_{j}^- + (a-m-j)(a-m+j+2\lambda-2)q_j^-\\
& = (a+m+j)(2\lambda+a+m-j-2)q_j^-.
\end{aligned}
\end{equation}
Since $a-m-j$ is even, and by assumption we have
\begin{equation*}
0 \neq 2\gamma(\lambda, a+m-j-2) = 2\lambda+a+m-j-2,
\end{equation*}
equation \eqref{proof-condition-lambda} holds if and only if $q_j^- = 0$. That is, $f_{-j} = 0$.
\end{proof}

Now, we are ready to show Proposition \ref{prop-condition-lambda}.

\begin{proof}[Proof of Proposition \ref{prop-condition-lambda}]
Note that, as $\lambda \notin \Z$, we have in particular:
\begin{equation*}
\lambda+j-1 \notin M_{a+m-j}^{m-j-1} \text{ and }\gamma(\lambda, a+m-j-2) \neq 0, \text{ for all }j = 0, 1, \ldots, N.
\end{equation*}

As before, since $m>N$ we have $a-m+N < a+m-N$. Thus, from \eqref{expression_fN-and-f-N} and from Lemma \ref{lemma-Gegenbauer-Mlk-condition}, we deduce that $f_N = 0$ necessarily. Now, by using Lemma \ref{lemma-fj=0-implies-fj-1=0}(1) repeatedly, we have 
\begin{equation*}
f_N = f_{N-1} = \dots = f_0 = 0.
\end{equation*}
And analogously, by using Lemma \ref{lemma-fj=0-implies-fj-1=0}(2) repeatedly, we have 
\begin{equation*}
f_0 = f_{-1} = \dots = f_{-N} = 0.
\end{equation*}
Hence $f_{\pm j} = 0$ for all $j = 0, 1, \ldots, N$.
\end{proof}

\begin{rem} Note that the argument in the proof of Lemma \ref{lemma-fj=0-implies-fj-1=0}(1) relies strongly on the fact $m > N$, while Lemma \ref{lemma-fj=0-implies-fj-1=0}(2) holds for any $m \geq N$.
\end{rem}

\subsection{Phase 2: Obtaining all $f_{\pm j}$ up to constant}\label{section-phase2}

Let us begin Phase 2, during which we will determine the remaining functions $f_{\pm j}$. Before starting actual computations, it is useful to outline our approach in detail. This may help understand the main goal of this phase without delving into technical arguments.

We proceed as follows:

\begin{itemize}
\setlength{\itemsep}{1pt}
\item By using the expressions of $f_{\pm N}$ in \eqref{expression_fN-and-f-N}, we obtain $f_{\pm(N-1)}$ by solving $(B_{N}^\pm)$.

\item Next, by using the expressions of $f_{\pm N}$ and $f_{\pm(N-1)}$, we obtain $f_{\pm(N-2)}$ by solving $(B_{N-1}^\pm)$.

\item We repeat this process for any $0 < j < N$ and obtain $f_{\pm(j-1)}$ from the previous two functions $f_{\pm j}, f_{\pm (j+1)}$ by solving $(B_j^\pm)$.

\item At the end, for $j = 0$, we obtain $f_0$ by the same procedure, but this time we obtain two expressions; one that comes from the expressions of $f_j$, and another that comes from the expressions of $f_{-j}$. We will have to check that these two expressions coincide in Phase 3.
\end{itemize}
Therefore, the functions $f_{\pm j}$ can be obtained recursively by solving $(B_j^\pm)$. If we represent the order in which these functions are obtained in a diagram, it appears as follows (with zeros on both sides to indicate that we start from \lq\lq nothing\rq\rq):
\begin{equation}\label{diagram-simple}
\overbrace{0 \rightarrow f_{-N}}^\text{Phase 1} 
\underbrace{\rightarrow f_{-N+1} \rightarrow \cdots \rightarrow f_{-1} \rightarrow}_\text{Phase 2}
\overbrace{f_{-0} = f_{+0}}^\text{Phase 3}
\underbrace{\leftarrow f_1 \leftarrow \cdots \leftarrow f_{N-1} \leftarrow}_\text{Phase 2}
\overbrace{f_N \leftarrow 0}^\text{Phase 1}
\end{equation}
In the above outline, we did not mention equations $(A_j^\pm)$; however, we also check that they are solved by the obtained $f_{\pm j}$. More concretely, in the step $j$, when we solve $(B_j^\pm)$ and obtain $f_{\pm(j-1)}$ by using the expressions of $f_{\pm j}$ and $f_{\pm(j+1)}$, the function $f_{\pm(j-1)}$ is obtained up to subtraction by some constant term $c_{j-1}^\pm \in \C$ that comes from integrating $(B_j^\pm)$. Before proceeding to the next step, we verify that $f_{\pm(j-1)}$ solves $(A_{j-1}^\pm)$ if and only if $c_{j-1}^\pm = 0$. By doing this, we ensure a complete determination of the functions $f_{\pm j}$ while confirming that both $(A_j^\pm)$ and $(B_j^\pm)$ are satisfied.

By the explanation above, the actual order in which we solve the equations is as follows:
\begin{equation*}
\begin{tikzcd}
(A_N^-) \arrow[d] &[-0.8cm] (A_{N-1}^-) \arrow[d] &[-0.8cm]\cdots &[-0.8cm] (A_1^-) \arrow[d]  &[-0.8cm](A_0^-) &[-0.8cm] (A_{0}^+) &[-0.8cm] (A_{1}^+) \arrow[d] &[-0.8cm] \cdots &[-0.8cm] (A_{N-1}^+) \arrow[d] &[-0.8cm] (A_N^+) \arrow[d]\\[-0.2cm]
(B_N^-) \arrow[ru] & (B_{N-1}^-) \arrow[ru] & \cdots \arrow[ru] &(B_1^-) \arrow[ru] &[-0.8cm] &[-0.8cm] &[-0.8cm] (B_{1}^+) \arrow[lu] &[-0.8cm] \cdots \arrow[lu] &[-0.8cm] (B_{N-1}^+) \arrow[lu] &[-0.8cm] (B_N^+) \arrow[lu]
\end{tikzcd}
\end{equation*}

In Figures \ref{diagram_N_1} to \ref{diagram_N_4} we write a little more complex diagrams for several values of $N \in \N$ that represent the hierarchy of the system in a more complete way than the diagram \eqref{diagram-simple}.

\begin{figure}
\begin{minipage}{0.50\textwidth}
\centering
\includegraphics[scale=0.65]{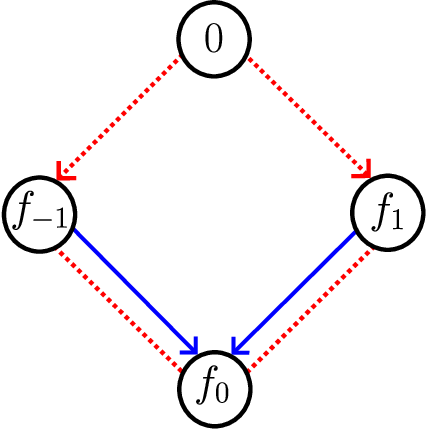}
\caption{Hierarchy for $N = 1$}\label{diagram_N_1}
\end{minipage}\hfill
\begin{minipage}{0.50\textwidth}
\centering
\includegraphics[scale=0.65]{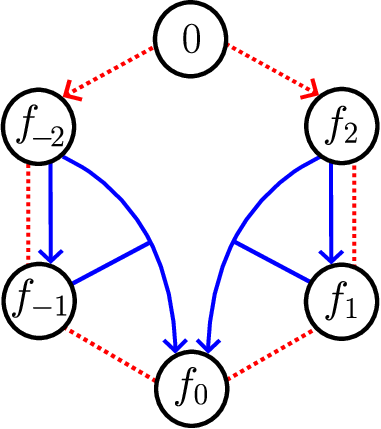}
\caption{Hierarchy for $N = 2$}\label{diagram_N_2}
\end{minipage}
\end{figure}

\begin{figure}
\begin{minipage}{0.50\textwidth}
\centering
\includegraphics[scale=0.65]{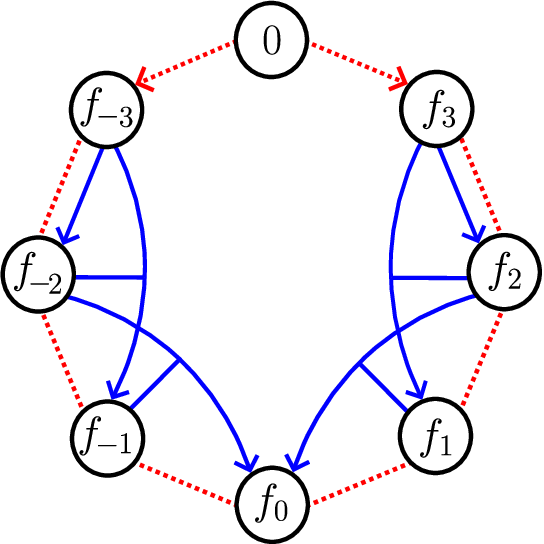}
\caption{Hierarchy for $N = 3$}\label{diagram_N_3}
\end{minipage}\hfill
\begin{minipage}{0.50\textwidth}
\centering
\includegraphics[scale=0.65]{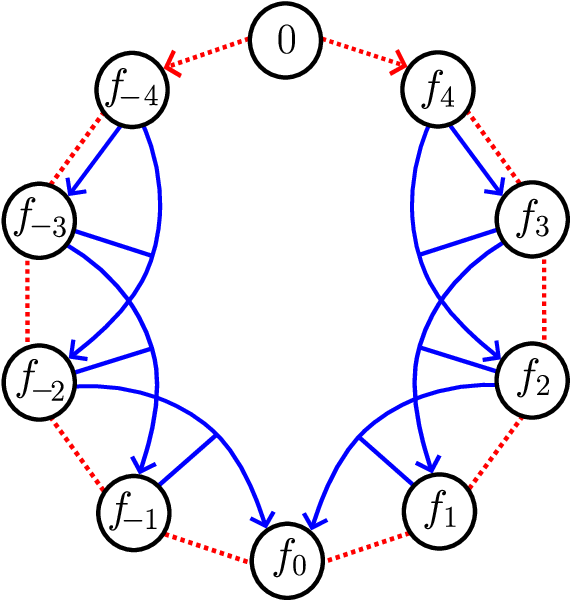}
\caption{Hierarchy for $N = 4$}\label{diagram_N_4}
\end{minipage}
\end{figure}

Concretely, we represent the functions as vertices in a graph, including 0 as the starting point. Equations $(A_{j}^\pm)$ and $(B_{j}^\pm)$ are depicted as edges connecting the related functions. Equations $(A_{j}^\pm)$ are shown as dotted edges in red, while $(B_{j}^\pm)$ are drawn as solid edges in blue. Moreover, some of these edges are directed (i.e., illustrated as arrows), to represent the order in which we solve the equations of the system. The first two arrows are dotted-red, corresponding to solving $(A_N^\pm)$ in Phase 1, while the subsequent arrows are solid-blue, representing the integration of $(B_j^\pm)$ in Phase 2. It is noteworthy that unlike the dotted-red edges (representing equations $(A_j^\pm)$), the solid-blue edges (representing equations $(B_j^\pm)$) connect three vertices, with the exception of the initial one associated to $(B_N^\pm)$, which connects only two, as can be deduced from \eqref{ODEsystem}.

Now that we have outlined the hierarchy of the system and the order in which we derive the polynomials $f_{\pm j}$, we will present a detailed proof of these results in the following.\\

For any $j = 0, 1, \ldots, N$, let $\Gamma_j^\pm$ be the following gamma factors:
\begin{equation}\label{def-Gamma_j+-}
\begin{aligned}
\Gamma_j^+ &:= \frac{\Gamma\left(\lambda+N-1+\left[\frac{a+m-N+1}{2}\right]\right)}{\Gamma\left(\lambda+N-1+\left[\frac{a+m-2N+j+1}{2}\right]\right)} = \frac{\Gamma\left(\lambda+\left[\frac{a+m+N-1}{2}\right]\right)}{\Gamma\left(\lambda\left[\frac{a+m+j-1}{2}\right]\right)},\\[4pt]
\Gamma_j^- &:= \frac{\Gamma\left(\lambda+N-1+\left[\frac{a-m-N+1}{2}\right]\right)}{\Gamma\left(\lambda+N-1+\left[\frac{a-m-2N+j+1}{2}\right]\right)} = \frac{\Gamma\left(\lambda+\left[\frac{a-m+N-1}{2}\right]\right)}{\Gamma\left(\lambda+\left[\frac{a-m+j-1}{2}\right]\right)}.
\end{aligned}
\end{equation}
These constants have the following properties, which are a direct consequence of the definition.

\begin{lemma}\label{lemma-Gamma_j+-} Let $\Gamma_j^\pm$ be the gamma factors defined in \eqref{def-Gamma_j+-} above, and let $\gamma(\mu, \ell)$ be as in \eqref{gamma-def}. Then, the following statements hold:
\begin{enumerate}[label=\normalfont{(\arabic*)}]
\item $\Gamma_N^+ = \Gamma_N^- = 1$.
\item For any $0 \leq j < N$, we have
\begin{equation*}
\Gamma_j^\pm = \Gamma_{j+1}^\pm \cdot \gamma(\lambda+N-1, a\pm m-2N+j).
\end{equation*}
\item For any $0 \leq j < N$, we have 
\begin{equation*}
\Gamma_j^\pm = \prod_{d=1}^{N-j}\gamma(\lambda+N-1, a\pm m-N-d).
\end{equation*}
\item $\Gamma_j^\pm \neq 0$ for all $0 \leq j < N$ if and only if $\Gamma_0^\pm \neq 0$.
\end{enumerate}
\end{lemma}

For $j = 0, 1, \dots, N$ and $r = 0, 1, \ldots, N-j$, we define the following constants:
\begin{equation}\label{definition-constants-Cjr+-}
C_{j, r}^\pm := \begin{pmatrix}
N-j\\
r
\end{pmatrix}
\frac{\Gamma(2N-r+1)}{\Gamma(N+j+1)}
\frac{\Gamma(N\mp m+1)}{\Gamma(N\mp m -r +1)}
\frac{\Gamma(\lambda+a-N+r-1)}{\Gamma(\lambda+a-N-1)}.
\end{equation}
Now, we have the following:

\begin{lemma} \label{lemma-recurrence-fj}
\begin{itemize}
\item[\normalfont{(1)}] For a given $j=1, \ldots, N-1$, suppose that $f_j$ and $f_{j+1}$ satisfy \eqref{expression-fj+} below, and suppose further that $\Gamma_{j-1}^+ \neq 0$. Then, $f_{j-1}$ solves $(A_{j-1}^+)$ and $(B_j^+)$ if and only if it also satisfies \eqref{expression-fj+}.
\begin{equation}\label{expression-fj+}
(-i)^{N-j} \frac{\Gamma(2N+1)}{\Gamma(N+j+1)}\Gamma_j^+
 f_j(t) = q_N^+ \sum_{r=0}^{N-j} C_{j,r}^+\Geg_{a+m-2N+j+2r}^{\lambda +N-1-r}(it).
\end{equation}
\item[\normalfont{(2)}] For a given $j=1, \ldots, N-1$, suppose that $f_{-j}$ and $f_{-j-1}$ satisfy \eqref{expression-fj-} below, and suppose further that $\Gamma_{j-1}^- \neq 0$. Then, $f_{-j+1}$ solves $(A_{j-1}^-)$ and $(B_j^-)$ if and only if it also satisfies \eqref{expression-fj-}.
\begin{equation}\label{expression-fj-}
i^{N-j} \frac{\Gamma(2N+1)}{\Gamma(N+j+1)}\Gamma_j^-
 f_{-j}(t) = q_N^- \sum_{r=0}^{N-j} C_{j,r}^-\Geg_{a-m-2N+j+2r}^{\lambda +N-1-r}(it).
\end{equation}
\end{itemize}
\end{lemma}

\begin{rem}
Note that \eqref{expression-fj+} and \eqref{expression-fj-}
coincide with \eqref{expression_fN-and-f-N} for $j = N$.
\end{rem}
\begin{proof}
We prove (1). The second assertion can be proved in a similar fashion. 

Take $j = 1, \ldots, N-1$ and suppose that $f_j$ and $f_{j+1}$ satisfy \eqref{expression-fj+}. Let us prove first that $f_{j-1}$ solves $(B_{j}^+)$ if and only if it satisfies \eqref{expression-fj+} up to subtraction by some constant. From the hypothesis $\Gamma_{j-1}^+ \neq 0$ and from Lemma \ref{lemma-Gamma_j+-}(2), it is clear that 
\begin{equation}\label{proof-constant-hypothesis}
\Gamma_{j-1}^+ = \Gamma_j^+ \gamma(\lambda+N-1, a+m-2N+j-1) \neq 0.
\end{equation}
In particular,
\begin{equation}\label{proof-constant-gammaj}
\begin{aligned}
\frac{1}{2}(-i)^{N-j}\frac{\Gamma(2N+1)}{\Gamma(N+j+1)}\Gamma_{j}^+ = & \enspace \frac{1}{2}(-i)^{N-j-1}\frac{\Gamma(2N+1)}{\Gamma(N+j+2)}\Gamma_{j+1}^+\\
& \times (-i)(N+j+1)\gamma(\lambda + N-1, a+m-2N+j) \neq 0.
\end{aligned}
\end{equation}
Since $f_j$ and $f_{j+1}$ satisfy \eqref{expression-fj+}, by using \eqref{gamma-product-property}, \eqref{derivative-Gegenbauer-1} and \eqref{derivative-Gegenbauer-2}, one can easily show that $(B_j^+)$ multiplied by \eqref{proof-constant-gammaj} amounts to
\begin{equation}\label{proof-equationBj+}
\begin{aligned}
-\frac{1}{2}(-i)^{N-j}\frac{\Gamma(2N+1)}{\Gamma(N+j)}\Gamma_{j}^+\frac{d}{dt}f_{j-1} = q_N^+\Big[2j\sum_{r=0}^{N-j}C_{j, r}^+\Geg_{a+m-2N+j+2r-2}^{\lambda+N-r}(it)\\
+\big((j-m)(\lambda + a-j-1)+2j(r+j-N)\big)\sum_{r=0}^{N-j} C_{j, r}^+\Geg_{a+m-2N+j+2r}^{\lambda+N-1-r}(it)\\
+ (N-j)(N+j+1)\left(\lambda+ \left[\frac{a+m+j-1}{2}\right]\right)\sum_{r=0}^{N-j-1}C_{j+1, r}^+\Geg_{a+m-2N+j+2r}^{\lambda+N-r}(it)\Big].
\end{aligned}
\end{equation}
Moreover, by using the three-term relation \eqref{KKP-1} and by reordering the terms, the right-hand side of \eqref{proof-equationBj+}, up to multiplication by the constant $q_N^+$, can be rewritten as follows:
\begin{align*}
&(j-m)(\lambda + a-j-1)C_{j,N-j}^+\Geg_{a+m-j}^{\lambda+j-1}(it) + \sum_{r=0}^{N-j-1}\Big[(\lambda+a+m-N+j+r-1)(N-j) \\
&\times(N+j+1)C_{j+1,r}^+ + ((j-m)(\lambda + a-j-1)+2j(r+j-N))C_{j,r}^+\Big]\Geg_{a+m-2N+j+2r}^{\lambda+N-r-1}(it)\\
& + 2jC_{j, N-j}^+\Geg_{a+m-j-2}^{\lambda+j}(it)
+\sum_{r=0}^{N-j-1}\Big[2jC_{j, r}^+
+(N-j)(N+j+1)C_{j+1, r}^+\Big]\Geg_{a+m-2N+j+2r-2}^{\lambda+N-r}(it)\\	
=& \; (j-m)(\lambda + a-j-1)C_{j,N-j}^+\Geg_{a+m-j}^{\lambda+j-1}(it) + \Big[(\lambda+a+m-2)(N-j)(N+j+1)C_{j+1, N-j-1}^+\\
& + ((j-m)(\lambda+a-j-1)-2j)C_{j, N-j-1}^+ + 2jC_{j, N-j}^+\Big]\Geg_{a+m-j-2}^{\lambda+j}(it) + \sum_{r=0}^{N-j-2}\Big[2jC_{j, r+1}^+\\
& +(\lambda+a+m-N+j+r-1)(N-j)(N+j+1)C_{j+1, r}^+ + ((j-m)(\lambda+a-j-1)\\
& +2j(r+j-N))C_{j,r}^+ + (N-j)(N+j+1)C_{j+1, r+1}^+\Big]\Geg_{a+m-2N+j+2r}^{\lambda+N-r-1}(it) + \Big[2jC_{j, 0}^+\\& + (N-j)(N+j+1)C_{j+1, 0}^+\Big]\Geg_{a+m-2N+j-2}^{\lambda+N}(it)\\
= &\sum_{r=0}^{N-j+1}C_{j-1, r}^+\Geg_{a+m-2N+j+2r-2}^{\lambda+N-r}(it).
\end{align*}
In the last equality, we have used the identities below that follow from a straightforward computation by using the definition of $C_{j,r}^+$ given in \eqref{definition-constants-Cjr+-}:
\begin{equation*}
\begin{array}{l l}
\hfill (j-m)(\lambda + a-j-1)C_{j,N-j}^+ & = \; C_{j-1, N-j+1}^+,\\[8pt]
(\lambda+a+m-2)(N-j)(N+j+1)C_{j+1, N-j-1}^+ &\\[4pt]
\hspace{2cm} +((j-m)(\lambda+a-j-1)-2j)C_{j, N-j-1}^+ + 2jC_{j, N-j}^+ & = \; C_{j-1, N-j}^+,\\[8pt]
(\lambda+a+m-N+j+r-1)(N-j)(N+j+1)C_{j+1, r}^+ &\\[4pt]
\hspace{2cm} + ((j-m)(\lambda+a-j-1) +2j(r+j-N))C_{j,r}^+ &\\[4pt]
\hfill + (N-j)(N+j+1)C_{j+1, r+1}^+ +2jC_{j, r+1}^+ & = \; C_{j-1, r+1}^+,\\[8pt]
\hfill 2jC_{j, 0}^++ (N-j)(N+j+1)C_{j+1, 0}^+ & = \; C_{j-1, 0}^+.
\end{array}
\end{equation*}
Thus, \eqref{proof-equationBj+} amounts to
\begin{equation*}
-\frac{1}{2}(-i)^{N-j}\frac{\Gamma(2N+1)}{\Gamma(N+j)}\Gamma_{j}^+\frac{d}{dt}f_{j-1} = q_N^+\sum_{r=0}^{N-j+1}C_{j-1, r}^+\Geg_{a+m-2N+j+2r-2}^{\lambda+N-r}(it).
\end{equation*}

Now, by multiplying both sides by
\begin{equation*}
2i\gamma(\lambda+N-r-1, a+m-2N+j+2r-1) = 2i\gamma(\lambda+N-1, a+m-2N+j-1) \stackrel{\mathclap{\scriptsize\mbox{by } \eqref{proof-constant-hypothesis}}}{\neq} 0,
\end{equation*}
and by using \eqref{derivative-Gegenbauer-1}, we can integrate $f_{j-1}$ and obtain
\begin{equation}\label{proof-expresion-fj-1}
(-i)^{N-j+1}\frac{\Gamma(2N+1)}{\Gamma(N+j)}\Gamma_{j-1}^+f_{j-1} = q_N^+\sum_{r=0}^{N-j+1}C_{j-1, r}^+\Geg_{a+m-2N+j+2r-1}^{\lambda+N-1-r}(it) + c_{j-1}^+
\end{equation}
for some constant $c_{j-1}^+ \in \C$. Hence, we proved that $f_{j-1}$ solves $(B_j^+)$ if and only if it satisfies \eqref{expression-fj+} modulo subtraction by $c_{j-1}^+$.

To complete the proof, let us show that $f_{j-1}$ solves $(A_{j-1}^+)$ if and only if $c_{j-1}^+ = 0$. If we multiply $(A_{j-1}^+)$ by
\begin{equation*}
\frac{1}{2}(-i)^{N-j+1}\frac{\Gamma(2N+1)}{\Gamma(N+j)}\Gamma_{j-1}^+ \neq 0,
\end{equation*}
by \eqref{expression-fj+}, \eqref{proof-expresion-fj-1} and \eqref{derivative-Gegenbauer-1}, we obtain the following equation:
\begin{align*}
&\frac{1}{2}S_{a+m-j+1}^{\lambda+j-2}\left(q_N^+\sum_{r=0}^{N-j+1}C_{j-1, r}^+\Geg_{a+m-2N+j+2r-1}^{\lambda+N-1-r}(it) + c_{j-1}^+\right)\\
&-2q_N^+(N-j+1)(N+j)\left(\lambda + \left[\frac{a+m+j-2}{2}\right]\right)\sum_{r=0}^{N-j}C_{j, r}^+\Geg_{a+m-2N+j+2r}^{\lambda+N-1-r}(it) = 0,
\end{align*}
which by \eqref{Slmu-identity-1}, \eqref{derivative-Gegenbauer-2} and the identity 
\begin{equation*}
(N-j+1)(N+j)C_{j, r}^+ = (N-j+1-r)C_{j-1, r}^+,
\end{equation*}
amounts to the following:
\begin{align*}
& q_N^+\sum_{r=0}^{N-j+1}(N-j+1-r)(\vartheta_t + 2\lambda + j + a + m-3)C_{j-1, r}^+\Geg_{a+m-2N+j-1+2r}^{\lambda+N-1-r}\\
&-2q_N^+(N-j+1)(N+j)\left(\lambda + \left[\frac{a+m+j-2}{2}\right]\right)\sum_{r=0}^{N-j}C_{j, r}^+\Geg_{a+m-2N+j+2r}^{\lambda+N-1-r}(it)\\
&+ \frac{1}{2}(a+m-j+1)(2\lambda+a+m+j-3)c_{j-1}^+\\
= \;& 2q_N^+\sum_{r=0}^{N-j}(N-j+1-r)C_{j-1, r}^+\Bigg[(\lambda+a+m+j-2+r-N)\Geg_{a+m-2N+j-1+2r}^{\lambda+N-1-r}(it)\\
& + \Geg_{a+m-2N+j-3+2r}^{\lambda+N-r}(it) - \left(\lambda + \left[\frac{a+m+j-2}{2}\right]\right)\Geg_{a+m-2N+j-1+2r}^{\lambda+N-r}(it)\Bigg]\\
&+ \frac{1}{2}(a+m-j+1)(2\lambda+a+m+j-3)c_{j-1}^+\\
=&\; 0.
\end{align*}
However, by the three-term relation \eqref{KKP-1}, the sum of Gegenbauer polynomials above vanish. Thus, $(A_{j-1}^+)$ is satisfied if and only if
\begin{equation}\label{proof-equation-Aj-1}
\frac{1}{2}(a+m-j+1)(2\lambda+a+m+j-3)c_{j-1}^+ = 0.
\end{equation}
We note that if $a+m+j-1$ is odd, then $c_{j-1}^+ = 0$ necessarily by \eqref{condition-space-f+-j}. In particular, \eqref{proof-equation-Aj-1} above is satisfied. On the contrary, if $a+m+j-1$ is even, we have
\begin{equation*}
\gamma(\lambda+N-1, a+m-2N+j-1) = \frac{1}{2}(2\lambda+a+m+j-3),
\end{equation*}
which is non-zero from \eqref{proof-constant-hypothesis}. Hence, \eqref{proof-equation-Aj-1} is satisfied if and only if $c_{j-1}^+ = 0$. In particular, $f_{j-1}$ is given by \eqref{expression-fj+} and it satisfies $(A_{j-1}^+)$ and $(B_j^+)$.
\end{proof}

\begin{rem}\label{rem-lemma-recurrence-forj=N}
Note that by taking a closer look at the proof, Lemma \ref{lemma-recurrence-fj} above is also true for $j= N$. That is, if $f_N$ is given by \eqref{expression-fj+}, and if $\Gamma_{N-1}^+ = \gamma(\lambda+N-1, a+m-N-1) \neq 0$, then $f_{N-1}$ satisfies $(A_{N-1}^+)$ and $(B_N^+)$ if and only if it is also given by \eqref{expression-fj+}. (The analogous statement for $f_{-N}$ and $f_{1-N}$ clearly holds as well.)
\end{rem}

Now, as a consequence of the previous result, we are able to complete Phase 2 when $\Gamma_0^\pm \neq 0$, as the next lemma shows.

\begin{lemma}\label{lemma-expressions-fj}
\begin{itemize}
\item[\normalfont{(1)}] Suppose that $f_{N}$ is given by \eqref{expression_fN-and-f-N} and suppose further that $\Gamma_0^+ \neq 0$. Then,
$f_0, f_1, \ldots, f_{N}$ solve $(A_j^+)$ and $(B_{j}^+)$ (for all $j = 0, \ldots, N$) if and only if $f_j$ is given by \eqref{expression-fj+} for all $j = 0, \ldots, N$.
\item[\normalfont{(2)}] Suppose that $f_{-N}$ is given by \eqref{expression_fN-and-f-N} and suppose further that $\Gamma_0^- \neq 0$. Then,  $f_{-N}, f_{1-N}, \ldots, f_0$ solve $(A_j^-)$ and $(B_{j}^-)$ (for all $j = 0, \ldots, N$) if and only if $f_{-j}$ is given by \eqref{expression-fj-} for all $j = 0, \ldots, N$.
\end{itemize}
\end{lemma}

\begin{proof}
The result is a direct consequence of Lemma \ref{lemma-recurrence-fj}, Remark \ref{rem-lemma-recurrence-forj=N} and Lemma \ref{lemma-Gamma_j+-}(4).
\end{proof}

In the following, we attempt to complete Phase 2 in the cases $\Gamma_0^+ = 0$ and $\Gamma_0^- = 0$. We start by doing an important observation.
\begin{rem}\label{rem-normalization-fj+-} 
If $\Gamma_0^\pm \neq 0$, we note that \eqref{expression-fj+} and \eqref{expression-fj-} are equivalent to the following expressions for a suitable redefinition of the constants $q_N^\pm$:
\begin{align}\label{expression-fj+-normalized}
f_j(t) &= (-i)^j\frac{\Gamma(N+j+1)}{\Gamma(N+1)}\frac{\Gamma\left(\lambda+\left[\frac{a+m+j-1}{2}\right]\right)}{\Gamma\left(\lambda+\left[\frac{a+m-1}{2}\right]\right)}\widetilde{q_N^+}\sum_{r=0}^{N-j}C_{j,r}^+\Geg_{a+m-2N+j+2r}^{\lambda+N-1-r}(it),\\
\label{expression-fj--normalized}
f_{-j}(t) &= i^j\frac{\Gamma(N+j+1)}{\Gamma(N+1)}\frac{\Gamma\left(\lambda+\left[\frac{a-m+j-1}{2}\right]\right)}{\Gamma\left(\lambda+\left[\frac{a-m-1}{2}\right]\right)}\widetilde{q_N^-}\sum_{r=0}^{N-j}C_{j,r}^-\Geg_{a-m-2N+j+2r}^{\lambda+N-1-r}(it).
\end{align}
In fact, it suffices to define $\widetilde{q_N^\pm}$ such that
\begin{equation*}
q_N^\pm = (\pm i)^{N}\frac{\Gamma(2N+1)}{\Gamma(N+1)}\Gamma_0^\pm\widetilde{q_N^\pm}.
\end{equation*}
The key of this renormalization of the constants $q_N^\pm$ is that while \eqref{expression-fj+} and \eqref{expression-fj-} hold only when $\Gamma_0^\pm \neq 0$, the expressions \eqref{expression-fj+-normalized} and \eqref{expression-fj--normalized} above are satisfied independently of the value of $\Gamma_0^\pm$. This can be deduced from Propositions \ref{prop-phase2} and \ref{prop-phase3} below.
\end{rem}

\begin{lemma}\label{lemma-Gamma+zero}
Suppose that $\Gamma_0^+ = 0$, or equivalently, that
\begin{equation*}
\gamma(\lambda+N-1, a+m-N-s) = 0, \enspace \text{ for some } s = 1, 2, \ldots, N.
\end{equation*}
Then, $f_0, f_1, \ldots, f_N$ satisfying \eqref{condition-space-f+-j} solve $(A_j^+)$ and $(B_j^+)$ for all $j = 0, \ldots, N$ if and only if $f_{j}$ is given by \eqref{expression-fj+-normalized} for all $j = 0, 1, \ldots, N$ for some $\widetilde{q_N^+} \in \C$. In particular $f_N = f_{N-1} = \cdots = f_{N-s+1} \equiv 0$.
\end{lemma}
\begin{proof}
By definition (see \eqref{gamma-def}), $\gamma(\lambda+N-1, a+m-N-s) = 0$ amounts to
\begin{equation}\label{proof-condition-Gamma+=0}
a+m-N-s \in 2\N\text{ and } \lambda+\frac{a+m+N-s-2}{2}=0.
\end{equation}
As we pointed out in the proof of Proposition \ref{prop-condition-lambda}, since $m>N$, by \eqref{condition-space-f+-j} and \eqref{expression_fN-and-f-N}, we have
\begin{equation*}
f_N \in \Pol_{a-m+N}[t]_\text{even} \subsetneq \Pol_{a+m-N}[t]_\text{even} \ni \Geg_{a+m-N}^{\lambda+N-1}(it). 
\end{equation*}
Thus, by Lemma \ref{lemma-Gegenbauer-Mlk-condition}, $\lambda+N-1 \in M_{a+m-N}^{m-N-1}$ or $f_N \equiv 0$. The first condition is not possible due to \eqref{proof-condition-Gamma+=0}. Hence, $f_N \equiv 0$ necessarily.

Now, observe that, by \eqref{proof-condition-Gamma+=0}
\begin{equation*}
\lambda+j-1 =-\frac{a+m+N-s-2j}{2} > -\left[\frac{a+m-j+1}{2}\right], \text{ for } j = N-s+1, \ldots, N.
\end{equation*}
In particular $\lambda+j-1 \notin M_{a+m-j}^{m-j-1}$ for $j = N-s+1, \ldots, N$ as
\begin{equation*}
M_{a+m-j}^{m-j-1} = \left\{-\left[\frac{a+m-j+1}{2}\right] - d : d= 0, 1, \ldots, \left[\frac{a-m+j}{2}\right]\right\}. 
\end{equation*}
Thus, by using Lemma \ref{lemma-fj=0-implies-fj-1=0}(1) recursively, we deduce that $f_N = f_{N-1} = \ldots = f_{N-s+1} \equiv 0$. In particular, $f_{N-s+1}, \ldots, f_N$ 
clearly solve $(A_{N-s+1}^+), \ldots, (A_N^+)$ and $(B_{N-s+2}^+), \ldots, (B_N^+)$, and they
satisfy \eqref{expression-fj+-normalized} since for $j = N-s+1, \ldots, N$
\begin{equation*}
\frac{\Gamma\left(\lambda+\left[\frac{a+m+j-1}{2}\right]\right)}{\Gamma\left(\lambda+\left[\frac{a+m-1}{2}\right]\right)} = \prod_{d=1}^{j}\gamma(\lambda+N-1, a+m-2N+j-d) = 0.
\end{equation*}
Now, by $(A_{N-s}^+)$, Theorem \ref{thm-Gegenbauer-solutions} and Lemma \ref{lemma-relationS-G} we have
\begin{equation*}
f_{N-s}(t) = q_{N-s}^+\Geg_{a+m-N+s}^{\lambda+N-s-1}(it),
\end{equation*}
for some constant $q_{N-s}^+ \in \C$. Note that by \eqref{proof-condition-Gamma+=0} and Lemma \ref{lemma-Gegenbauer-Mlk-condition}, $\Geg_{a+m-N+s}^{\lambda+N}(it)$ is constant since
\begin{equation*}
\lambda+N-s-1 = -\left[\frac{a+m-N+s+1}{2}\right] \in M_{a+m-N+s}^{k}, \text{ for all } 0 \leq k < \left[\frac{a+m-N+s}{2}\right].
\end{equation*}
By the same reason,
\begin{equation*}
\sum_{r=0}^{s}C_{N-s,r}^+\Geg_{a+m-N-s+2r}^{\lambda+N-1-r}(it)
\end{equation*}
is constant. Thus, for an appropriate renormalization $\widetilde{q_N^+}$ of $q_{N-s}^+$, we have
\begin{equation*}
f_{N-s}(t) = (-i)^{N-s}\frac{\Gamma(2N-s+1)}{\Gamma(N+1)}\frac{\Gamma\left(\lambda+\left[\frac{a+m+N-s-1}{2}\right]\right)}{\Gamma\left(\lambda+\left[\frac{a+m-1}{2}\right]\right)}\widetilde{q_N^+}\sum_{r=0}^{s}C_{N-s,r}^+\Geg_{a+m-N-s+2r}^{\lambda+N-1-r}(it).
\end{equation*}
In particular, $f_{N-s}$ satisfies \eqref{expression-fj+-normalized}.

To determine the remaining functions $f_0, \ldots, f_{N-s-1}$, one can repeat the argument in the proof of Lemma \ref{lemma-recurrence-fj}(1) and recursively obtain $f_{j-1}$ from $f_{j}$ and $f_{j+1}$ (for $j = 1, \ldots, N-s$).

The same arguments are valid also in this case as, by \eqref{proof-condition-Gamma+=0}, one has:
\begin{equation*}
\gamma(\lambda+N-1, a+m-2N+j-1)\neq 0, \text{ for all } j = 1, \ldots, N-s.
\end{equation*}
This allows us to use \eqref{derivative-Gegenbauer-1} to integrate $(B_j^+)$ and obtain $f_{j-1}$ up to subtraction by a constant $c_{j-1}^+$, and to show that $(A_{j-1}^+)$ is satisfied if and only if $c_{j-1}^+ = 0$, for $j = 0, 1, \ldots, N-s$.

By doing this we obtain that, indeed, all $f_{0}, f_1, \ldots, f_N$ are given by \eqref{expression-fj+-normalized} and that they solve $(A_{j}^+)$ and $(B_j^+)$ (for all $j = 0, 1, \ldots, N)$.
\end{proof}

Now we are ready to complete Phase 2.

\begin{prop}[\textbf{Phase 2}]\label{prop-phase2}
Let $\Gamma_0^-$ be the constant defined in \eqref{def-Gamma_j+-}. Then, the following assertions hold.
\begin{enumerate}[label=\normalfont{(\arabic*)}]
\item Suppose $\Gamma_0^- \neq 0$. Then $f_0, f_1, \ldots, f_N$ (respectively $f_{0}, f_{-1}, \ldots, f_{-N}$) solve $(A_j^+)$ and $(B_j^+)$ (respectively $(A_j^-)$ and $(B_j^-))$ for all $j = 0, 1, \ldots N$, if and only if $f_j$ is given by \eqref{expression-fj+-normalized} (respectively $f_{-j}$ is given by \eqref{expression-fj--normalized}) for all $j = 0, 1, \ldots, N$.

\item Suppose $\Gamma_0^- = 0$. Then, the unique solution of the system is the zero solution.
\end{enumerate}
\end{prop}

\begin{proof}
(1) The first statement is a consequence of Lemma \ref{lemma-expressions-fj}, Remark \ref{rem-normalization-fj+-} and Lemma \ref{lemma-Gamma+zero}.

(2) For the second statement, we use Lemma \ref{lemma-fj=0-implies-fj-1=0}. Suppose that $\Gamma_0^- = 0$, that is
\begin{equation*}
\gamma(\lambda+N-1, a-m-N-s) = 0, \enspace \text{ for some } s = 1, 2, \ldots, N.
\end{equation*}
In other words,
\begin{equation}\label{proof-condition-Gamma-=0}
a-m-N-s \in 2\N \text{ and } \lambda+\frac{a-m+N-s-2}{2} = 0.
\end{equation}
As we argued in the proof of Lemma \ref{lemma-Gamma+zero}, by \eqref{condition-space-f+-j}, \eqref{expression_fN-and-f-N} and by Lemma \ref{lemma-Gegenbauer-Mlk-condition}, we have that $f_N \equiv 0$, or $\lambda+N-1 \in M_{a+m-N}^{m-N-1}$. However, by \eqref{proof-condition-Gamma-=0} the second condition cannot be satisfied. Hence $f_N \equiv 0$ necessarily. In fact, by \eqref{proof-condition-Gamma-=0} we have that
\begin{equation*}
\lambda+j-1 =-\frac{a-m+N-s-2j}{2} > -\left[\frac{a+m-j+1}{2}\right], \text{ for all } j = 0, \ldots, N.
\end{equation*}
So $\lambda+j-1 \notin M_{a+m-j}^{m-j-1}$ for all $j = 0, \ldots, N$. Thus, by Lemma \ref{lemma-fj=0-implies-fj-1=0}(1) we have $f_0 = f_1 = \cdots = f_N \equiv 0$.

Analogously, and by \eqref{proof-condition-Gamma-=0} again, we have
\begin{equation*}
\gamma(\lambda, a+m-j-2) \neq 0, \text{ for all } j=0, 1, \ldots, N.
\end{equation*}
Indeed, if $a+m-j-2$ is odd, the left-hand side is equal to 1. And if $a+m-j-2$ is even, we have
\begin{equation*}
2\gamma(\lambda, a+m-j-2) = 2\lambda + a+m-j-2 =
2m-j-N+s > s \geq 1.
\end{equation*}
Thus, by Lemma \ref{lemma-fj=0-implies-fj-1=0}(2) we have
$f_{-N} = \cdots = f_0 \equiv 0$.
\end{proof}

\begin{rem} As a consequence of Propositions \ref{prop-condition-lambda} and \ref{prop-phase2}, we deduce that a necessary condition for the system \eqref{ODEsystem} to have a non-zero solution is that
\begin{equation*}
\lambda \in \Z\setminus\left\{-\frac{a-m+N-d-2}{2}: d= 1, \ldots, N\right\}.
\end{equation*}
We will see in Phase 3 that this condition can be improved.
\end{rem}

\subsection{Phase 3: Proving $f_{-0} = f_{+0}$}\label{section-phase3}
Due to Proposition \ref{prop-phase2}, one remaining task in order to solve the system \eqref{ODEsystem} is to determine when $f_{-0} = f_{+0}$. That is, when the following equality holds:
\begin{equation}\label{identity-phase3}
\widetilde{q_N^-}\sum_{r=0}^{N}C_{0,r}^-\Geg_{a-m-2N+2r}^{\lambda+N-1-r}(it) = \widetilde{q_N^+}\sum_{r=0}^{N}C_{0,r}^+\Geg_{a+m-2N+2r}^{\lambda+N-1-r}(it).
\end{equation}

Note that if $m-N \leq a < m$, we have $f_{-j} \equiv 0$ for all $j = 0, 1, \ldots, N$ by \eqref{condition-space-f+-j}. Hence, in this case Phase 3 is trivial. In the following we accomplish Phase 3 in the non-trivial case, that is, for $a \geq m$. This is done concretely in Proposition \ref{prop-phase3} below.
We consider the case $m-N \leq a < m$ at the end, in Lemma \ref{lemma-case-m-N<=a<m}.

\begin{prop}[\textbf{Phase 3}]\label{prop-phase3} Let $(\lambda, N, a, m) \in \Z\times \N \times \N \times \Z$ with $m > N$, and suppose further that $a \geq m$. Then, the identity \eqref{identity-phase3} holds if and only if the tuple $(\lambda, N, a, m, \widetilde{q_N^\pm})$ satisfies \eqref{condition-phase3-lambda} and \eqref{condition-phase3-qN+-} below.
\begin{equation}\label{condition-phase3-lambda}
\lambda \in \Lambda(N,a) \enspace (\text{see \eqref{def-lambda-set}}),
\end{equation}
\begin{equation}\label{condition-phase3-qN+-}
\widetilde{q_N^+} - \frac{\Gamma\left(\lambda+\left[\frac{a+m-1}{2}\right]\right)}{\Gamma\left(\lambda+\left[\frac{a-m-1}{2}\right]\right)}\widetilde{q_N^-} = 0.
\end{equation}
\end{prop}
In order to prove the result above, we fix some notation and show a preliminary lemma. Unless otherwise is stated, from now on we suppose that $a \geq m$.

Let $\alpha \equiv \alpha(\lambda, N, a, m)$ and $\beta \equiv \beta(\lambda, N, a, m)$ be the following polynomials on $\lambda$ of degree $N$
\begin{align}
\label{def-alpha}
\alpha(\lambda, N, a, m) &:= \sum_{r = 0}^{N}(-1)^r C_{0,r}^+\frac{\Gamma\left(\left[\frac{a+m}{2}\right]\right)}{\Gamma\left(\left[\frac{a+m}{2}\right]-N+r\right)},\\
\label{def-beta}
\beta(\lambda, N, a, m) &:= \sum_{r = 0}^{N}(-1)^r C_{0,r}^+\frac{\Gamma\left(\left[\frac{a+m+2}{2}\right]\right)}{\Gamma\left(\left[\frac{a+m+2}{2}\right]-N+r\right)},
\end{align}
and let $P \in \C[\lambda]$ be defined as follows:
\begin{equation}\label{def-P(lambda)}
\begin{aligned}
P(\lambda):= &\left(\lambda+\left[\frac{a+m-1}{2}\right]\right)\left[\frac{a+m}{2}\right]\alpha(\lambda, N, a, m)\beta(\lambda, N, a, -m)\\
&- \left(\lambda+\left[\frac{a-m-1}{2}\right]\right)\left[\frac{a-m}{2}\right]\alpha(\lambda, N, a, -m)\beta(\lambda, N, a, m).
\end{aligned}
\end{equation}
Further, we define $Q \in \C[\lambda]$ as
\begin{equation}\label{def-Q(lambda)}
Q(\lambda) := \frac{\Gamma(N+m+1)\Gamma(N-m+1)}{\Gamma(m)\Gamma(1-m)}\prod_{q=0}^{2N}(\lambda-N-1+a+q).
\end{equation}
Then, we have
\begin{lemma}\label{lemma-P=Q}
Suppose $a \geq m+2N+2$. Then, the polynomials $P$ and $Q$ defined in \eqref{def-P(lambda)} and \eqref{def-Q(lambda)} above coincide.
\end{lemma}
\begin{proof}
In order to show $P = Q$, it suffices to prove the following three facts:
\begin{enumerate}[label=\normalfont{(\arabic*)}]
\item $\deg P = \deg Q = 2N+1$.
\item The top terms of $P$ and $Q$ coincide.
\item $P(\lambda) = 0$ for any $\lambda \in \Lambda(N,a) = \{N+1-a-q: q=0, 1, \ldots, 2N\}$. 
\end{enumerate}
The first one is straightforward, since $\alpha$ and $\beta$ have degree $N$ by \eqref{def-alpha} and \eqref{def-beta}, and since $Q$ is the product of $2N+1$ monomials.

The second one can be shown easily. In fact, by \eqref{definition-constants-Cjr+-}, the top term of $P$ is given by
\begin{equation*}
\left(\left[\frac{a+m}{2}\right]-\left[\frac{a-m}{2}\right]\right)\frac{\Gamma(N+m+1)\Gamma(N-m+1)}{\Gamma(1+m)\Gamma(1-m)},
\end{equation*}
which coincides with the top term of $Q$ since $\left[\frac{a+m}{2}\right]-\left[\frac{a-m}{2}\right] = m$.

Let us show (3), which is the hardest part. Set $\lambda = N+1-a-q$ with $q = 0, 1, \ldots, 2N$. Then, $\alpha, \beta$ amount to
\begin{align*}
\alpha(\lambda, N, a, m) &= \frac{\Gamma(2N-q+1)}{\Gamma(N+1)}\frac{\Gamma(\left[\frac{a+m}{2}\right])}{\Gamma(\left[\frac{a+m}{2}\right] -N + q)}\widetilde{\alpha}(a, m),\\
\beta(\lambda, N, a, m) &= \frac{\Gamma(2N-q+1)}{\Gamma(N+1)}\frac{\Gamma(\left[\frac{a+m+2}{2}\right])}{\Gamma(\left[\frac{a+m+2}{2}\right] -N + q)}\widetilde{\alpha}(a+2, m),
\end{align*}
where
\begin{equation*}
\widetilde{\alpha}(a, m) := \sum_{r=0}^{\min(q,N)}\begin{pmatrix}
N\\
r
\end{pmatrix} 
\frac{\Gamma(2N-r+1)}{\Gamma(2N-q+1)}\frac{\Gamma(q+1)}{\Gamma(q-r+1)}\frac{\Gamma(N-m+1)}{\Gamma(N-m-r+1)}\frac{\Gamma(\left[\frac{a+m}{2}\right] -N + q)}{\Gamma(\left[\frac{a+m}{2}\right] -N + r)}.
\end{equation*}
Now, by dividing $P(\lambda)$ by
\begin{equation}\label{proof-constant-dividing}
\left(\frac{\Gamma(2N-q+1)}{\Gamma(N+1)}\right)^2 \frac{\Gamma(\left[\frac{a-m+2}{2}\right])}{\Gamma(\left[\frac{a-m}{2}\right] -N + q)}\frac{\Gamma(\left[\frac{a+m+2}{2}\right])}{\Gamma(\left[\frac{a+m}{2}\right] -N + q)} \neq 0,
\end{equation}
we obtain that $P(\lambda) = 0$ if and only if 
\begin{equation}\label{proof-alpha-pm-equality}
\widetilde{\alpha}(a, m)\widetilde{\alpha}(a+2, -m) = \widetilde{\alpha}(a, -m)\widetilde{\alpha}(a+2, m).
\end{equation}
Let us show that we actually have
\begin{equation*}
\begin{aligned}
\widetilde{\alpha}(a, -m) &= \widetilde{\alpha}(a, m),\\
\widetilde{\alpha}(a+2, -m) &= \widetilde{\alpha}(a+2, m),
\end{aligned}
\end{equation*}
which in particular implies \eqref{proof-alpha-pm-equality}. Since $a$ is arbitrary, it suffices to show the first equality. By using the Pochhammer notation of the rising factorial (see \eqref{Pochhamer-symbol}), $\widetilde{\alpha}(a, \pm m)$ can be rewritten as
\begin{equation}\label{proof-alpha-hypergeometric}
\begin{aligned}
\widetilde{\alpha}(a, \pm m) &= (2N-q+1)_q \left(\left[\frac{a\pm m}{2}\right]-N\right)_q \sum_{r=0}^{\min(q,N)}\frac{(-q)_r(-N)_r( \pm m-N)_r}{(-2N)_r \left(\left[\frac{a\pm m}{2}\right]-N\right)_r r!}\\
& = (2N-q+1)_q \left(\left[\frac{a\pm m}{2}\right]-N\right)_q {}_3F_2\left(\begin{matrix}
-q, -N, \pm m -N\\
-2N, \left[\frac{a\pm m}{2}\right]-N
\end{matrix}\; ; 1\right),
\end{aligned}
\end{equation}
where ${}_3F_2$ denotes the hypergeometric series with $(p,q) = (3,2)$ (see \eqref{def-hypergeometric}). Thus, the equality $\widetilde{\alpha}(a, -m) = \widetilde{\alpha}(a, m)$ amounts to
\begin{equation*}
_3F_2\left(\begin{matrix}
-q, -N, -m -N\\
-2N, \left[\frac{a-m}{2}\right]-N
\end{matrix}\; ; 1\right) = \frac{\left(\left[\frac{a+ m}{2}\right]-N\right)_q}{\left(\left[\frac{a- m}{2}\right]-N\right)_q}{}_3F_2\left(\begin{matrix}
-q, -N, m -N\\
-2N, \left[\frac{a+ m}{2}\right]-N
\end{matrix}\; ; 1\right),
\end{equation*}
which holds now thanks to Kummer's identity \eqref{kummer}. Hence, \eqref{proof-alpha-pm-equality} holds; in other words, $P(\lambda) = 0$.
\end{proof}
\begin{rem} Note that in the proof above we used the condition $a \geq m+2N+2$ in order to guarantee that \eqref{proof-constant-dividing} holds. In fact, if $a \geq m+2N+2$, we have
\begin{equation*}
\left[\frac{a - m}{2}\right] -N +q \geq q+1 \geq 1.
\end{equation*}
Therefore
\begin{equation*}
\frac{\Gamma(\left[\frac{a\pm m+2}{2}\right])}{\Gamma(\left[\frac{a\pm m}{2}\right] -N + q)}  \neq 0.
\end{equation*}
Moreover, the condition $a \geq m+2N+2$ is necessary for expressing $\widetilde{\alpha}(a, -m)$ in terms of the hypergeometric function ${}_3F_2$ as demonstrated in \eqref{proof-alpha-hypergeometric}. Without this condition, the constant $\left(\left[\frac{a - m}{2}\right] -N\right)_r$ may vanish for some values of $r = 0, \ldots, \min(q,N)$, leading to a situation where ${}_3F_2$ would not be well-defined.
\end{rem}

Now, we are ready to show Proposition \ref{prop-phase3}.

\begin{proof}[Proof of Proposition \ref{prop-phase3}]
First, we note that the left-hand side of \eqref{identity-phase3} has degree $a-m \geq 0$, while the right-hand side has degree $a+m$. If we write everything in one side, we have
\begin{equation}\label{identity-phase3-onesided}
\begin{aligned}
\sum_{r=0}^{N}\left(\widetilde{q_N^+}C_{0,r}^+\Geg_{a+m-2N+2r}^{\lambda+N-1-r}(it) -\widetilde{q_N^-}C_{0,r}^-\Geg_{a-m-2N+2r}^{\lambda+N-1-r}(it)\right) \\
= \delta_{a+m} t^{a+m} + \delta_{a+m-2} t^{a+m-2} + \cdots + \delta_{\varepsilon+2} t^{\varepsilon+2} + \delta_\varepsilon t^\varepsilon, 
\end{aligned}
\end{equation}
for some coefficients $\delta_d \in \C$ that depend on $\lambda, N,  a, m$ and $\widetilde{q_N^\pm}$. Here, $\varepsilon$ denotes the following index that depends on the parity of $a+m$:
\begin{equation*}
\varepsilon := a+m-2\left[\frac{a+m}{2}\right]\in\{0, 1\}.
\end{equation*}
$(\Rightarrow)$ Suppose that \eqref{identity-phase3} holds. Let us show that both \eqref{condition-phase3-lambda} and \eqref{condition-phase3-qN+-} hold. We divide this direction of the proof into the following three cases:
\begin{enumerate}[label=\normalfont{(\arabic*)}]
\item $a \geq m+2N+2$.
\item $m+N \leq a \leq m+2N+1$.
\item $m \leq a < m+N$.
\end{enumerate}
(1) Suppose $a \geq m+2N+2$. If \eqref{identity-phase3} holds, then, in particular $\delta_\varepsilon = \delta_{\varepsilon+2} = 0$; in other words, the first and second terms in \eqref{identity-phase3-onesided} must vanish. Note that since $a \geq m+2N+2$, then all Gegenbauer polynomials $\Geg_{a\pm m-2N+2r}^{\lambda+N-1-r}(it)$ appearing in both sides of \eqref{identity-phase3} have degree greater than or equal to $2$. Therefore, they all have non-zero first and second terms (those associated to $t^\varepsilon$ and $t^{\varepsilon+2}$, respectively). By using this fact, and by the definition of $\Geg_{\ell}^\mu$ (see \eqref{Gegenbauer-polynomial(renormalized)}), one can compute explicitly $\delta_\varepsilon$ and $\delta_{\varepsilon+2}$ and obtain:
\begin{align}\label{proof-delta-epsilon-1}
\delta_\varepsilon \fallingdotseq &\; (-1)^m\beta(\lambda, N, a, m)\widetilde{q_N^+} - \frac{\Gamma(\left[\frac{a-m+2}{2}\right] + m)}{\Gamma(\left[\frac{a-m+2}{2}\right])} \beta(\lambda, N, a, -m)\widetilde{q_N^-},\\
\label{proof-delta-epsilon-2}
\begin{split}
\delta_{\varepsilon+2} \fallingdotseq &\;(-1)^m\left(\lambda + \left[\frac{a+m-1}{2}\right]\right)\alpha(\lambda, N, a, m)\widetilde{q_N^+}\\
& -\left(\lambda + \left[\frac{a-m-1}{2}\right]\right)\frac{\Gamma(\left[\frac{a-m}{2}\right] + m)}{\Gamma(\left[\frac{a-m}{2}\right])} \alpha(\lambda, N, a, -m)\widetilde{q_N^-},
\end{split}
\end{align}
where $\alpha$ and $\beta$ are defined as in \eqref{def-alpha} and \eqref{def-beta}, and by the symbol $\fallingdotseq$ we mean equal up to multiplication by a non-zero constant. The condition $\delta_\varepsilon = \delta_{\varepsilon+2} = 0$ can be seen as a homogeneous system of two equations in the variables $\widetilde{q_N^\pm}$ of the form:
\begin{equation*}
\begin{cases}
A_1 \widetilde{q_N^+} + B_1 \widetilde{q_N^-} = 0,\\
A_2 \widetilde{q_N^+} + B_2 \widetilde{q_N^-} = 0.
\end{cases}
\end{equation*}
The necessary and sufficient condition for this homogeneous system to have a non-zero solution $(\widetilde{q_N^+}, \widetilde{q_N^-})$, is that the determinant of the matrix defining the system must be zero. In other words,
\begin{equation}\label{proof-determinant}
\det \begin{pmatrix}
A_1 & B_1\\
A_2 & B_2
\end{pmatrix} = A_1B_2-A_2B_1 = 0.
\end{equation}
By using the expressions of $\delta_\varepsilon$ and $\delta_{\varepsilon+2}$ given in \eqref{proof-delta-epsilon-1} and \eqref{proof-delta-epsilon-2} above, one can easily show that \eqref{proof-determinant} is equivalent to $P(\lambda) = 0$, being $P$ the polynomial defined in \eqref{def-P(lambda)}. Now, by Lemma \ref{lemma-P=Q} we deduce that $P(\lambda) = 0$ implies $\lambda \in \Lambda(N,a)$. Thus, \eqref{condition-phase3-lambda} holds.

In order to show \eqref{condition-phase3-qN+-}, it suffices to simplify $\delta_\varepsilon = 0$ by using that $\lambda \in \Lambda(N,a)$. From the proof of Lemma \ref{lemma-P=Q}, we deduce that, for $\lambda = N+1-a-q$ with $q = 0, 1, \ldots, 2N$, we have:
\begin{equation}\label{proof-beta-relation}
\beta(\lambda, N, a, -m)= \frac{\Gamma(\left[\frac{a-m+2}{2}\right])}{\Gamma(\left[\frac{a+m+2}{2}\right])}\frac{\Gamma(\left[\frac{a+m+2}{2}\right]-N+q)}{\Gamma(\left[\frac{a-m+2}{2}\right] -N+q)}  \beta(\lambda, N, a, m).
\end{equation}
Now, by dividing the right-hand side of \eqref{proof-delta-epsilon-1} by $(-1)^m\beta(\lambda, N, a, m) \neq 0$, we obtain
\begin{equation*}
\widetilde{q_N^+} - (-1)^m \frac{\Gamma(\left[\frac{a+m+2}{2}\right]-N+q)}{\Gamma(\left[\frac{a-m+2}{2}\right] -N+q)} \widetilde{q_N^-} = 0,
\end{equation*}
which clearly coincides with \eqref{condition-phase3-qN+-} by the identity below:
\begin{equation*}
\left[\frac{a\pm m+2}{2}\right] -N+q = -\left(\lambda+ \left[\frac{a\mp m-3}{2}\right]\right).
\end{equation*}

The fact that $\beta(\lambda, N, a, m)$ is non-zero is not trivial, but can be shown in the following way. By \eqref{def-beta}, we have
\begin{equation*}
\beta(\lambda, N, a, m) = (N+1)_N\left(\left[\frac{a+m+2}{2}\right]-N\right)_N{}_3F_2\left(\begin{matrix}
-q, -N, m-N\\ -2N, \left[\frac{a+m+2}{2}\right]-N
\end{matrix}; 1\right).
\end{equation*}
Since $a \geq m$, it suffices to show that the ${}_3F_2$-term is non-zero. This term is an alternating sum, so it is not clear whether is zero or not. However, by using \eqref{aar-identity}, we can rewrite it as a product as follows:
\begin{align*}
{}_3F_2\left(\begin{matrix}
-q, m-N, -N\\ \left[\frac{a+m+2}{2}\right]-N, -2N
\end{matrix}; 1\right) = &\frac{\Gamma(\left[\frac{a+m+2}{2}\right]-N)\Gamma(\left[\frac{a+m+2}{2}\right]-2N+q+m)}{\Gamma(\left[\frac{a+m+2}{2}\right]-N+q)\Gamma(\left[\frac{a+m+2}{2}\right]-m)}\\
&\times {}_3F_2\left(\begin{matrix}
-q, m-N, -N\\ -\left[\frac{a+m+2}{2}\right]+m-q+1
\end{matrix}, -2N; 1\right).
\end{align*}
Now, since $a \geq m$ and $m>N$, the gamma factor on the right-hand side is strictly positive. By the same reason, the ${}_3F_2$-part in the right-hand side is a sum of positive terms. Therefore $\beta(\lambda, N, a, m) > 0$.

(2) Now, suppose $m+N \leq a \leq m+2N+1$. If \eqref{identity-phase3} is satisfied, then clearly $\delta_{a-m+2} = 0$. Since the left-hand side of \eqref{identity-phase3} has degree $a-m$, this term comes only from the right-hand side. A straightforward computation yields to
\begin{equation*}
\begin{aligned}
\delta_{a-m+2} = & \frac{(-1)^{m-N-1}(2i)^{a-m+2}}{(m-N-1)!(a-m+2)!}\frac{\Gamma(\lambda+a)}{\Gamma(\lambda+\left[\frac{a+m-1}{2}\right])}\\
& \times \sum_{r=0}^{N}\begin{pmatrix}
N\\
r
\end{pmatrix}
\frac{\Gamma(2N-r+1)}{\Gamma(N+1)}\frac{\Gamma(\lambda+a-N+r-1)}{\Gamma(\lambda+a-N-1)}.
\end{aligned}
\end{equation*}
The sum part can be easily rewritten as
\begin{equation*}
\begin{aligned}
(N+1)_N\sum_{r=0}^{N}\frac{(-N)_r(\lambda+a-N-1)_r}{(-2N)_r r!} & = (N+1)_N{}_2F_1\left(\begin{matrix}
-N, \lambda+a-N-1\\
-2N
\end{matrix}; \; 1 \right)\\
& = (\lambda+a)_N.
\end{aligned}
\end{equation*}
where in the second equality we used the Chu--Vandermonte identity \eqref{Chu-Vandermonde}.
Therefore,
\begin{equation*}
\delta_{a-m+2} \fallingdotseq \frac{\Gamma(\lambda+a+N)}{\Gamma(\lambda+\left[\frac{a+m-1}{2}\right])}.
\end{equation*}
Thus, $\delta_{a-m+2} = 0$ implies $\lambda \in \Lambda_1 := \{1-N-a, 2-N-a, \ldots, \left[\frac{a-m+2N+2}{2}\right] -N-a\}$.

On the other hand, if \eqref{identity-phase3} is satisfied, clearly $\delta_{a+m} = 0$. In other words, one of the following two conditions is satisfied:
\begin{enumerate}
\item[(A)] The term of degree $a+m$ in $\Geg_{a+m}^{\lambda-1}(it)$ vanishes.
\item[(B)] $C_{0, N}^+ = 0$.
\end{enumerate}
(A) is equivalent to $\lambda-1 \in M_{a+m}^0$ by Lemma \ref{lemma-Gegenbauer-Mlk-condition}; in other words, to
\begin{equation*}
\frac{\Gamma(\lambda+a+m-1)}{\Gamma(\lambda+\left[\frac{a+m-1}{2}\right])} = 0.
\end{equation*}
However, this gamma factor vanishes for any $\lambda \in \Lambda_1$. Therefore, (A) holds for $\lambda\in \Lambda_1$. On the other hand, (B) amounts to
\begin{equation*}
C_{0, N}^+ = \frac{\Gamma(N-m+1)}{\Gamma(1-m)}\frac{\Gamma(\lambda+a-1)}{\Gamma(\lambda+a-N-1)},
\end{equation*}
which implies $\lambda \in \Lambda_2 := \{-a, 1-a, \ldots, N-a, N+1-a\}$. In conclusion, we deduce that if \eqref{identity-phase3} is satisfied, then necessarily $\lambda \in \Lambda_1 \cup \Lambda_2$. However, since $m+N \leq a \leq m+2N+1$, we have
\begin{equation*}
\begin{aligned}
1 \leq 1+ \left[\frac{N}{2}\right] \leq \left[\frac{a-m+2}{2}\right] \leq \left[\frac{2N+3}{2}\right] = N+1.
\end{aligned}
\end{equation*}
Therefore $\Lambda_1 \cup \Lambda_2 = \Lambda(N,a)$, which proves \eqref{condition-phase3-lambda}. 

To prove identity \eqref{condition-phase3-qN+-}, we cannot directly apply the argument used for $a \geq m+2N+2$, but a similar approach is still possible. We focus on the vanishing condition of the first terms (i.e., $\delta_\varepsilon = 0$), and derive \eqref{condition-phase3-qN+-} from there.
As before, we can express $\beta(\lambda, N, a, m)$ by using the function ${}_3F_2$ as follows:
\begin{equation*}
\beta(\lambda, N, a, m) = (N+1)_N\left(\left[\frac{a+m+2}{2}\right] -N\right)_N  {}_3F_2\left(\begin{matrix}
-q, -N, m -N\\
-2N, \left[\frac{a+m+2}{2}\right]-N
\end{matrix}\; ; 1\right).
\end{equation*}
It is important to note that this approach cannot be applied to $\beta(\lambda, N, a, -m)$, as the coefficient $\left(\left[\frac{a-m-2}{2}\right] -N\right)_r$ may vanish when $m+N \leq a < m+2N+2$. Nevertheless, even though we cannot rewrite $\beta(\lambda, N, a, -m)$ using the function ${}_3F_2$, the relation \eqref{proof-beta-relation} still holds, as we prove below.

For sufficiently small $\eta > 0$, define $\beta_\eta(\lambda, N, a, m)$ as follows:
\begin{equation*}
\beta_\eta(\lambda, N, a, m) := \sum_{r = 0}^{N}(-1)^r C_{0,r}^+\frac{\Gamma\left(\left[\frac{a+m+2}{2}\right]\right)}{\Gamma\left(\left[\frac{a+m+2}{2}\right]+\eta-N+r\right)}.
\end{equation*}
Note that if $0<\eta<1$, $\beta_\eta(\lambda, N, a, \pm m)$ can be expressed by using the function ${}_3F_2$ as:
\begin{equation*}
\beta_\eta(\lambda, N, a, \pm m) = (N+1)_N\left(\left[\frac{a\pm m+2}{2}\right] -N + \eta \right)_N  {}_3F_2\left(\begin{matrix}
-q, -N, \pm m -N\\
-2N, \left[\frac{a\pm m+2}{2}\right]-N +\eta
\end{matrix}\; ; 1\right).
\end{equation*}
Now, by using Kummer's identity \eqref{kummer}, we have
\begin{equation*}
_3F_2\left(\begin{matrix}
-q, -N, -m -N\\
-2N, \left[\frac{a-m+2}{2}\right]-N + \eta
\end{matrix}\; ; 1\right) = \frac{\left(\left[\frac{a+ m+2}{2}\right]-N +\eta\right)_q}{\left(\left[\frac{a- m+2}{2}\right]-N +\eta\right)_q}{}_3F_2\left(\begin{matrix}
-q, -N, m -N\\
-2N, \left[\frac{a+ m+2}{2}\right]-N+\eta
\end{matrix}\; ; 1\right),
\end{equation*}
which implies
\begin{equation*}
\beta_\eta(\lambda, N, a, -m) = \frac{\Gamma\left(\left[\frac{a-m+2}{2}\right] + \eta\right)\Gamma\left(\left[\frac{a+m+2}{2}\right] -N+q + \eta\right)}{\Gamma\left(\left[\frac{a-m+2}{2}\right] -N+q + \eta\right)\Gamma\left(\left[\frac{a+m+2}{2}\right] +\eta\right)}\beta_\eta(\lambda, N, a, m).
\end{equation*}
Now, by taking the limit when $\eta \rightarrow 0$, we obtain precisely \eqref{proof-beta-relation}. Therefore, we can divide by $(-1)^m\beta(\lambda, a, N, m) \neq 0$ in both sides of \eqref{proof-delta-epsilon-1} and obtain
\begin{equation*}
\widetilde{q_N^+} - (-1)^m \frac{\Gamma(\left[\frac{a+m+2}{2}\right]-N+q)}{\Gamma(\left[\frac{a-m+2}{2}\right] -N+q)} \widetilde{q_N^-} = 0,
\end{equation*}
which is nothing but \eqref{condition-phase3-qN+-}.

(3) Suppose now $m \leq a < m+N$. As we deduced in the previous case, if \eqref{identity-phase3} is satisfied, then $\delta_{a-m+2} = 0$, which implies $\lambda \in \Lambda_1 = \{1-N-a, 2-N-a, \ldots, \left[\frac{a-m+2}{2}\right]-a\}$. But since $m \leq a < m+N$, then $\left[\frac{a-m+2}{2}\right]-a = - \left[\frac{a+m-1}{2}\right] \leq - \left[\frac{2m-1}{2}\right] = 1-m$. Therefore, $\Lambda_1 \subset \Lambda(N,a)$ and \eqref{condition-phase3-lambda} is satisfied. As for \eqref{condition-phase3-qN+-}, the same argument used in case (2) works.\medskip

$(\Leftarrow)$ Let us show now that identity \eqref{identity-phase3} is satisfied assuming that \eqref{condition-phase3-lambda} and \eqref{condition-phase3-qN+-} hold. We do it in two steps. First, we make use of Lemma \ref{lemma-KOSS} and rewrite the right-hand side of \eqref{identity-phase3}. Then, we simplify both sides by using properties of the hypergeometric function ${}_3F_2$ to show that they actually coincide. From \eqref{condition-phase3-qN+-}, we have that \eqref{identity-phase3} amounts to:
\begin{equation}\label{proof-phase3-identity}
\widetilde{q_N^-}\sum_{r=0}^{N}C_{0,r}^-\Geg_{a-m-2N+2r}^{\lambda+N-1-r}(it) = \widetilde{q_N^-}\sum_{r=0}^{N}\frac{\Gamma\left(\lambda+\left[\frac{a+m-1}{2}\right]\right)}{\Gamma\left(\lambda+\left[\frac{a-m-1}{2}\right]\right)}C_{0,r}^+\Geg_{a+m-2N+2r}^{\lambda+N-1-r}(it).
\end{equation}
Let us show that the right-hand side can be rewritten as (omitting the constant $\widetilde{q_N^-}$):
\begin{equation}\label{proof-phase3-KOSS}
\sum_{r=0}^{N}\frac{\Gamma\left(\lambda+\left[\frac{a-m-1}{2}\right] + q-r \right)}{\Gamma\left(\lambda+\left[\frac{a-m-1}{2}\right]\right)}C_{0,r}^+\Geg_{a-m-2N+2q}^{\lambda+N-1-r}(it).
\end{equation}
For $\lambda = N+1-a-q \in \Lambda(N,a)$ and $r = 0, \ldots, N$, we set $b := -\lambda-N+1+r = a+r+q-2N$ and $\ell := a+m-2N+2r$. From $a \geq m$ and $q \geq \max(0, N-a+m)$, we clearly have $b, \ell \in \N$. Now, if $2b-\ell = a-m-2N+2q \geq 0$, then the right-hand side of \eqref{proof-phase3-identity} amounts to \eqref{proof-phase3-KOSS} due to Lemma \ref{lemma-KOSS}. On the other hand, if $2b-\ell = a-m-2N+2q <0$, then \eqref{proof-phase3-KOSS} vanishes since all Gegenbauer polynomials in the sum have negative degrees. Similarly, the right-hand side of \eqref{proof-phase3-identity} is also zero. In fact, as $m \leq a < m+2N-2q$, we have
\begin{align*}
\lambda + \left[\frac{a+m-1}{2}\right] &= N+1-a-q + \left[\frac{a+m-1}{2}\right] = N-q - \left[\frac{a-m}{2}\right] > 0,\\[4pt]
\lambda + \left[\frac{a-m-1}{2}\right] &= N-q - \left[\frac{a+m}{2}\right] \leq N-q-m < 0,
\end{align*}
which implies
\begin{equation*}
\frac{\Gamma\left(\lambda+\left[\frac{a+m-1}{2}\right]\right)}{\Gamma\left(\lambda+\left[\frac{a-m-1}{2}\right]\right)} = 0.
\end{equation*}
Hence, in order to show \eqref{proof-phase3-identity} it suffices to prove
\begin{equation}\label{proof-phase3-after-KOSS}
\sum_{r=0}^{N}C_{0,r}^-\Geg_{a-m-2N+2r}^{\lambda+N-1-r}(it) = \sum_{r=0}^{N}\frac{\Gamma\left(\lambda+\left[\frac{a-m-1}{2}\right] + q-r \right)}{\Gamma\left(\lambda+\left[\frac{a-m-1}{2}\right]\right)}C_{0,r}^+\Geg_{a-m-2N+2q}^{\lambda+N-1-r}(it).
\end{equation}
By applying \eqref{TTR-generalized}, the right-hand side of \eqref{proof-phase3-after-KOSS} above can be rewritten as
\begin{align*}
\sum_{r=0}^N\sum_{s=0}^{q-r}
\begin{pmatrix}
q-r\\
s
\end{pmatrix}
\frac{\Gamma(s-m)}{\Gamma(-m)}C_{0,r}^+\Geg_{a-m-2N+2s+2r}^{\lambda+N-1-r-s}(it)\\
= \sum_{r=0}^N\sum_{d=r}^q
\begin{pmatrix}
q-r\\
d-r
\end{pmatrix}
\frac{\Gamma(d-r-m)}{\Gamma(-m)}C_{0,r}^+\Geg_{a-m-2N+2d}^{\lambda+N-1-d}(it),
\end{align*}
which can be expressed as a sum of the type $S_1 + S_2$, where $S_1$ and $S_2$ are, respectively
\begin{align*}
S_1 &:= \sum_{r=0}^{N-1}\sum_{d=r}^{N-1}
\begin{pmatrix}
q-r\\
d-r
\end{pmatrix}
\frac{\Gamma(d-r-m)}{\Gamma(-m)}C_{0,r}^+\Geg_{a-m-2N+2d}^{\lambda+N-1-d}(it),\\
S_2 &:= \sum_{r=0}^N\sum_{d=N}^q
\begin{pmatrix}
q-r\\
d-r
\end{pmatrix}
\frac{\Gamma(d-r-m)}{\Gamma(-m)}C_{0,r}^+\Geg_{a-m-2N+2d}^{\lambda+N-1-d}(it).
\end{align*}
Then, in order to prove \eqref{proof-phase3-after-KOSS} it suffices to show that the following identities hold:
\begin{align}
\label{proof-S1}
S_1 & =\sum_{d=0}^{N-1}C_{0,d}^-\Geg_{a-m-2N+2d}^{\lambda+N-1-d}(it),\\
\label{proof-S2}
S_2 &= C_{0,N}^-\Geg_{a-m}^{\lambda-1}(it),
\end{align}
as the left-hand side of \eqref{proof-phase3-after-KOSS} is
\begin{align*}
\sum_{r=0}^{N}C_{0,r}^-\Geg_{a-m-2N+2r}^{\lambda+N-1-r}(it) = \sum_{d=0}^{N-1}C_{0,d}^-\Geg_{a-m-2N+2d}^{\lambda+N-1-d}(it) + C_{0,N}^-\Geg_{a-m}^{\lambda-1}(it).
\end{align*}
We start by showing \eqref{proof-S1}. By changing the order of the indices in the sum $S_1$, we have 
\begin{align*}
S_1 = \sum_{d=0}^{N-1}\Big(\sum_{r=0}^{d}
\begin{pmatrix}
q-r\\
d-r
\end{pmatrix}
\frac{\Gamma(d-r-m)}{\Gamma(-m)}C_{0,r}^+\Big)\Geg_{a-m-2N+2d}^{\lambda+N-1-d}(it).
\end{align*}
Therefore, \eqref{proof-S1} is satisfied if and only if
\begin{equation}\label{proof-S1-last}
C_{0,d}^- = \sum_{r=0}^{d}
\begin{pmatrix}
q-r\\
d-r
\end{pmatrix}
\frac{\Gamma(d-r-m)}{\Gamma(-m)}C_{0,r}^+.
\end{equation}
By using the definition of $C_{0,r}^\pm$ given in \eqref{definition-constants-Cjr+-}, we have
\begin{align*}
C_{0, d}^- &= (N+1)_N \frac{(-q)_d(-N)_d(-m-N)_d(-1)^d}{(-2N)_d \cdot d!},\\
\begin{pmatrix}
q-r\\
d-r
\end{pmatrix}
\frac{\Gamma(d-r-m)}{\Gamma(-m)}C_{0,r}^+ &= 
(N+1)_N (-m)_d (-q)_d (-1)^d \frac{(-d)_r(-N)_r (m-N)_r}{(-2N)_r (1-d+m)_r \cdot r! d!}.
\end{align*}
Thus, \eqref{proof-S1-last} is satisfied if and only if
\begin{align*}
{}_3F_2\left(
\begin{matrix}
-d, -N, m-N\\
-2N, 1-d+m
\end{matrix}
;\; 1\right) = \frac{(-N)_d(-m-N)_d}{(-m)_d(-2N)_d},
\end{align*}
which holds thanks to the Pfaff--Saalschütz identity \eqref{pfaff-saalschutz}. Hence \eqref{proof-S1} is satisfied. Next, we show \eqref{proof-S2}. By a similar argument as the one used above, one can rewrite $S_2$ by using \eqref{pfaff-saalschutz} as
\begin{align*}
S_2 = \sum_{d=N}^q (N+1)_N \frac{(-q)_d(-N)_d(-m-N)_d(-1)^d}{(-2N)_d \cdot d!}\Geg_{a-m-2N+2d}^{\lambda-N-1-d}(it).
\end{align*}
If $0 \leq q \leq N-1$, \eqref{proof-S2} is trivially satisfied since both sides are zero. Suppose then $N \leq q \leq 2N$.
If $d \geq N+1$, then $d-N-1 \geq 0$. Therefore, $(-N)_d = (-N) (1-N) \cdots (d-N-1) = 0$. In other words, the sum $S_2$ only has one term, namely, the one for $d = N$. Thus,
\begin{align*}
S_2 = (N+1)_N\frac{(-q)_N(-m-N)_N}{(-2N)_N}\Geg_{a-m}^{\lambda-1}(it) = C_{0,N}^-\Geg_{a-m}^{\lambda-1}(it).
\end{align*}
Hence, both \eqref{proof-S2} and \eqref{proof-S1} hold, which implies that \eqref{identity-phase3} is satisfied.
\end{proof}

From Proposition \ref{prop-phase3} above (together with Proposition \ref{prop-phase2} of Phase 2) we deduce that when $a \geq m$, a solution of the system $(f_{-N},\ldots, f_N)$ must be unique up to constant, and that it must satisfy $\lambda \in \Lambda(N,a)$; conditions that we obtained from imposing $f_{-0} = f_{+0}$. However, when $m-N \leq a < m$ we must adopt another approach as $f_{\pm 0} \equiv 0$ automatically by \eqref{condition-space-f+-j}. Nevertheless, in this case we also obtain the condition $\lambda \in \Lambda(N,a)$, as the following result shows.

\begin{lemma}\label{lemma-case-m-N<=a<m} Suppose $m-N \leq a < m$. Then, a non-zero solution of the system $(f_{-N}, \ldots, f_N)$ must satisfy $\lambda \in \Lambda(N,a)$.
\end{lemma}
\begin{proof}
First, observe that by \eqref{condition-space-f+-j} we have
\begin{align*}
f_{-j} \equiv 0 \text{ for all } j=0, \ldots, N, \quad \text{and} \quad f_{j} \equiv 0 \text{ for all } j = 0, \ldots, m-a-1.
\end{align*}
Thus, a solution has the form $(0, \ldots,0, f_{m-a}, \dots, f_N)$. By Proposition \ref{prop-phase2}, this solution must satisfy \eqref{expression-fj+-normalized}. Thus, by \eqref{condition-space-f+-j}, the terms of degree greater than $a-m+j$ appearing in the expression \eqref{expression-fj+-normalized} of $f_j$ must vanish. In particular, the term $t^{a-m+j+2}$ must be zero. A direct computation by using the definition of the Gegenbauer polynomial shows that this term is equal to
\begin{align*}
\enspace \frac{\Gamma(N+j+1)}{\Gamma(N+1)}\frac{(-1)^{m-N-1}2^j(2i)^{a-m+2}}{(a-m+j+2)!(m-N-1)!}\frac{\Gamma\left(\lambda+a+j\right)}{\Gamma\left(\lambda+\left[\frac{a+m-1}{2}\right]\right)}\\
\times (N+j+1)_{N-j} \sum_{r=0}^{N-j}\frac{(j-N)_r(\lambda+a-N-1)_r}{(-2N)_r r!}. 
\end{align*}
If we divide by the non-zero constants, and apply the Chu--Vandermonde identity \eqref{Chu-Vandermonde} in the sum part, this expression amounts to
\begin{equation*}
\frac{\Gamma\left(\lambda+a+j\right)}{\Gamma\left(\lambda+\left[\frac{a+m-1}{2}\right]\right)} \frac{\Gamma\left(\lambda+a+N\right)}{\Gamma\left(\lambda+a+j\right)} = \frac{\Gamma\left(\lambda+a+N\right)}{\Gamma\left(\lambda+\left[\frac{a+m-1}{2}\right]\right)}.
\end{equation*}
Therefore, the solution $(0, \ldots,0, f_{m-a}, \dots, f_N)$ must satisfy
\begin{equation}\label{proof-m-N<=a<m-gamma-factor}
 \frac{\Gamma\left(\lambda+a+N\right)}{\Gamma\left(\lambda+\left[\frac{a+m-1}{2}\right]\right)} = 0.
\end{equation}
On the other hand, we must have
\begin{equation*}
\prod_{s=m-a+1}^{N}\gamma(\lambda, a+m+N-2-s) = \frac{\Gamma\left(\lambda + m-1\right)}{\Gamma\left(\lambda + \left[\frac{a+m-1}{2}\right]\right)} \neq 0,
\end{equation*}
otherwise we would have a zero solution, by Lemma \ref{lemma-Gamma+zero}. Hence, by dividing by this gamma factor, \eqref{proof-m-N<=a<m-gamma-factor} amounts to
\begin{equation*}
 \frac{\Gamma\left(\lambda+a+N\right)}{\Gamma\left(\lambda + m-1\right)} = 0,
\end{equation*}
which is clearly equivalent to $\lambda \in \Lambda(N,a) = \{N+1-a-q: q = N-a+m, \ldots, 2N\}$.
\end{proof}

\begin{rem}\label{rem-lambda_Gamma0-}
From Proposition \ref{prop-phase3} and Lemma \ref{lemma-case-m-N<=a<m} above, we deduce that for any $a \geq m-N$ a non-zero solution of the system must satisfy $\lambda \in \Lambda(N,a)$. A direct computation shows that for $\lambda \in \Lambda(N,a)$ we have necessarily $\Gamma_0^- \neq 0$, which confirms what we already proved in Proposition \ref{prop-phase2}(2).
\end{rem}

As a consequence of Propositions \ref{prop-phase2} and \ref{prop-phase3}, Lemma \ref{lemma-case-m-N<=a<m} and condition \eqref{condition-space-f+-j}, we deduce that a solution of the system $(f_{-N}, \ldots, f_N)$ must be given by \eqref{expression-fj--normalized} for $f_{-j}$ $(j = 0, 1, \ldots, N)$ and by \eqref{expression-fj+-normalized-one-constant} below for $f_j$ $(j = 1, \ldots, N)$.
\begin{align}
\label{expression-fj+-normalized-one-constant}
f_j(t) &= 
\begin{cases}
0, &\text{if } m-N\leq a < m, \text{ and } j = 0, 1, \ldots, m-a-1,\\
\eqref{expression-fj+-normalized-one-constant-sub}, & \text{otherwise},
\end{cases}\\
\label{expression-fj+-normalized-one-constant-sub}
&(-i)^j\frac{\Gamma(N+j+1)}{\Gamma(N+1)}\frac{\Gamma\left(\lambda+\left[\frac{a+m+j-1}{2}\right]\right)}{\Gamma\left(\lambda+\left[\frac{a-m-1}{2}\right]\right)}\widetilde{q_N^-}\sum_{r=0}^{N-j}C_{j,r}^+\Geg_{a+m-2N+j+2r}^{\lambda+N-1-r}(it),
\end{align}
for any constant $\widetilde{q_N^-} \in \C$. In particular, the solution of the system is unique up to constant.

\subsection{Proof of Theorem \ref{Thm-solvingequations}}
\label{section-proof_of_solving_equations-subsection}
From Phases 1 to 3 we have that a solution of the system $(f_{-N}, \ldots, f_N)$ must be given by \eqref{expression-fj--normalized} and \eqref{expression-fj+-normalized-one-constant} for any constant $\widetilde{q_N^-} \in \C$. Moreover, from Proposition \ref{prop-phase3} and Lemma \ref{lemma-case-m-N<=a<m} we know that this solution exists only for $\lambda \in \Lambda(N,a)$. Before proving Theorem \ref{Thm-solvingequations}, we rewrite the expression \eqref{expression-fj+-normalized-one-constant} in a way that will allow us to show easily that \eqref{condition-space-f+-j} holds, and that the tuple $(f_{-N}, \ldots, f_N)$ is indeed non-zero.

\begin{lemma}\label{lemma-simplification-fj+} Let $a \geq m-N$ and $\lambda = N+1-a-q\in \Lambda(N,a)$. Then, the expression \eqref{expression-fj+-normalized-one-constant} for $f_j$ can be rewritten as follows:
\begin{align}\label{expression-fj+-normalized-final}
f_j(t) &= (-i)^j\frac{\Gamma(N+j+1)}{\Gamma(N+1)}\frac{\Gamma\left(\lambda+\left[\frac{a-m+j-1}{2}\right] +q-N\right)}{\Gamma\left(\lambda+\left[\frac{a-m-1}{2}\right]\right)}\widetilde{q_N^-}\sum_{r=0}^{N-j}D_{j,r}\Geg_{a-m+j+2(q-2N+r)}^{\lambda+N-1-r}(it),
\end{align}
where $D_{j,r}$ is defined as
\begin{equation}\label{def-constants-Djr}
D_{j, r} := \begin{pmatrix}
N-j\\
r
\end{pmatrix}
\frac{\Gamma(2N-r+1)}{\Gamma(N+j+1)}
\frac{\Gamma(N + m+1)}{\Gamma(N + m -r +1)}
\frac{\Gamma(q-2N+r)}{\Gamma(q-2N)}.
\end{equation}
\end{lemma}
\begin{rem} Note that $D_{j,r}$ above is very similar to $C_{j,r}^-$ defined in \eqref{definition-constants-Cjr+-}. They differ on the last gamma factor, which in the case of $D_{j,r}$ is equal to 
\begin{equation*}
\frac{\Gamma(1-N-a-\lambda+r)}{\Gamma(1-N-a-\lambda)}.
\end{equation*}
\end{rem}
\begin{proof}[Proof of Lemma \ref{lemma-simplification-fj+}]
First, we note that for any $j = 0,\ldots, N$, the expression \eqref{expression-fj+-normalized-final} has degree less than or equal to $a-m+j$. In fact, fix $q$ and think about the degree of the term $D_{j,r} \Geg_{a-m+j+(q-2N+r)}^{\lambda+N-1-r}$ for a given $r$. If $q-2N+r \leq 0$, then $a-m+j+2(q-2N+r) \leq a-m+j$, so the term has degree $\leq a-m+j$. On the contrary, if $q-2N+r > 0$, then $D_{j,r} = 0$ as $\frac{\Gamma(q-2N+r)}{\Gamma(q-2N)} = 0$. Thus, all terms in the sum have degree $\leq a-m+j$. In particular, this expression for $f_j$ satisfies \eqref{condition-space-f+-j}.

The argument above shows that, in particular, if $m-N \leq a < m$ and $0 \leq j <m-a$ then \eqref{expression-fj+-normalized-one-constant} and \eqref{expression-fj+-normalized-final} coincide. Suppose then that $a \geq m$, or that $j \geq a-m$ if $m-N \leq a < m$. In order to show that \eqref{expression-fj+-normalized-one-constant} and \eqref{expression-fj+-normalized-final} coincide, we argue as in the proof of Proposition \ref{prop-phase3}. First, notice that
\begin{multline}\label{proof-simplif-f_j-KOSS}
\sum_{r=0}^{N-j}\frac{\Gamma\left(\lambda+\left[\frac{a+m+j-1}{2}\right]\right)}{\Gamma\left(\lambda+\left[\frac{a-m-1}{2}\right]\right)}C_{j,r}^+\Geg_{a+m-2N+j+2r}^{\lambda+N-1-r}(it) \\= \sum_{r=0}^{N-j}\frac{\Gamma\left(\lambda+\left[\frac{a-m-j-1}{2}\right]+q-r\right)}{\Gamma\left(\lambda+\left[\frac{a-m-1}{2}\right]\right)}C_{j,r}^+\Geg_{a-m+2q-j-2N}^{\lambda+N-1-r}(it).
\end{multline}
In fact, set $\ell := a+m-2N+j+2r$ and $b := -\lambda-N+1+r = a+q-2N+r$. Since $q \geq \max(0, N-a+m)$ we have $\ell, b \in \N$. Now, if $2b-\ell = a-m+2q-j-2N \geq 0$, then \eqref{proof-simplif-f_j-KOSS} holds thanks to \eqref{identity-KOSS}. On the other hand, if $2b-\ell =a-m+2q-j-2N <0$ the right-hand side of \eqref{proof-simplif-f_j-KOSS} vanishes. But so does the left-hand side as
\begin{align*}
&\lambda + \left[\frac{a+m+j-1}{2}\right] = N+1-a-q + \left[\frac{a+m-1}{2}\right] = N-q - \left[\frac{a-m-j}{2}\right] > 0,\\[4pt]
&\lambda + \left[\frac{a-m-1}{2}\right] =
\begin{cases}
N-q - \left[\frac{a+m}{2}\right] \leq N-q-m < 0, &\text{ if } a \geq m,\\
N+1-a-q + \left[\frac{a-m-1}{2}\right] < -m < 0, &\text{ if } N-m \leq a < m.
\end{cases}
\end{align*}
Thus, \eqref{proof-simplif-f_j-KOSS} holds. Now, by using \eqref{TTR-generalized} the right-hand side can be rewritten as
\begin{multline*}
\frac{\Gamma(\lambda+\left[\frac{a-m+j-1}{2}\right] + q-N)}{\Gamma(\lambda + \left[\frac{a-m-1}{2}\right])}\\
\times \sum_{r=0}^{N-j}C_{j,r}^+ \sum_{s=0}^{N-j-r}\begin{pmatrix}
N-j-r\\s
\end{pmatrix}
\frac{\Gamma(q-N-m+s)}{\Gamma(q-N-m)}\Geg_{a-m+j+2(q-2N+s+r)}^{\lambda+N-1-r-s}(it).
\end{multline*}
By a change of indices, the sum part can be written as
\begin{align*}
&\sum_{r=0}^{N-j}\sum_{d=r}^{N-j}C_{j,r}^+\begin{pmatrix}
N-j-r\\d-r
\end{pmatrix}
\frac{\Gamma(q-N-m+d-r)}{\Gamma(q-N-m)}\Geg_{a-m+j+2(q-2N+d)}^{\lambda+N-1-d}(it)\\
=& \sum_{d=0}^{N-j}\left(\sum_{r=0}^{d}C_{j,r}^+\begin{pmatrix}
N-j-r\\d-r
\end{pmatrix}
\frac{\Gamma(q-N-m+d-r)}{\Gamma(q-N-m)}\right)\Geg_{a-m+j+2(q-2N+d)}^{\lambda+N-1-d}(it).
\end{align*}
Thus, in order to show that \eqref{expression-fj+-normalized-one-constant} and \eqref{expression-fj+-normalized-final} coincide, it suffices to prove
\begin{equation*}
\sum_{r=0}^{d}C_{j,r}^+\begin{pmatrix}
N-j-r\\d-r
\end{pmatrix}
\frac{\Gamma(q-N-m+d-r)}{\Gamma(q-N-m)} = D_{j,d}
\end{equation*}
By the definition of $C_{j,r}^+$, the left-hand side is equal to
\begin{align*}
\frac{(N+j+1)_{N-j}(q-N-m)_d(N-j-d+1)_d}{d!}{}_3F_2\left(
\begin{matrix}
m-N, -q, -d\\
-2N, 1-q+N+m-d
\end{matrix}\;;1
\right),
\end{align*}
which by the Pfaff--Saalschütz identity \eqref{pfaff-saalschutz} amounts to
\begin{align*}
&\frac{(N+j+1)_{N-j}(N-j-d+1)_d(-N-m)_d(q-2N)_d}{(-2N)_d\cdot d!}\\
= & \;\frac{(N-j)!(N+j+1)_{N-j-d}(N+m-d+1)_{d}(q-2N)_d}{d!(N-j-d)!}\\
= &\;  D_{j,d}.
\end{align*}
The lemma is now proved.
\end{proof}

\begin{rem}\label{rem-gamma-factor-qjN} Note that the gamma factor
\begin{equation}\label{gamma-factor-qjN}
\frac{\Gamma\left(\lambda+\left[\frac{a-m+j-1}{2}\right] +q-N\right)}{\Gamma\left(\lambda+\left[\frac{a-m-1}{2}\right]\right)}
\end{equation}
appearing in \eqref{expression-fj+-normalized-final} is a well-defined non-zero constant for any $a \geq m-N$, any $j = 0, \ldots, N$ and any $\lambda = N+1-a-q \in \Lambda(N,a)$. In fact, if $j-2N+2q = 0$, \eqref{gamma-factor-qjN} equals to $1$. If $j-2N+2q \geq 1$, then \eqref{gamma-factor-qjN} amounts to
\begin{equation*}
\frac{\Gamma\left(\lambda+\left[\frac{a-m-1+j-2N+2q}{2}\right] \right)}{\Gamma\left(\lambda+\left[\frac{a-m-1}{2}\right]\right)} = 
\prod_{d=1}^{j-2N+2q}\gamma(\lambda, a-m-2+j-2N+2q-d).
\end{equation*}
Since for $\lambda = N+1-a-q \in \Lambda(N,a)$ and for $d\geq 1$ we have
\begin{equation*}
\gamma(\lambda, a-m-2+j-2N+2q-d) = \begin{cases}
1, & \text{ if } a-m+j-d \in 2\N+1,\\
\frac{-a-m+j-d}{2} \leq \frac{-a-m+N-1}{2} < 0, & \text{ if } a-m+j-d \in 2\N,
\end{cases}
\end{equation*}
we deduce that \eqref{gamma-factor-qjN} cannot vanish.
Similarly, if $j-2N+2q \leq -1$ then \eqref{gamma-factor-qjN} amounts to
\begin{equation*}
\left(\frac{\Gamma\left(\lambda+\left[\frac{a-m-1}{2}\right]\right)}{\Gamma\left(\lambda+\left[\frac{a-m-1+j-2N+2q}{2}\right] \right)}\right)^{-1} = 
\prod_{d=1}^{j-2N+2q}\gamma(\lambda, a-m-2-d)^{-1}.
\end{equation*}
However, since for $\lambda = N+1-a-q \in \Lambda(N,a)$ and for $d\geq 1$ we have
\begin{equation*}
\gamma(\lambda, a-m-2-d) = \begin{cases}
1, &\text{ if } a-m-d \in 2\N+1,\\
\frac{-a-2q-m+2N-d}{2} \leq \frac{-2m+N-1}{2} <0, &\text{ if } a-m-d \in 2\N,
\end{cases}
\end{equation*}
then \eqref{gamma-factor-qjN} is well-defined, and it is clearly non-zero.
\end{rem}

A consequence of the previous lemma is the following.

\begin{prop}\label{prop-solution-nonzero} For any $a \geq m-N$ and any $\lambda \in \Lambda(N,a)$, the tuple $(f_{-N}, \ldots, f_N)$ given by \eqref{expression-fj--normalized} and \eqref{expression-fj+-normalized-final} satisfies the condition \eqref{condition-space-f+-j} and is non-zero.
\end{prop}
\begin{proof}
By the proof of Lemma \ref{lemma-simplification-fj+}, we deduce that $f_j$ given by \eqref{expression-fj+-normalized-final} satisfies $f_j \in \Pol_{a-m+j}[t]_\text{even}$ for any $j = 0, \ldots, N$. On the other hand, by \eqref{expression-fj--normalized}, it is clear that $f_{-j} \in \Pol_{a-m-j}[t]_\text{even}$ since each term in the sum has degree $a-m+j+2r-2N \leq a-m-j$. 

Now, in order to show that the tuple $(f_{-N}, \ldots, f_N)$ is non-zero, we distinguish between the following three cases:
\begin{enumerate}[label=\normalfont{(\arabic*)}]
\item $a \geq m+N$.
\item $m \leq a < m+N$.
\item $m-N \leq a < m$.
\end{enumerate}
The goal is to prove that we can always find a non-zero function $f_{\pm j}$. The value of $j$ will depend on each case.

(1) Suppose $a\geq m+N$. By \eqref{expression-fj--normalized} we have that $f_{-N}$ is given by
\begin{equation*}
f_{-N}(t) = i^N\frac{\Gamma(2N+1)}{\Gamma(N+1)}\Gamma_0^-\Geg_{a-m-N}^{\lambda-1}(it).
\end{equation*}
A direct computation shows that for any $\lambda \in \Lambda(N,a)$, we have $\Gamma_0^- \neq 0$. Moreover, since $a-m-N \geq 0$, the Gegenbauer polynomial above is non-zero since the constant term is a non-zero constant. Thus, $f_{-N}$ is non-zero.

(2) If $m \leq a < m+N$, we have that $f_{m-a} = f_{-(a-m)}$ is non-zero. In fact, by \eqref{expression-fj--normalized} we have
\begin{align*}
f_{m-a}(t) &= i^{a-m}\frac{\Gamma(N+a-m+1)}{\Gamma(N+1)}\frac{\Gamma\left(\lambda+a-m-1\right)}{\Gamma\left(\lambda+\left[\frac{a-m-1}{2}\right]\right)}C_{a-m, N-a+m}^-.
\end{align*}
The second gamma factor does not vanish since $\Gamma_0^- \neq 0$. Moreover, by \eqref{definition-constants-Cjr+-}
\begin{equation*}
C_{a-m, N-a+m}^- = \frac{\Gamma(N+m+1)}{\Gamma(a+1)}\frac{\Gamma(N-a+m-q)}{\Gamma(-q)},
\end{equation*}
which is clearly non-zero for any $\lambda = N+1-a-q \in \Lambda(N,a)$. Hence $f_{m-a} \not\equiv 0$.

(3) Similarly, if $m-N \leq a < m$ we have that $f_{m-a}$ does not vanish. In fact, by \eqref{expression-fj+-normalized-final} we have
\begin{equation*}
f_{m-a}(t) = (-i)^{m-a}\frac{\Gamma(N+m-a+1)}{\Gamma(N+1)}\frac{\Gamma\left(\lambda+q-N-1\right)}{\Gamma\left(\lambda+\left[\frac{a-m-1}{2}\right]\right)}D_{m-a, 2N-q}.
\end{equation*}
The second gamma factor does not vanish by Remark \ref{rem-gamma-factor-qjN}. Moreover, by \eqref{def-constants-Djr}, we have
\begin{equation*}
D_{m-a, 2N-q} = \begin{pmatrix}
N-m+a\\
2N-q
\end{pmatrix}
\frac{\Gamma(q+1)}{\Gamma(N+m-a+1)}\frac{\Gamma(N+m+1)}{\Gamma(N+a+1)}(-1)^{2N-q}(2N-q)!,
\end{equation*}
which is clearly non-zero for any $\lambda = N+1-a-q \in \Lambda(N,a)$. Thus, $f_{m-a} \not\equiv 0$.

In all cases we have found a non-vanishing function $f_{\pm j}$. In particular, the tuple $(f_{-N}, \ldots, f_N)$ is non-zero.
\end{proof}

Now, we are ready to show Theorem \ref{Thm-solvingequations}.

\begin{proof}[Proof of Theorem \ref{Thm-solvingequations}]
By Propositions \ref{prop-phase2} and \ref{prop-phase3}, Remark \ref{rem-lambda_Gamma0-} and Lemmas \ref{lemma-case-m-N<=a<m} and \ref{lemma-simplification-fj+}, a non-zero solution $(f_{-N}, \ldots, f_N)$ of the system exists if and only if $\lambda \in \Lambda(N,a)$, and in that case it must be given by \eqref{expression-fj--normalized} for $f_{-j}$ $(j = 0, \ldots, N)$ and by \eqref{expression-fj+-normalized-final} for $f_j$ $(j = 1, \ldots, N)$ for some constant $\widetilde{q_N^-} \in \C$.
By a straightforward computation using the definition of the constants $A(d,r)$ and $B(d,r)$, we have 
\begin{align*}
i^{-j}\frac{\Gamma(N+j+1)}{\Gamma(N+1)}\frac{\Gamma\left(\lambda + \left[\frac{a-m+j+1}{2}\right]-\nu\right)}{\Gamma\left(\lambda + \left[\frac{a-m-1}{2}\right]\right)}D_{j,r} &= i^{-j}(-1)^{\nu-1}(-1)^r A(N-j, r),\\
i^{j}\frac{\Gamma(N+j+1)}{\Gamma(N+1)}\frac{\Gamma\left(\lambda + \left[\frac{a-m+j-1}{2}\right]\right)}{\Gamma\left(\lambda + \left[\frac{a-m-1}{2}\right]\right)}C_{j,r}^- &= i^{-j}(-1)^r B(N+j, r).
\end{align*}
Therefore, the tuple $(f_{-N}, \ldots, f_N)$ can be rewritten as follows:
\begin{align*}
f_j(t) &= i^{-j}(-1)^{\nu-1}\widetilde{q_N^-}\sum_{r=0}^{N-j}(-1)^r A(N-j, r) \Geg_{a-m+j+2(1-N-\nu+r)}^{\lambda+N-1-r}(it) & (j= 1, \ldots, N),\\
f_{-j}(t) &= i^{-j}\widetilde{q_N^-}\sum_{r=0}^{N-j}(-1)^rB(N+j,r)\Geg_{a-m+j-2N+2r}^{\lambda+N-1-r}(it) & (j= 0, 1, \ldots, N).
\end{align*}
In other words, we have
\begin{equation*}
(f_{-N}, \ldots, f_N) = \widetilde{q_N^-}(g_{m+N}, \ldots, g_{m-N})
\end{equation*}
where the right-hand side is given by \eqref{solution-all}. Moreover, we know that this tuple is non-zero thanks to Proposition \ref{prop-solution-nonzero}. The result is now proved.
\end{proof}

\section{The case $m < -N$}\label{section-case_m_lessthan_-N}
In this section we recall a duality theorem from \cite{Per-Val24} that shows the existence of an isomorphism between the F-systems $\Sol(\n_+; \sigma^{2N+1}_\lambda, \tau_{m, \nu})$ and $\Sol(\n_+; \sigma^{2N+1}_\lambda, \tau_{-m, \nu})$ when $|m| \geq N$. This allows us to deduce the case $m < -N$ from the case $m > N$ in Theorems \ref{thm_main_class} and \ref{thm_main_const}.

Given $m \in \Z$ with $m \leq -N$, we set $n := -m \geq N$. For $a \in \N$ with $a \geq n-N$, we define the following map:
\begin{align*}
\psi^\pm: \bigoplus_{k=n-N}^{n+N}\Pol_{a-k}[t]_\text{even} & \longrightarrow \Hom_{SO(2)}\left(V^{2N+1}, \C_{\pm n}\otimes \Pol^a(\n_+)\right)\\
(g_k)_{k=n-N}^{n+N} & \longmapsto \psi^\pm((g_k)_{k=n-N}^{n+N})(\zeta) :=  \sum_{k=n-N}^{n+N} (T_{a-k}g_k)(\zeta)h_k^\pm.
\end{align*}
Here, $T_b$ denotes the map \eqref{def-T_b} and $h_k^\pm$ are the generators \eqref{generators-Hom(V2N+1,CmHk)}. As before, we are identifying $\Pol(\n_+) \simeq \Pol(\zeta_1, \zeta_2, \zeta_3)$ by using the basis $\{N_1^+, N_2^+, N_3^+\}$ of $\n_+$.
Now, we define the following involution:
\begin{align*}
\Phi: \Hom_\C(V^{2N+1}, \Pol(\n_+))& \longrightarrow \Hom_\C(V^{2N+1}, \Pol(\n_+))\\
\varphi(\zeta) = (\varphi_d(\zeta))_{d=0}^{2N} &\longmapsto \Phi(\varphi)(\zeta):= \antidiag(1, -1, 1, -1, \ldots, 1)\varphi(\zeta_1,-\zeta_2, \zeta_3)\\
& \hspace{2.25cm} = ((-1)^d\varphi_{2N-d}(\zeta_1, -\zeta_2, \zeta_3))_{d=0}^{2N}.
\end{align*}
Observe that this involution relates $\psi^+$ and $\psi^-$, as if $\varphi = \psi^+((g_k)_{k = n-N}^{n+N})(\zeta)$, we have
\begin{equation*}
\Phi(\psi^+)(\zeta_1, \zeta_2, \zeta_3) = \psi^-((-1)^{k-n+N}(g_k)_{k=n-N}^{n+N})(\zeta_1, \zeta_2, \zeta_3).
\end{equation*}
Moreover, set
\begin{equation*}
\Phi_\text{Sol}^+ := \Phi\rvert_{\Sol(\n_+; \sigma_\lambda^{2N+1}, \tau_{n, \nu})},
\end{equation*}
then, the following holds
\begin{prop}[{\cite[Prop. 8.2]{Per-Val24}}]\label{prop-duality} For any $n \in \N$ with $n \geq N$, the map
\begin{equation*}
\Phi_\text{\normalfont{Sol}}^+: \Sol(\n_+; \sigma_\lambda^{2N+1}, \tau_{n, \nu}) \arrowsimeq \Sol(\n_+; \sigma_\lambda^{2N+1}, \tau_{-n, \nu})
\end{equation*}
is a well-defined isomorphism.
\end{prop}

With the result above it is clear that we can deduce the case $m < -N$ from the case $m > N$ in Theorems \ref{thm_main_class} and \ref{thm_main_const}. We write the detailed proofs below.

\begin{proof}[Proof of Theorem \ref{thm_main_class} for $m < -N$]
From Theorem \ref{F-method-thm} and Proposition \ref{prop-duality}, we have isomorphisms
\begin{equation*}
\Diff_{G^\prime}\left(C^\infty(S^3, \V_\lambda^{2N+1}), C^\infty(S^2, \L_{m,\nu})\right) \arrowsimeq \Sol(\n_+; \sigma_\lambda^{2N+1}, \tau_{m, \nu}) \arrowsimeq \Sol(\n_+; \sigma_\lambda^{2N+1}, \tau_{-m, \nu}).
\end{equation*}
The equivalences (i) $\Leftrightarrow$ (ii) $\Leftrightarrow$ (iii) follow now from Theorem \ref{thm-step2}.
\end{proof}

\begin{proof}[Proof of Theorem \ref{thm_main_const} for $m < -N$]
From Proposition \ref{prop-duality} and the definition of the map $\Phi^+_\text{Sol}$, the expression of the operator $\D_{\lambda, \nu}^{N, m}$ can be obtained by changing the coordinate $u_s$ by the coordinate $(-1)^s u_{2N-s}$ (for $s=0, \ldots, 2N$), and by switching $i$ by $i$ (or, equivalently, $z$ by $\overline{z}$) in the expression of the operator $\D_{\lambda, \nu}^{N, -m}$. By doing the change of indices given by $d = 2N-d$, we obtain precisely \eqref{Operator-}.
\end{proof}

\section{Localness Theorem}\label{section-localnessthm}
In this section we prove the localness theorem (Theorem \ref{Localness-thm}). The proof is based on the work \cite{KS15, KS18} and uses several known results about distributions and symmetry breaking operators.

First, recall from \cite[Fact 5.1]{KS18} that there is a natural bijection
\begin{equation}\label{distribution-bijection}
\begin{aligned}
\Hom_{SO_0(3,1)}\left(C^\infty(S^3, \V_\lambda^{2N+1}), C^\infty(S^2, \L_{m,\nu})\right) \arrowsimeq \left(\Dist(S^3, (\V_\lambda^{2N+1})^*) \otimes \C_{m, \nu}\right)^{ \Delta(P^\prime)}.
\end{aligned}
\end{equation}
Here, $\Dist(S^3, (\V_\lambda^{2N+1})^*)$ stands for the space of $(\V_\lambda^{2N+1})^*$-valued distributions, where $(\V_\lambda^{2N+1})^*$ denotes the tensor product bundle $(\V_\lambda^{2N+1})^\vee \otimes \Omega_{S^3}$ of the dual bundle of $\V_\lambda^{2N+1}$ and the density bundle on $S^3$, $\Omega_{S^3} = |\bigwedge^\text{top} T^\vee(S^3)|$. The right-hand side of \eqref{distribution-bijection} stands for the $P^\prime$-invariant elements under the diagonal action on the tensor product of the $G$-module $\Dist(S^3, (\V_\lambda^{2N+1})^*)$ and the $P^\prime$-module $\C_{m,\nu}$.

On the other hand, recall that the Bruhat decomposition of the group $G = SO_0(4,1)$ is given by $G = N_+wP \cup P$, where $w := \diag(1, 1, 1, 1, -1) \in G$. Then, since the manifold $S^3 = G/P$ is covered by the two open subsets $U_1 := N_+wP/P \simeq \R^3$ and $U_2 := N_-P/P \simeq \R^3$, distributions on $S^3$ are determined uniquely by the restriction to these two open sets:
\begin{equation*}
\Dist(G/P, (\V_\lambda^{2N+1})^*)\simeq (\Dist(G)\otimes (\V_\lambda^{2N+1})^*)^{\Delta(P)} \hookrightarrow \Dist(U_1, (\V_\lambda^{2N+1})^*\rvert_{U_1}) \oplus \Dist(U_2, (\V_\lambda^{2N+1})^*\rvert_{U_2}).
\end{equation*}
Denote by $\varphi_1$ and $\varphi_2$ the respective diffeomorphisms $\R^3 \arrowsimeq U_1$ and $\R^3 \arrowsimeq U_2$. Then, by trivializing the vector bundle $(\V_\lambda^{2N+1})^*$ as in the following diagram 
\begin{equation*}
\begin{tikzcd}[column sep = 0.5cm]
\R^3 \times (V^{2N+1})^\vee \arrow[r,"\sim"] \arrow[d] & (V_\lambda^{2N+1})^*\rvert_{U_1} \arrow[r, phantom, "\subset"] \arrow[d] & (V_\lambda^{2N+1})^* \arrow[r, phantom, "\supset"] \arrow[d] & (V_\lambda^{2N+1})^*\rvert_{U_2} \arrow[d] & \R^3 \times (V^{2N+1})^\vee \arrow[l,"\sim"'] \arrow[d]\\
\R^3 \arrow[r, "\sim", "\varphi_1"'] & U_1 \arrow[r, phantom, "\subset"] & S^3  \arrow[r, phantom, "\supset"] & U_2 & \R^3 \arrow[l, "\sim"', "\varphi_2"]
\end{tikzcd}
\end{equation*}
the injection above amounts to
\begin{equation}\label{distribution-injection}
\begin{aligned}
\Dist(G/P, (\V_\lambda^{2N+1})^*) &\hookrightarrow \left(\Dist(\R^3) \otimes (V^{2N+1})^\vee\right) \oplus \left(\Dist(\R^3) \otimes (V^{2N+1})^\vee\right)\\
f &\mapsto (F_1, F_2)
\end{aligned}
\end{equation}
where $F_1(a) = f(\varphi_1(a))$ and $F_2(b) = f(\varphi_2(b))$.

In the following lemma we give a description of the $P$-action on the first projection space. Note that since the open set $U_1 = N_+wP/P = PwP/P$ is $P$-invariant, we can define the action geometrically. Let $n_+: \R^3 \rightarrow N_+$ be the natural isomorphism given by $n_+(c) = \exp(\sum_{j=1}^3 c_jN_j^+)$ (see \eqref{elements}, cf. \eqref{open-Bruhat-cell}).

\begin{lemma}[{\cite[Lem. 5.12]{KS18}}]\label{lemma_P_action_dist} Let $P = MAN_+$ act on $\Dist(\R^3) \otimes (V^{2N+1})^\vee$ by:
\begin{align*}
\left(\pi
\begin{pmatrix}
1 & &\\
& g &\\
& & 1
\end{pmatrix}
F_1
\right)(a) &= (\sigma^{2N+1})^\vee(g)F_1(g^{-1}a), \quad \text{for } g \in SO(3),\\
\left(\pi(e^{tH_0})F_1\right)(a) &= e^{(\lambda-3)t}F_1(e^{-t}a), \quad \text{for } t \in \R,\\
\left(\pi(n_+(c))F_1\right)(a) & = F_1(a-c), \quad \text{for } c \in \R^3.
\end{align*}
Then, the first projection $f \mapsto F_1$ in \eqref{distribution-injection} is a $P$-homomorphism.
\end{lemma}

One can describe the $P$-action on the second projection space as well. However, in this case, since $U_2 = N_-P/P$ is not $N_+$-invariant, the action of $N_+$ shall be substituted by its infinitesimal action on the Lie algebra $\n_+$. We omit a detailed description of this action, but it can be verified in \cite[Lem. 5.13]{KS18}.

By using Lemma \ref{lemma_P_action_dist}, we can give a description of the $P^\prime$-action on $\Dist(\R^3)\otimes \Hom_\C(V_\lambda^{2N+1}, \C_{m, \nu})$ induced by the diagonal $P^\prime$-action on the right-hand side of \eqref{distribution-bijection} through the first projection in \eqref{distribution-injection}. In fact, define the space
\begin{equation*}
\left(\Dist(\R^3)\otimes \Hom_\C(V_\lambda^{2N+1}, \C_{m, \nu})\right)^{\Delta(P^\prime)}
\end{equation*}
as the space of $\Hom_\C(V_\lambda^{2N+1}, \C_{m, \nu})$-valued distributions $\mathcal{T}_1$ on $\R^3 \simeq \R^2 \oplus \R$ such that the following conditions are satisfied.
\begin{align}
\label{SO(2)-invariance}
\C_m(g^\prime) \circ \mathcal{T}_1((g^\prime)^{-1} y, y_3) \circ (\sigma^{2N+1})^{-1}(g^\prime) &=  \mathcal{T}_1(y, y_3), &&\quad \text{ for all } g^\prime \in SO(2),\\
\mathcal{T}_1(e^ty, e^ty_3) &= e^{(\lambda+\nu-3)t}\mathcal{T}_1(y,y_3), &&\quad \text{ for all } t\in \R,\\
\label{N_+'-invariance}
\mathcal{T}_1(y-z, y_3) & = \mathcal{T}_1(y, y_3), &&\quad \text{ for all } z\in \R^2.
\end{align}
Then, we have the following result
\begin{prop}[{\cite[Prop. 5.15]{KS18}}]\label{prop_dist_first}
The map 
\begin{align*}
\Hom_{SO_0(3,1)}\left(C^\infty(S^3, \V_\lambda^{2N+1}), C^\infty(S^2, \L_{m,\nu})\right) &\longrightarrow \left(\Dist(\R^3)\otimes \Hom_\C(V_\lambda^{2N+1}, \C_{m, \nu})\right)^{\Delta(P^\prime)}\\
\mathbb{T} &\longmapsto \mathcal{T}_1,
\end{align*}
induced from \eqref{distribution-bijection} and the first projection of \eqref{distribution-injection} satisfies the following key property:
\center{$\mathcal{T}_1 = 0$ if and only if $\mathbb{T}$ is a differential operator.}
\end{prop}

Now we are ready to prove Theorem \ref{Localness-thm}.

\begin{proof}[Proof of Theorem \ref{Localness-thm}]
This result can be proved following a similar argument to that of \cite[Thm. 3.6]{KS18}.

Recall from \cite[Lem. 4.1]{Per-Val24} that the representation $V^{2N+1}$ can be decomposed as a $SO(2)$-module as follows:
\begin{equation*}
V^{2N+1} \simeq \bigoplus_{\ell = -N}^N \C_\ell.
\end{equation*}
Hence, if $|m| > N$, then necessarily
\begin{equation}\label{index_zero}
\Hom_{SO(2)}(V^{2N+1}\rvert_{SO(2)}, \C_m) = \{0\}.
\end{equation}

Now, take a symmetry breaking operator 
\begin{equation*}
\mathbb{T} \in \Hom_{SO_0(3,1)}\left(C^\infty(S^3, \V_\lambda^{2N+1}), C^\infty(S^2, \L_{m,\nu})\right).
\end{equation*}
Let $p_3: \R^3 \rightarrow \R$ be the projection on the third coordinate, and let $p_3^*: \Dist(\R) \rightarrow \Dist(\R^3)$ be the pullback of distributions. Then, by the $N_+^\prime$-invariance
\eqref{N_+'-invariance}, $\mathcal{T}_1$ depends only on the last coordinate; that is, $\mathcal{T}_1$ is of the form $p_3^*f$ for some $f \in \Dist(\R) \otimes \Hom_\C(V_\lambda^{2N+1}, \C_{m, \nu})$. Moreover, by the $SO(2)$-invariance \eqref{SO(2)-invariance}, we have
\begin{equation*}
f \in \Dist(\R) \otimes \Hom_{SO(2)}(V_\lambda^{2N+1}\rvert_{SO(2)}, \C_{m, \nu}).
\end{equation*}
However, by \eqref{index_zero} we have necessarily $f = 0$, which implies $\mathcal{T}_1 = 0$. Now, by Proposition \ref{prop_dist_first} we deduce that $\mathbb{T}$ is a differential operator.
\end{proof}

\section{Sporadicity of SBOs for $|m| > N$}\label{section-sporadicity}
In this section we prove Theorem \ref{thm-sporadicity}, that assures that any non-zero SBO in \eqref{SBO-space} for $|m| > N$ is sporadic (see Definition \ref{def-sporadic} below).

\begin{defn}\label{def-sporadic} Given $V^{2N+1} \in \widehat{SO(3)}$ and $\C_m \in \widehat{SO(2)}$, let $\P \subset \C^2$ be the set of parameters $(\lambda, \nu)$ such that the space of SBOs \eqref{SBO-space} is non-zero. For a fixed $(N, m)$ and $(\lambda, \nu) \in \P$, we say that an SBO $\T_{\lambda, \nu}^{N, m}$ is \textbf{sporadic} with respect to the pair $(SO_0(4,1), SO_0(3,1))$ if $(\lambda, \nu)$ is an isolated point of $\P$. In other words, if there is an open neighbourhood $U$ of $(\lambda, \nu)$ on $\C^2$ such that $U\cap\P = \{(\lambda, \nu)\}$.
\end{defn}

Although Definition~\ref{def-sporadic} above is stated for the pair $(SO_0(4,1), SO_0(3,1))$ and the representations $V^{2N+1}$ and $\C_m$, the notion of sporadicity can be formulated also for other pairs $(G, G^\prime)$ and representations $V, W$ in the obvious way.

\begin{rem}
\begin{enumerate}[topsep = 0pt, left = 0pt, label=\normalfont{(\arabic*)}]
\item By Theorems \ref{thm_main_class} and \ref{Localness-thm}, we deduce that, for $(N,m) \in \N\times \Z$ with $|m| > N$, the set $\P$ is given as follows:
\begin{equation*}
\P = \{(\lambda, \nu) \in \C^2: 
\lambda \in \Z_{\leq 1-|m|}, \nu \in [1-N, N+1]\cap\Z\}.
\end{equation*}

\item Note that the notion of sporadicity given in \cite{KS18} by Kobayashi--Speh is slightly different to the one defined above. However, it is straightforward to check that Definition \ref{def-sporadic} implies the one given in \cite{KS18}. In particular, sporadic SBOs cannot be obtained by residues of meromorphic families of SBOs, contrarily to the \emph{regular} SBOs (cf. \cite{KS18}).
\end{enumerate}
\end{rem}

Now we state the main result of this section.

\begin{thm}\label{thm-sporadicity} Any non-zero SBO in \eqref{SBO-space} satisfying $|m| > N$ is sporadic.
\end{thm}

\begin{proof}
The proof is straightforward from Remark 7.2(1), that shows in particular that $\P$ is discrete.
\end{proof}

\section{Appendix: Hypergeometric Series and Gegenbauer Polynomials}\label{section-appendix}
In this section we define and give some properties of hypergeometric series ${}_pF_q$ and renormalized Gegenbauer polynomials $\Geg_\ell^\mu$.

\subsection{Hypergeometric Series}
Here, we recall the definition and some properties of the generalized hypergeometric function ${}_pF_q$. In particular, we give several useful identities of the functions ${}_2F_1$ and ${}_3F_2$.

Given $p, q \in \N$ and $a_1, \ldots, a_p \in \C$, $b_1, \ldots, b_q \in \C\setminus \Z_{\leq 0}$, the hypergeometric series is defined as
\begin{equation}\label{def-hypergeometric}
{}_pF_q\left(
\begin{matrix}
a_1, \ldots, a_p\\
b_1, \ldots, b_q 
\end{matrix}
; z\right) := \sum_{n=0}^{\infty} \frac{(a_1)_n \cdots (a_p)_n}{(b_1)_n \cdots (b_q)_n}\frac{z^n}{n!},
\end{equation}
where $(x)_n$ denotes the Pochhammer rising factorial:
\begin{equation}\label{Pochhamer-symbol}
(x)_n := \begin{cases}
1, & \text{ if } n = 0,\\
x (x+1) \cdots (x+n-1), & \text{ if } n \geq 1.
\end{cases}
\end{equation}
The radius of convergence of \eqref{def-hypergeometric} is known. For instance, if some $a_j$ is a non-positive integer, the series has a finite number of terms. If $a_j$ are not non-positive integers, it converges absolutely for any $z \in \C$ if $p < q +1 $, for $|z| < 1$ if $p = q+1$, and it diverges if $p > q+1$ (cf. \cite[Thm. 2.1.1]{AAR99}). When this series has a non-zero radius of convergence, it defines an analytic function which is often called the \emph{generalized} hypergeometric function. When $p = q+ 1 = 2$, this function is known as the \emph{ordinary} (or \emph{Gaussian}) hypergeometric function.

When $p = q+1$ and $|z| = 1$ there is not a general criterion of convergence, but it is known that ${}_{q+1}F_q$ converges absolutely for $|z|=1$ if $\Re(\sum_{j=1}^q b_j - \sum_{j=1}^{q+1} a_j) >0$, conditionally for $z = e^{it} \neq 1$ if $-1 < \Re(\sum_{j=1}^q b_j - \sum_{j=1}^{q+1}a_j) \leq 0$ and diverges if $ \Re(\sum_{j=1}^q b_j - \sum_{j=1}^{q+1}a_j) \leq -1$. When $z = 1$, several formulas giving the precise value of the series are known, as is the Gauss formula for ${}_2F_1$:

\begin{lemma}[{\normalfont\cite[Eq. (48)]{Gau12}}] For $\Re(c-b-a)> 0$, the following holds.
\begin{equation*}
{}_2F_1\left(\begin{matrix}
a, b\\
c
\end{matrix}\; ; 1 \right) = \frac{\Gamma(c)\Gamma(c-a-b)}{\Gamma(c-a)\Gamma(c-b)}.
\end{equation*}
\end{lemma}

A straightforward corollary is the following one, due to Chu and Vandermonde.

\begin{lemma}[{\normalfont\cite[Cor. 2.2.3]{AAR99}}] For any $a, b \in \C$ and any $n \in \N$, we have
\begin{equation}\label{Chu-Vandermonde}
{}_2F_1\left(\begin{matrix}
-n, a\\
c
\end{matrix}\; ; 1 \right) = \frac{(c-a)_n}{(c)_n}.
\end{equation}
\end{lemma}
For $(p, q) = (2, 3)$ we have the Pfaff--Saalschütz identity:
\begin{lemma}[{\normalfont\cite[Thm. 2.2.6]{AAR99}}] The following identity holds for any $a, b, c \in \C$ and any $n \in \N$.
\begin{equation}\label{pfaff-saalschutz}
{}_3F_2\left(
\begin{matrix}
-n, a, b\\
c, 1+a+b-c-n
\end{matrix}\; ; 1\right) = \frac{(c-a)_n(c-b)_n}{(c)_n(c-a-b)_n}.
\end{equation}
\end{lemma}

We end this subsection by stating two useful transformation identities for ${}_3F_2$. The first one can be found in \cite{AAR99}, and the second one is known as Kummer's identity.

\begin{lemma}[{\normalfont\cite[Eq. (2.4.12)]{AAR99}}] The following identity holds for any $a, b, d, e \in \C$, and any $n \in \N_+$.
\begin{equation}\label{aar-identity}
{}_3F_2\left(
\begin{matrix}
a, b, e+n-1\\
d, e
\end{matrix}\; ; 1\right) = \frac{\Gamma(d)\Gamma(d-a-b)}{\Gamma(d-a)\Gamma(d-b)}{}_3F_2\left(
\begin{matrix}
a, d-b, 1-n\\
a+b-d+1, e
\end{matrix}\; ; 1\right).
\end{equation}
\end{lemma}

\begin{lemma}[{\normalfont\cite[Cor. 3.3.5]{AAR99}}] The following identity holds for any $a, b, c, d, e \in \C$ when both sides converge.
\begin{equation}\label{kummer}
{}_3F_2\left(
\begin{matrix}
a, b, c\\
d, e
\end{matrix}\; ; 1\right) = \frac{\Gamma(e)\Gamma(d+e-a-b-c)}{\Gamma(e-a)\Gamma(d+e-b-c)}{}_3F_2\left(
\begin{matrix}
a, d-b, d-c\\
d, d+e-b-c
\end{matrix}\; ; 1\right).
\end{equation}
\end{lemma}

\subsection{Renormalized Gegenbauer polynomials}
For any $\mu \in \C$ and any $\ell \in \N$, the Gegenbauer polynomial (also known as ultraspherical polynomial), named after the Austrian mathematician Leopold H. Gegenbauer (\cite{Geg74}) is defined as follows (cf. \cite[Sec. 6.4]{AAR99}, \cite[Sec. 3.15.1]{EMOT53}):
\begin{equation*}
\begin{aligned}
C_\ell^\mu(z) & := \sum_{k = 0}^{[\frac{\ell}{2}]}(-1)^k \frac{\Gamma(\ell + \mu - k)}{\Gamma(\mu)k!(\ell -  2k)!}(2z)^{\ell - 2k}\\
& = \frac{\Gamma(\ell + 2\mu)}{\Gamma(2\mu)\Gamma(\ell + 1)}{}_2F_1\left(
\begin{matrix}
-\ell, \ell +2\mu \\
\mu+ \frac{1}{2}
\end{matrix}
\; ; \frac{1-z}{2} \right).
\end{aligned}
\end{equation*}

This polynomial satisfies the Gegenbauer differential equation
\begin{equation*}
G_\ell^\mu f(z) = 0,
\end{equation*}
where $G_\ell^\mu$ is the Gegenbauer differential operator:
\begin{equation*}
G_\ell^\mu := (1-z^2)\frac{d^2}{dz^2} - (2\mu +1)z\frac{d}{dz} + \ell(\ell +2\mu).
\end{equation*}
Observe that $C_\ell^\mu$ vanishes when $\mu = 0, -1, -2, \cdots, - [\frac{\ell-1}{2}]$. In order to avoid this, we renormalize the Gegenbauer polynomial as follows:
\begin{equation}\label{Gegenbauer-polynomial(renormalized)}
\widetilde{C}^\mu_\ell(z) := \frac{\Gamma(\mu)}{\Gamma\left(\mu + [\frac{\ell + 1}{2}]\right)}C^\mu_\ell(z) =\sum_{k=0}^{[\frac{\ell}{2}]}\frac{\Gamma(\mu + \ell - k)}{\Gamma\left(\mu+ [\frac{\ell + 1}{2}]\right)}\frac{(-1)^k }{k! (\ell - 2k)!}(2z)^{\ell - 2k}.
\end{equation}
By doing this, we obtain that $\Geg_\ell^\mu(z)$ is a non-zero polynomial for any $\ell \in \N$ and any $\mu \in \C$. This renormalization of $C_\ell^\mu$ is the same as the one given in \cite{KKP16, KP16b, Per-Val23, Per-Val24}. We write below the first six Gegenbauer polynomials.

\begin{itemize}
\item $\widetilde{C}^\mu_0(z) = 1$.
\item $\widetilde{C}^\mu_1(z) = 2z$.
\item $\widetilde{C}^\mu_2(z) = 2(\mu + 1) z^2 -1$.
\item $\widetilde{C}^\mu_3(z) = \frac{4}{3}(\mu + 2) z^3 -2z$.
\item $\widetilde{C}^\mu_4(z) = \frac{2}{3}(\mu +2)(\mu+3)z^4 -
2(\mu + 2)z^2 + \frac{1}{2}$.
\item $\widetilde{C}^\mu_5(z) = \frac{4}{15}(\mu+3)(\mu+4)z^5 - \frac{4}{3}(\mu+3)z^3 + z$.
\end{itemize}

Observe that from the definition, the degree of each term of $\Geg_\ell^\mu(z)$ has the same parity as $\ell$, and the number of terms is at most $\left[\frac{\ell}{2}\right]+1$. Depending on the value of $\mu$, some terms may vanish, but the term of degree $\ell - 2\left[\frac{\ell}{2}\right] \in \{0,1\}$ (i.e., the first term), is non-zero for any $\mu \in \C$.
By a direct computation one can show that this term is given by:
\begin{equation*}
\begin{aligned}
\frac{(-1)^{\left[\frac{\ell}{2}\right]}}{\left[\frac{\ell}{2}\right]!} (2z)^{\ell - 2\left[\frac{\ell}{2}\right]}.
\end{aligned}
\end{equation*}
Note that if $\mu + \left[\frac{\ell+1}{2}\right] = 0$, then $\Geg_{\ell}^\mu(z)$ is constant for any $\ell \geq 0$. For $\ell < 0$ we just define the renormalized Gegenbauer polynomial as $\Geg_\ell^\mu \equiv 0$. 

The next result shows that the solutions of the Gegenbauer differential equation on $\Pol_\ell[z]_\text{{\normalfont{even}}}$ are precisely the renormalized Gegenbauer polynomials, where
\begin{equation*}
\Pol_\ell[z]_\text{even} = \spanned_\C\{z^{\ell - 2b}: b = 0, 1, \ldots, \left[\frac{\ell}{2}\right]\}.
\end{equation*}
\begin{thm}[{\cite[Thm 11.4]{KP16b}}]\label{thm-Gegenbauer-solutions} For any $\mu \in \C$ and any $\ell \in \N$, the following holds.
\begin{equation*}
\{f(z) \in \Pol_\ell[z]_\text{{\normalfont{even}}} : G_\ell^\mu f(z) = 0\} = \C \widetilde{C}_\ell^\mu(z).
\end{equation*}
\end{thm}
In $\C[z] \setminus \Pol_\ell[z]_\text{{\normalfont{even}}}$, the Gegenbauer differential equation may have other solutions depending on the parameters $\mu, \ell$. For a concrete statement of this fact see \cite[Thm. 11.4]{KP16b}.

For any $\mu \in \C$ and any $\ell \in \N$, the \textit{imaginary} Gegenbauer differential operator $S_\ell^\mu$ is defined as follows (cf. \cite[(4.7)]{KKP16}):
\begin{equation}\label{Gegen-imaginary}
S_\ell^\mu = - \left((1 + t^2) \frac{d^2}{dt^2} + (1 + 2\mu)t\frac{d}{dt} - \ell(\ell + 2\mu)\right).
\end{equation}
This operator and $G_\ell^\mu$ are in the following natural relation.

\begin{lemma}[{\cite[Lem. 14.2]{KKP16}}]\label{lemma-relationS-G} Let $f(z)$ be a polynomial in the variable $z$ and define $g(t) = f(z)$, where $z = it$. Then, the following identity holds:
\begin{equation*}
\left(S_\ell^\mu g\right)(t) = \left(G_\ell^\mu f\right)(z).
\end{equation*}
\end{lemma}

Next, we give a condition on the parameter $\mu \in \C$ for which the \lq\lq higher terms\rq\rq{ }of $\Geg_\ell^\mu$ vanish. At the end, we recall a \lq\lq degree decay\rq\rq{ }property that happens when special parameters are considered.

For any $\ell \in \N$ and any $k = 0, 1, \ldots, \left[\frac{\ell}{2}\right]$ we set
\begin{equation}\label{def-Mlk-set}
M_\ell^k := \left\{-\left[\frac{\ell+1}{2}\right]-d : d= 0, 1, \ldots, \left[\frac{\ell}{2}\right]-k-1\right\}.
\end{equation}
If $\ell = 0, 1$ or if $k = \left[\frac{\ell}{2}\right]$ we set $M_\ell^k = \emptyset$. In any other case we have that $M_\ell^k \neq \emptyset$ from the expression above.

\begin{lemma}\label{lemma-Gegenbauer-Mlk-condition}
Let $\mu \in \C$ and $\ell \in \N$. Then, the term $z^{\ell-2k} \enspace (0 \leq k < \left[\frac{\ell}{2}\right])$ in $\Geg_\ell^\mu(z)$ vanishes if and only if $\mu \in M_\ell^k$. In particular, if $\mu \notin \Z$ then $\deg(\Geg_\ell^\mu(z)) = \ell$.
\end{lemma}
\begin{proof}
The proof is straightforward from the definition of $\Geg_\ell^\mu$. Note that the term of degree $\ell -2\left[\frac{\ell}{2}\right] \in \{0, 1\}$ is always non-zero, so the result is trivially true also for $k = \left[\frac{\ell}{2}\right]$ since we defined $M_\ell^{\left[\frac{\ell}{2}\right]} = \emptyset$.

Suppose $0 \leq k < \left[\frac{\ell}{2}\right]$. From \eqref{Gegenbauer-polynomial(renormalized)} it is clear that the term $z^{\ell-2k}$ vanishes if and only if
\begin{equation*}
\begin{aligned}
\frac{\Gamma\left(\mu + \ell -k \right)}{\Gamma\left(\mu + \left[\frac{\ell+1}{2}\right]\right)} &= \left(\mu + \left[\frac{\ell+1}{2}\right]\right) \cdot \left(\mu + \left[\frac{\ell+1}{2}\right] + 1\right) \cdots \left(\mu + \left[\frac{\ell+1}{2}\right] + \left[\frac{\ell}{2}\right] -k -1\right)\\
&= 0,
\end{aligned}
\end{equation*}
that is, if and only if $\mu \in M_\ell^k$. The lemma is now proved.
\end{proof} 

\begin{rem}\label{rem-Mlk}
From the definition of $M_\ell^k$ note that for any $k = 0, 1, \ldots, \left[\frac{\ell}{2}\right]-1$ the following hold:
\begin{itemize}
\item[(1)] $M_\ell^{k+1} \subsetneq M_\ell^{k}$. 
\item[(2)] $M_\ell^k \setminus M_\ell^{k+1} = \{\ell + k + 1\}$.
\end{itemize}
In particular, by Lemma \ref{lemma-Gegenbauer-Mlk-condition}, if $\mu \in M_\ell^k$, all the terms $z^{\ell-2k}, z^{\ell-2(k-1)}, \ldots, z^{\ell}$ in $\Geg_\ell^\mu(z)$ vanish. 
\end{rem}

\begin{lemma}[{\cite[Lem. 4.12]{KOSS15}}]\label{lemma-KOSS} For $\ell,b \in \N$ such that $2b \geq \ell$ the following holds:
\begin{equation*}
C_\ell^{-b}(z) = C_{2b-\ell}^{-b}(z),
\end{equation*}
or equivalently
\begin{equation}\label{identity-KOSS}
\frac{\Gamma\left(-b + \left[\frac{\ell+1}{2}\right]\right)}{\Gamma\left(-b + \left[\frac{2b-\ell+1}{2}\right]\right)}\Geg_{\ell}^{-b}(z) = \Geg_{2b-\ell}^{-b}(z).
\end{equation}
\end{lemma}

\subsection{Algebraic properties of renormalized Gegenbauer polynomials and imaginary Gegenbauer operators}
In this section we collect some useful properties of imaginary Gegenbauer operators and renormalized Gegenbauer polynomials. We start by showing some algebraic identities and follow by recalling one recurrence relation among Gegenbauer polynomials.

For $\mu \in \C$ and $\ell \in \N$, we define the following gamma factor:
\begin{equation}\label{gamma-def}
\gamma(\mu, \ell) := \displaystyle{\frac{\Gamma(\mu + \left[\frac{\ell+2}{2}\right])}{\Gamma(\mu + \left[\frac{\ell + 1}{2}\right])} = \begin{cases}
1,&\text{if } \ell \text{ is odd},\\
\mu + \frac{\ell}{2},&\text{if } \ell \text{ is even}.
\end{cases}}
\end{equation}
Note that this definition can be extended to $\ell \in \Z$. For $\ell < 0$ it suffices to take $k \in \N$ such that $\ell + 2k \in \N$ and define $\gamma(\mu, \ell) := \gamma(\mu-k,\ell + 2k)$. 

Another remarkable property of $\gamma(\mu, \ell)$ is that it satisfies the following identity, which is straightforward from the definition:
\begin{equation}\label{gamma-product-property}
\gamma(\mu, \ell)\gamma(\mu, \ell+1) = \mu + \left[\frac{\ell +1}{2}\right].
\end{equation}

During the rest of the section, we denote by $\vartheta_t$ the one-dimensional Euler operator $t\frac{d}{dt}$.

\begin{lemma}[{\cite[Lem. 14.4]{KKP16}}] For any $\mu \in \C$ and any $\ell \in \N$, the following identities hold:
\begin{align}
\label{derivative-Gegenbauer-1}
\frac{d}{dt}\Geg_\ell^\mu(it) &= 2i\gamma(\mu, \ell)\Geg_{\ell -1}^{\mu + 1}(it),\\[4pt]
\label{derivative-Gegenbauer-2}
\left(\vartheta_t-\ell\right)\widetilde{C}_\ell^\mu(it) &= \Geg_{\ell -2}^{\mu + 1}(it).
\end{align}
\end{lemma}

\begin{lemma}\label{lemma-KKP-identities}
For any $\mu \in \C$ and any $\ell, d \in \N$ the following recursive relation holds:
\begin{align}
\label{TTR-generalized}
\frac{\Gamma\left(\mu+ \left[\frac{\ell+1}{2}\right]\right)}{\Gamma\left(\mu - d + \left[\frac{\ell+1}{2}\right]\right)}
\Geg_\ell^\mu(it) = \sum_{s=0}^{d}
\begin{pmatrix}
d\\
s
\end{pmatrix}
\frac{\Gamma\left(\mu+\ell+s-d\right)}{\Gamma\left(\mu+\ell-d\right)}\Geg_{\ell+2(s-d)}^{\mu-s}(it).
\end{align}
\end{lemma}

\begin{proof}
The identity above can be proved easily by induction on $d\in \N$. If $d = 0$, the identity is trivial since both sides amount to $\Geg_{\ell}^\mu(it)$. If $d = 1$, then \eqref{TTR-generalized} amounts to
\begin{align}
\label{KKP-1}
(\mu + \ell-1)\Geg_\ell^{\mu-1}(it) + \Geg_{\ell-2}^{\mu}(it) = \left(\mu + \left[\frac{\ell-1}{2}\right]\right)\Geg_\ell^{\mu}(it),
\end{align}
which coincides with the three-term relation \cite[Eq. (14.15)]{KKP16}.
Now suppose that \eqref{TTR-generalized} is true for some $d\in \N$. Then, for $d+1$, we have
\begin{align*}
\frac{\Gamma\left(\mu+ \left[\frac{\ell+1}{2}\right]\right)}{\Gamma\left(\mu - d -1 + \left[\frac{\ell+1}{2}\right]\right)}
\Geg_\ell^\mu(it) &= \left(\mu - d -1 + \left[\frac{\ell+1}{2}\right]\right)\frac{\Gamma\left(\mu+ \left[\frac{\ell+1}{2}\right]\right)}{\Gamma\left(\mu - d + \left[\frac{\ell+1}{2}\right]\right)}\Geg_\ell^\mu(it)\\
&= \left(\mu - d -1 + \left[\frac{\ell+1}{2}\right]\right)\sum_{s=0}^{d}
\begin{pmatrix}
d\\
s
\end{pmatrix}
\frac{\Gamma\left(\mu+\ell+s-d\right)}{\Gamma\left(\mu+\ell-d\right)}\Geg_{\ell+2(s-d)}^{\mu-s}(it)\\
&= 
\sum_{s=0}^{d}
\begin{pmatrix}
d\\
s
\end{pmatrix}
\frac{\Gamma\left(\mu+\ell+s-d\right)}{\Gamma\left(\mu+\ell-d\right)}(\mu+\ell+s-2d-1)\Geg_{\ell+2(s-d)}^{\mu-s-1}(it)\\
& \quad +\sum_{s=0}^{d}
\begin{pmatrix}
d\\
s
\end{pmatrix}
\frac{\Gamma\left(\mu+\ell+s-d\right)}{\Gamma\left(\mu+\ell-d\right)}\Geg_{\ell+2(s-d-1)}^{\mu-s}(it),
\end{align*}
where in the second equality we used the induction hypothesis,  and in the third we applied \eqref{KKP-1}. Then, by straightforward computations, the right-hand side amounts to
\begin{align*}
\Geg_{\ell-2(d+1)}^\mu(it) + 
\sum_{s=1}^{d}\Bigg(
\begin{pmatrix}
d\\
s-1
\end{pmatrix}
\frac{\Gamma\left(\mu+\ell+s-1-d\right)}{\Gamma\left(\mu+\ell-d\right)}(\mu+\ell+s-2d-2) +\\
\begin{pmatrix}
d\\
s
\end{pmatrix}
\frac{\Gamma\left(\mu+\ell+s-d\right)}{\Gamma\left(\mu+\ell-d\right)}
\Bigg)\Geg_{\ell+2(s-d-1)}^{\mu-s}(it) + 
\frac{\Gamma\left(\mu+\ell\right)}{\Gamma\left(\mu+\ell-d\right)}(\mu+\ell-d-1)\Geg_{\ell}^{\mu-d-1}(it)\\
= \sum_{s=0}^{d+1}
\begin{pmatrix}
d+1\\
s
\end{pmatrix}
\frac{\Gamma\left(\mu+\ell+s-d-1\right)}{\Gamma\left(\mu+\ell-d-1\right)}\Geg_{\ell+2(s-d-1)}^{\mu-s}(it).
\end{align*}
Hence, \eqref{TTR-generalized} holds also for $d+1$.
\end{proof}

In \eqref{Gegen-imaginary} we have defined the imaginary Gegenbauer operator $S_\ell^\mu$. In the following lemma we give some algebraic properties of this operator.

\begin{lemma} For any $\mu \in \C$ and any $\ell, d \in \N$ the following identities hold:
\begin{align}
\label{Slmu-identity-1}
S_\ell^{\mu} - S_\ell^{\mu+d} &=  2d(\vartheta_t-\ell),\\
\label{Slmu-identity-2}
S_\ell^\mu - S_{\ell -2d}^{\mu + d} &= 2d(\vartheta_t + 2\mu + \ell).
\end{align}
\end{lemma}
\begin{proof}
The proof is straightforward from the definition of $S_\ell^\mu$.
\end{proof}
We remark that identities \eqref{Slmu-identity-1} and \eqref{Slmu-identity-2} above are a simple generalization of 
\cite[Eq. (14.9)]{KKP16} and \cite[Eq. (14.12)]{KKP16} respectively.

\section*{Acknowledgements}
The author is grateful to Toshiyuki Kobayashi and Toshihisa Kubo for their insightful comments on this work. He also thanks Temma Aoyama for fruitful discussions regarding the use of hypergeometric series in the proof of Proposition 4.13.
Finally, he extends his appreciation to the anonymous referee for a careful review and helpful comments on the manuscript.

The author was supported by Grant-in-Aid for JSPS International Research Fellows (JP24KF0075). He also acknowledge support of the Institut Henri Poincaré (UAR 839 CNRS-Sorbonne Université), and LabEx CARMIN (ANR-10-LABX-59-01).

\renewcommand{\refname}{References}

\begin{flushleft}
V. Pérez-Valdés, Graduate School of Mathematical Sciences, The University of Tokyo, 3-8-1 Komaba, Meguro-ku, Tokyo 153-8914, Japan.\\
Email address: \textbf{perez-valdes@g.ecc.u-tokyo.ac.jp}
\end{flushleft}
\end{document}